\documentclass[11pt,oneside,english]{amsart}
\usepackage[T1]{fontenc}
\usepackage[latin9]{inputenc}
\usepackage{textcomp}
\usepackage{amsthm}
\usepackage{amstext}
\usepackage{amssymb}
\usepackage{esint}

\makeatletter
\numberwithin{equation}{section}
\numberwithin{figure}{section}
\theoremstyle{plain}
\newtheorem{thm}{\protect\theoremname}[section]
  \theoremstyle{plain}
  \newtheorem*{conjecture*}{\protect\conjecturename}
  \theoremstyle{remark}
  \newtheorem{rem}[thm]{\protect\remarkname}
  \theoremstyle{plain}
  \newtheorem{lem}[thm]{\protect\lemmaname}
  \theoremstyle{plain}
  \newtheorem{prop}[thm]{\protect\propositionname}
  \theoremstyle{definition}
  \newtheorem{example}[thm]{\protect\examplename}
  \theoremstyle{plain}
  \newtheorem{cor}[thm]{\protect\corollaryname}

\def\makebbb#1{
    \expandafter\gdef\csname#1\endcsname{
        \ensuremath{\Bbb{#1}}}
}\makebbb{R}\makebbb{N}\makebbb{Z}\makebbb{C}\makebbb{H}\makebbb{E}\makebbb{H}\makebbb{P}\makebbb{B}\makebbb{Q}\makebbb{E}

\usepackage{babel}

\usepackage{babel}

\makeatother

\usepackage{babel}

\makeatother

\usepackage{babel}
  \providecommand{\conjecturename}{Conjecture}
  \providecommand{\corollaryname}{Corollary}
  \providecommand{\examplename}{Example}
  \providecommand{\lemmaname}{Lemma}
  \providecommand{\propositionname}{Proposition}
  \providecommand{\remarkname}{Remark}
\providecommand{\theoremname}{Theorem}

\begin{document}

\title{Real Monge-Ampère equations and Kähler-Ricci solitons on toric log
Fano varieties}

\author{Robert J. Berman, Bo Berndtsson}

\email{robertb@chalmers.se, bob@chalmers.se}

\curraddr{Mathematical Sciences - Chalmers University of Technology and University
of Gothenburg - SE-412 96 Gothenburg, Sweden }
\begin{abstract}
We show, using a direct variational approach, that the second boundary
value problem for the Monge-Ampère equation in $\R^{n}$ with exponential
non-linearity and target a convex body $P$ is solvable iff $0$ is
the barycenter of $P.$ Combined with some toric geometry this confirms,
in particular, the (generalized) Yau-Tian-Donaldson conjecture for
toric log Fano varieties $(X,\Delta)$ saying that $(X,\Delta)$ admits
a (singular) Kähler-Einstein metric iff it is K-stable in the algebro-geometric
sense. We thus obtain a new proof and extend to the log Fano setting
the seminal result of Zhou-Wang concerning the case when $X$ is smooth
and $\Delta$ is trivial. Li's toric formula for the greatest lower
bound on the Ricci curvature is also generalized. More generally,
we obtain Kähler-Ricci solitons on any log Fano variety and show that
they appear as the large time limit of the Kähler-Ricci flow. Furthermore,
using duality, we also confirm a conjecture of Donaldson concerning
solutions to Abreu's boundary value problem on the convex body $P$
in the case of a given canonical measure on the boundary of $P.$

\tableofcontents{}
\end{abstract}
\maketitle

\section{Introduction}

\subsection{Monge-Ampère equations in $\R^{n}$}

Let us start by recalling the setting for the second boundary value
problem for the real Monge-Ampère operator in the entire space $\R^{n}$
\cite{bak}. A convex function $\phi$ on $\R^{n}$ is said to be
a (classical) solution for the latter problem if it is smooth and
satisfies the following two conditions: 
\[
(i)\,\,\,\det(\frac{\partial^{2}\phi}{\partial x_{i}\partial x_{j}})=F(\phi,d\phi),
\]
 where $F$ is a given positive smooth function on $\R^{n+1}$ and
\[
(ii)\,\,\, d\phi(\R^{n})=\Omega
\]
 where $\Omega$ is a (necessarily convex) given domain in $\R^{n}.$
We will be concerned with the case when the domain $\Omega$ is \emph{bounded,
}i.e. its closure $P:=\bar{\Omega}$ is a convex body and 
\[
F(t,p)=e^{-\gamma t}g(p)^{-1}
\]
for $\gamma\in\R,$ where $g$ is a positive smooth function on $\R^{n}.$
After a trivial scaling we may as well assume that $\gamma=\pm1.$
As is well-known, the positive case is, by far, most challenging one
and the equation does then usually not admit any solutions. Our main
result gives the general structure of the solutions:
\begin{thm}
\label{thm:existe real ma intr}Let $P$ be a convex body containing
$0$ in its interior. Then there is a smooth convex function $\phi$
such that 
\[
g(d\phi)\det(\frac{\partial^{2}\phi}{\partial x_{i}\partial x_{j}})=e^{-\phi}
\]
and such that its gradient induces a diffeomorphism 
\[
d\phi:\,\R^{n}\rightarrow\mbox{int }(P)
\]
iff $0$ is the barycenter for the measure $g(p)dp$ on $P.$

Moreover, $\phi$ is then uniquely determined up to the action of
the additive group $\R^{n}$ by translations and
\begin{itemize}
\item $\phi(x)-\sup_{p\in P}\left\langle x,p\right\rangle $ is globally
bounded on $\R^{n}$ 
\item the Legendre transform $\phi^{*}$ of $\phi$ is Hölder continuous
up to the boundary of $P$ for any Hölder exponent in $[0,1[.$ 
\end{itemize}
\end{thm}
The proof uses a variational approach to construct a solution $\phi$
as a maximizer of the functional 
\[
\mathcal{G}(\phi):=\log\int_{\R^{n}}e^{-\phi}dx-\int_{P}\phi^{*}gdp
\]
on the space of all convex functions whose gradient image is $P.$
The main point of the argument is to establish a direct coercivity
estimate for the latter functional of independent interest, which
can be seen as a refined Moser-Trudinger type inequality (see Theorem
\ref{thm:coerciv k-e}). In fact, the argument shows that \emph{any}
asymptotically minimizing sequence of the functional above converges
- up to normalization - to a solution $\phi$ as in the previous theorem.
This extra flexibility will be used when establishing the convergence
of the Kähler-Ricci flow below.

\subsection{Toric Kähler-Einstein geometry}

We will mainly focus on the case when 
\[
g(p)=e^{\left\langle a,p\right\rangle }
\]
 for a given vector $a\in\R^{n}.$ The main differential-geometrical
motivation comes from the study of \emph{Kähler-Einstein metrics}
or more generally \emph{Kähler-Ricci solitons }on toric varieties.

\subsubsection{Kähler-Einstein metrics}

Recall that a Kähler form $\omega$ on a compact complex manifold
$X$ is a closed positive two-form, which equivalently means that,
locally, 
\begin{equation}
\mbox{\ensuremath{\omega}}=i\partial\bar{\partial}\phi\label{eq:omega in terms of k pot}
\end{equation}
for a local function $\phi,$ called the Kähler potential. The Kähler
metric $\omega$ is said to be\emph{ Kähler-Einstein }if the Riemannan
metric defined by its real part has constant Ricci curvature, which
in form notation is written as
\[
\mbox{Ric \ensuremath{\omega=\gamma\omega}}
\]
for some $\gamma=0,-1$ or $+1.$ Since the Ricci form $\mbox{Ric \ensuremath{\omega}}$
represents the first Chern class of $X:$ 
\[
c_{1}(X):=c_{1}(-K_{X}),
\]
where $-K_{X}:=\Lambda^{n}(TX),$ is the anti-canonical line bundle
on $X,$ it follows that, in the case $\gamma=\pm1,$ the Kähler potential
$\phi$ in \ref{eq:omega in terms of k pot} represents a positively
curved metric on the line bundle $-\gamma K_{X}.$ Hence, if $X$
admits a Kähler-Einstein metric then $c_{1}(X)$ is non-positive if
$\gamma\leq0$ and positive if $\gamma>0.$ Conversely, as shown in
the fundamental works of Yau and Aubin (when $\gamma<0)$ any complex
manifold $X$ with $c_{1}(X)$ non-positive admits a Kähler-Einstein
metric. However, in the case when $c_{1}(X)$ is positive, i.e. $X$
is a\emph{ Fano manifold, }there are well-known obstructions to the
existence of Kähler-Einstein metrics and the fundamental Yau-Tian-Donaldson
conjecture expresses all the obstructions in terms of a suitable notion
of algebro-geometric stability (see section \ref{sub:K-stability}):
\begin{conjecture*}
(Yau-Tian-Donaldson) A Fano manifold $X$ admits a Kähler-Einstein
metric iff it is $K-$polystable
\end{conjecture*}
More generally, from the point of view of current birational algebraic
geometry, or more precisely the Minimal Model Program (MMP), is is
natural to allow $X$ to be a\emph{ singular} Fano variety or more
generally to consider the category of \emph{log Fano varieties} $(X,\Delta)$,
where $X$ is a normal algebraic variety and $\Delta$ is a $\Q-$divisor
on $X$ such that the log anti-canonical line bundle $-(K_{X}+\Delta)$
is an ample $\Q-$line bundle. Here will also assume, as usual, that
the coefficients of $\Delta$ are $<1,$ but we do allow negative
coefficients (see section \ref{sub:Log-Fano-varieties}). The notion
of K-stability still makes sense for $X$ singular (see \cite{od,l-x}
for recent developments) and its log version was recently considered
in \cite{do3,li2,o-s}. As for the notion of a Kähler-Einstein metric
$\omega$ associated to a log Fano variety $(X,\Delta)$ it was recently
studied in \cite{bbegz}: by definition $\omega$ is a Kähler current
in $c_{1}(-(K_{X}+\Delta))$ with continuous potentials, satisfying
the following equation of currents on $X:$ 
\begin{equation}
\mbox{Ric \ensuremath{\omega-[\Delta]=\omega}}\label{eq:k-e eq for log pair intro}
\end{equation}
By the regularity result in \cite{bbegz} such a (singular) metric
$\omega$ restricts to a bona fide Kähler-Einstein metric on the Zariski
open set $X_{0}$ defined as the complement of $\Delta$ in the regular
locus of $X.$ See also section \ref{sub:Further-comparison-with}
below for relations to the theory of Kähler-Einstein metrics on Fano
manifolds with edge-cone singularities, where there is been great
progress recently.

The present paper concerns the case of \emph{toric} log Fano varieties
$(X,\Delta).$ In particular, $X$ is a\emph{ }toric variety, i.e.
a compact projective algebraic variety with an action of the complex
torus 
\[
T_{c}:=\C^{*n}\approxeq T\times\R^{n},
\]
(where $T$ is the real torus) such that $(X,T_{c})$ is an equivariant
compactification of $T_{c}$ with its standard action on itself and
the divisor $\Delta$ is supported ``at infinity'', i.e. in $X-T_{c}.$
As explained in section \ref{sec:Toric-varieties,-polytopes}) there
is a correspondence 
\[
(X,\Delta)\longleftrightarrow P
\]
 between $n-$dimensional toric log Fano varieties $(X,\Delta)$ and
rational convex polytopes $P$ in $\R^{n}$ containing $0$ in their
interior. Briefly, if $P$ is written as the intersection of affine
half-spaces $\left\langle l_{F},\cdot\right\rangle \geq-a_{F},$ where
the index $F$ runs over all facets of $P$ and $l_{F}$ denotes the
inward primitive lattice vector, normal to the facet $F,$ then 
\[
\Delta=\sum_{F}(1-a_{F})D_{F},
\]
 where $D_{F}$ is the toric invariant prime divisor defined by the
facet $F.$ Applying Theorem \ref{thm:existe real ma intr} to such
a polytope $P,$ with $g=1,$ we then deduce the following
\begin{thm}
\label{thm:-existence of ke intro} Let $X$ be a toric Fano variety.
Then the following is equivalent:
\begin{itemize}
\item $(X,\Delta)$ admits a toric log Kähler-Einstein metric
\item $0$ is the barycenter of the canonical polytope $P_{(X,\Delta)}$
associated to $X$
\item The log Futaki invariants of $(X,\Delta)$ vanish
\item $(X,\Delta)$ is log $K-$polystable with respect to toric degenerations
\end{itemize}
\end{thm}
This confirms the (generalized) Yau-Tian-Donaldson conjecture in the
category of toric log Fano varieties. Of course, it is natural to
ask if log K-polystability wrt toric degenerations implies log K-polystability
wrt \emph{any} test configuration? In fact, as shown very recently
in \cite{berm3}, in a general non-toric setting, the existence of
a log Kähler-Einstein metric does imply log K-polystability and hence
the full Yau-Tian-Donaldson conjecture holds for any toric log Fano
variety.

In the case when $X$ is smooth and $\Delta$ is the trivial divisor
the previous theorem was first shown in the seminal work \cite{w-z}
by Zhou-Wang, except for the last point, proven in \cite{z-z0}. One
of our motivations for considering Kähler-Einstein metrics on \emph{singular
}toric varieties $X$ comes from our recent work on the Ehrhart volume
conjecture for polytopes \cite{b-b2}. Another motivation comes from
the fact that, while there exist only a finite number of smooth Fano
varieties of dimension $n,$ there exists an infinite number of singular
ones. On the other hand it is well-known that the number becomes finite
if the Gorenstein index of $X$ is fixed. The most well-studied class
of toric Fano varities are those of Gorenstein index one, which correspond
to\emph{ reflexive} lattice polytopes $P$ (i.e. the dual $P^{*}$
is also a lattice polytope). This is a huge class of lattice polytopes
which plays an important role in string theory, as they give rise
to many examples of mirror symmetric Calabi-Yau manifolds \cite{ba}.
Already in dimension three there are 4319 isomorphism classes of such
polytopes \cite{k-s}, while there are only 105 families of \emph{smooth
}Fano threefolds, all in all. 

For a general log Fano variety $(X,\Delta)$ we also obtain a generalization
of recent results of Székelyhidi \cite{sz} and Li \cite{li1} concerning
greatest lower bounds on the Ricci curvature of metrics in $c_{1}(-(K_{X}+\Delta))$
(see Theorems \ref{thm:general R as ricci curv}, \ref{thm:R for lof fano}).

\subsubsection{Kähler-Ricci solitons}

In the case when $X$ is smooth it was furthermore shown in \cite{w-z}
that any toric Fano manifold admits a (shrinking)\emph{ Kähler-Ricci
soliton, }i.e. a Kähler metric $\omega$ and an associated holomorphic
vector field $V$ on $X$ such that 
\begin{equation}
\mbox{Ric \ensuremath{\omega=\omega}}+L_{V}\omega,\label{eq:k-r soliton eq intro}
\end{equation}
 where $L_{V}$ denotes the Lie derivative of $\omega$ wrt (the real
part of) $V.$ In the case when $(X,\Delta)$ is a log Fano variety
we will say that $\omega$ is a \emph{log Kähler-Ricci soliton} associated
to $(X,\Delta,V)$ if $\omega$ is a Kähler current in $c_{1}(-(K_{X}+\Delta))$
with continuous potentials satisfying the equation \ref{eq:k-r soliton eq intro},
with $\mbox{Ric \ensuremath{\omega}}$ replaced by the log Ricci curvature
$\mbox{Ric \ensuremath{\omega-}}[\Delta]$ and such that $\omega$
is smooth on $X_{0}.$ One motivation for studying Kähler-Ricci solitons
on a singular toric variety (even when $\Delta$ is trivial) is a
conjecture of Tian \cite{ti1} saying that on any Fano manifold the
Kähler-Ricci flow converges, modulo automorphisms, to a Kähler-Ricci
soliton on a Zariski open set of codimension at least two (the complex
structure is allowed to jump in the limit; see also \cite{so-t} for
the corresponding conjecture for general Fano varieties). 
\begin{thm}
\label{thm:existe of k-r sol intro}Any toric log Fano variety $(X,\Delta)$
admits a (singular) toric log Kähler-Ricci soliton $(\omega,V),$
where the metric $\omega$ is unique up to toric automorphisms and
the vector field $V$ is uniquely determined by the vanishing of the
modified log Futaki invariants associated to $(X,\Delta,V).$ Concretely,
$V$ is the invariant holomorphic vector field with components $a_{i},$
where the vector $a$ is the unique critical point of the Laplace
transform of the measure $1_{P_{(X,\Delta)}}dp.$ 
\end{thm}
We briefly remark that, given a log Fano variety $(X,\Delta),$ it
seems natural to expect that one can obtain a \emph{complete} Kähler-Ricci
soliton on the quasi-projective variety $X-\Delta$ (i.e. the complement
in $X$ of the support of $\Delta)$ by taking suitable limits of
log Kähler-Ricci solitons (see section \ref{sub:Relations-to-complete}).
This is in line with the discussion in \cite{do3} concerning limits
of Kähler-Einstein metrics with edge-cone singularities.

It is interesting to compare the uniqueness property for toric Kähler-Einstein
metrics contained in the previous theorem with the general results
in \cite{bern2,bbegz} saying that any two log Kähler-Einstein metrics
associated to a given log Fano variety $(X,\Delta)$ coincide up to
the action of the automorphism group of $\mbox{\ensuremath{(X,\Delta)}},$
when $\Delta$ has\emph{ positive} coefficients. However, when negative
coefficients are present it is well-known that this uniqueness property
fails in general and hence the uniqueness property in the toric category
- in the case when the divisor $\Delta$ has negative coefficients
- appears to be rather surprising (compare the discussion in example
\ref{ex: on the sphere}).

We will also show, building on \cite{bbegz}, that on any Fano variety
$X$ the Kähler-Ricci flow converges weakly, modulo automorphisms,
to a Kähler-Ricci soliton on $X$ (Theorem \ref{thm:conv of k-r flow}).
This gives a (weak) confirmation of the toric case of the conjecture
in \cite{so-t} (which asks for the stronger notion of Gromov-Hausdorff
convergence). We recall that in the case of a smooth Fano variety,
not necessary toric, the (strong) convergence towards a Kähler-Ricci
soliton - when one exists - was shown by Tian-Zhou \cite{t-z2}, using
Perelman's estimates.

W next turn to a dual version of Theorem \ref{thm:existe real ma intr}
formulated directly on the convex body $P.$ It concerns the ``Kähler-Einstein
case'' when $g=1$ and is motivated by the works of Abreu \cite{ab}
and Donaldson \cite{do2}. First recall that a Kähler metric $\omega$
on a complex complex manifold $X$ satisfies the Kähler-Einstein equation
precisely when $\omega$ is in $c_{1}(X)$ and its scalar curvature
$S_{\omega}$ is constant and equal to one with appropriate normalizations.
Moreover, as shown by Bando-Mabuchi $\omega$ is then a minimizer
of the Mabuchi K-energy functional.

\subsection{Abreu's equation on a convex body}

As shown by Abreu \cite{ab} in the toric setting the scalar curvature
of the Kähler metric on $T_{c}$ induced by the Hessian of a smooth
and strictly convex function $\phi$ on $\R^{n}$ may be written in
term of the Legendre transform $u$ of $\phi$ as $S(u),$ where $S$
is the following fourth order fully non-linear operator: 
\begin{equation}
S(u):=-\sum_{i=1}^{n}\frac{\partial^{2}u^{ij}}{\partial x_{i}\partial x_{j}},\label{eq:def of S oper intro}
\end{equation}
 where $(u^{ij})$ denotes the inverse of the Hessian matrix $(u_{ij})=(\frac{\partial^{2}u}{\partial x_{i}\partial x_{j}}).$
As a consequence any smooth solution $\phi$ as in Theorem \ref{thm:existe real ma intr}
(for $g=1)$ yields a solution to an equation in the interior of $P$
involving $S(u):$ 
\[
S(u)=1
\]
But there may be many very different solutions to the latter equation
since the boundary behavior of $u$ at $\partial P$ has to be taken
into account. To make this precise we note that there is canonical
measure $\sigma_{P}$ defined on the boundary of $P.$ It may be defined
in terms of the normal variations of the domain $P$ (see formula
\ref{eq:canonical measure text}). Following Donaldson \cite{d0}
any measure $\sigma,$ absolutely continuous wrt the induced Euclidean
measure $\lambda_{\partial P}$ on the boundary, defines a functional
$\mathcal{F}_{\sigma}$ on the space $\mathcal{C}^{\infty}$ of all
strictly convex functions $u$ on $P$ which are smooth in the interior
and continuous up to the boundary: 
\[
\mathcal{F}_{\sigma}(u):=-\int_{P}\log\det(u_{ij})dp+\mathcal{L}_{\sigma}(u),
\]
where $\mathcal{L}_{\sigma}$ is the linear functional 
\begin{equation}
\mathcal{L}_{\sigma}(u):=(\int_{\partial P}u\sigma-a\int_{P}udp),\,\,\,\, a:=\int_{\partial P}\sigma/\int_{P}dp\label{eq:def of general  linear functional intro}
\end{equation}
As explained in section \ref{sub:The-Mabuchi-funtional} the functional
$\mathcal{F}_{\sigma_{P}}$ may, in the case when $P=P_{(X,\Delta)}$
for a log Fano variety $(X,\Delta)$ be identified with the log version
of the Mabuchi K-energy functional. 
\begin{thm}
\label{thm:donaldson type intro}Let $P$ be a convex body containing
$0$ in its interior. Then the functional $\mathcal{F}_{\sigma_{P}}$
admits a minimizer $u$ in $\mathcal{C}^{\infty}$ iff $0$ is the
barycenter of $P.$ Moreover, the minimizer is then unique modulo
the addition of affine functions and satisfies Abreu's equation 
\[
S(u)=1
\]
in the interior of $P.$ 
\end{thm}
Donaldson conjectured (see Conjecture 7.2.2 in \cite{d0}) that, given
\emph{any} measure $\sigma$ as above there is a corresponding minimizer
under the following condition:

\[
\mathcal{L}_{\sigma}(u)>0
\]
for any non-affine convex function. In our ``canonical case'' where
the measure in question is $\sigma_{P}$ the latter condition is satisfied
precisely when $0$ is the barycenter of $P$ (see Lemma \ref{lem:lower bound on l when bary})
and the previous corollary thus confirms Donaldson's conjecture in
this case. The case when $P$ is a two-dimensional polytope and $\sigma$
coincides with a multiple of $\lambda_{\partial P}$ on each facet
was settled by Donaldson in a series of papers leading up to \cite{do2}.
As emphasized in \cite{d0} the main motivation for Donaldson\textquoteright{}s
conjecture comes from the toric version of the general Yau-Tian-Donaldson
conjecture concerning constant scalar curvature metrics in $c_{1}(L)$
for a given polarized manifold $(X,L),$ which, as explained by Donaldson,
corresponds to a certain measure on the boundary of the lattice polytope
$P_{(X,L)}$ determined by the integral structure. As it turns this
latter measure coincides with our measure $\sigma_{P}$ precisely
when $(X,L)$ is equal to $(X,-K_{X})$ for a toric Fano variety (see
section \ref{sub:Comparison-with-Donaldson's}). The point - from
our point of view - is that any toric line bundle $L\rightarrow X$
can always be written as $L=-(K_{X}+\Delta),$ where $(X,\Delta)$
is a toric log Fano variety and hence Theorem \ref{thm:existe of k-r sol intro}
furnishes, under the corresponding barycenter condition, a Kähler
current in $c_{1}(L)$ with constant Ricci curvature on $X-\Delta$
and where the singularities along $\Delta$ are encoded by the measure
$\sigma_{P}$ (compare Cor \ref{cor:existe of log ke if delta changes}).

\subsection{\label{sub:Further-comparison-with}Further comparison with previous
results and methods}

In terms of toric geometry the key ingredient in our approach is a
direct convexity argument showing that the Ding type functional $\mathcal{G}_{(X,\Delta,V)}$
associated to a toric log Fano variety $(X,\Delta)$ with a toric
vector field $V$ is\emph{ relatively proper }(in the sense of \cite{z-z,z-z2})
and even\emph{ relatively coercive }on the space of $T-$invariant
metrics, if the appropriate assumption on the barycenter (Futaki invariant)
holds (see Theorems \ref{thm:coerciv k-e}, \ref{thm:mab functional is coerc for toric is equi to ke}).
Given this relative properness we can adapt the variational approach
in \cite{begz,berm2,bbegz} to our setting to deduce the existence
of a maximizer satisfying the corresponding Kähler-Ricci solution
equation. The coercivity of $\mathcal{G}_{(X,\Delta,V)}$ implies
in particular that the corresponding Mabuchi K-energy type functional
$\mathcal{M}_{(X,\Delta,V)}$ is also relatively coercive. It should
be pointed out that in the general setting of a \emph{smooth} Fano
manifold $X,$ not necessarily toric, but with $\Delta=0,$ the properness
of the corresponding functionals - a priori assuming the existence
of a Kähler-Einstein metric - was shown by Tian \cite{ti1}, who also
conjectured its coercivity, eventually proved in \cite{p-s-s-w}.
For the corresponding results in the presence of a Kähler-Ricci soliton,
see \cite{c-t-z}. 

Another variational approach approach, in the more general setting
of constant scalar curvature Kähler metrics in $c_{1}(L),$ for $(X,L)$
smooth and toric, has been developed in \cite{z-z,z-z0,z-z2} building
on \cite{d0}. In particular, it is shown in \cite{z-z2}, that if
the corresponding Mabuchi functional $\mathcal{M}_{(X,L)}$ is relatively
proper, then it admits a weak minimizer. However, the question of
its regularity and whether it satisfies the constant scalar curvature
equation was left open. One virtue of the present approach is thus
that, when $L=-K_{:X},$ the minimizer can indeed be shown to satisfy
the Kähler-Einstein equation, even in the general setting of log Kähler-Einstein
metrics and Kähler-Ricci solitons. On the other hand, our methods
are closely tied to the Monge-Ampère operator and it does not seem
clear, at this point, how to extend them to the general setting of
constant scalar curvature metrics. 

Log Fano varieties $(X,\Delta)$ with $X$ smooth and $\Delta=(1-t)D$
for a smooth divisor $D$ and $t\in[0,1[$ have recently been studied
in depth in \cite{do3,j-m-r} from the point of view of \emph{edge-cone
singularities. }In particular, assuming that the corresponding Mabuchi
functional $\mathcal{M}_{(X,\Delta)}$ is proper it was shown in \cite{j-m-r}
how to use a continuity method to obtain Kähler-Einstein metrics on
$X-D$ which have cone singularities with an angle $2\pi t$ transversely
to $D$ (and in particular the metrics satisfy the equation \ref{eq:k-e eq for log pair intro}
on $X).$ More precisely, the metrics admit a polyhomogenous expansion
along $D$ in the sense of the ``edge calculus''. It seems likely
that, using these latter results, it can be shown that when $(X,(1-t)D)$
is moreover toric the Kähler-Einstein metrics constructed here also
have cone singularities etc. However, there is a technical problem
coming from the fact that in the toric setting $\mathcal{M}_{(X,\Delta)}$
is only \emph{relatively} proper due to the presence of holomorphic
vector fields. It should also be pointed that, under the assumption
that $t\in]0,1/2[$ (and similarly in the log smooth case where $X$
is smooth and $D$ has simple normal crossings) it is shown in \cite{cgh}
that\emph{ any} log Kähler-Einstein metrics has cone singularities.
Of course, it would also be very interesting to understand the relations
to edge-cone type singularities in the case when $X$ itself is singular.
However, at this point it does not even seem clear what the appropriate
local models are, even if $\Delta=0.$

It should also be pointed out that in the case when $X$ is a Fano
variety with quotient singularities, i.e. $X$ is an orbifold (which
corresponds to the polytope $P$ being simple) the existence of a
Kähler-Ricci soliton was obtained recently in \cite{s-x}, building
on \cite{w-z}. The orbifold situation was further studied in \cite{leg}. 

When the first draft of the present paper had been completed two new
preprints of Song-Wang \cite{so-j} and Li-Sun \cite{li-s} appeared
which are relevant for the discussion on edge-cone Kähler-Einstein
metrics above. In particular, in \textbf{\cite{so-j} }certain\emph{
toric} edge-cone Kähler-Einstein metrics are obtained on any given
smooth toric Fano variety $X,$ by a method of continuity. We have
included a discussion on the more precisely relations to \cite{so-j,li-s}
in section \ref{sub:Existence-of-K=0000E4hler-Einstein singular along div}.

\subsection{Organization}

In section \ref{sec:Monge-Amp=0000E8re-equations-in} we start by
setting up a variational approach to solving Monge-Ampère equations
in $\R^{n}$ with target a convex body. The core of the section is
a direct proof of a coercive Moser-Trudinger type inequality, which
is the basis of the proof of Theorem \ref{thm:existe real ma intr}
stated in the introduction. In the following section \ref{sec:Toric-varieties,-polytopes}
we give a fairly detailed exposition of toric varieties emphasizing
analytical aspects of toric log Fano varieties, which in particular
allows us to rephrase the results in the previous section in terms
of toric Kähler-Einstein geometry. Then in section we \ref{sec:K-energy-type-functionals}
explore the relations to the Mabuchi K-energy functional, Futaki invariant
and K-stability. Finally, in section \ref{sec:Convergence-of-the}
we show that the (weak) Kähler-Ricci flow on any toric Fano variety
converges weakly to a (singular) Kähler-Ricci soliton. 

At least part of the length of the paper is explained by our effort
to make the paper readable for the reader with a background in convex
analysis, as well as for the complex geometers.

\section{\label{sec:Monge-Amp=0000E8re-equations-in}Monge-Ampère equations
in $\R^{n}$ and Convex bodies }

In this section we will adopt a direct variational approach to solve
the Monge-Ampère equation in Theorem \ref{thm:existe real ma intr},
stated in the introduction. This means that the solutions will be
obtained as the maximizers of a certain functional $\mathcal{G}$
on a space $\mathcal{E}_{P}^{1}(\R^{n})$ of convex functions on $\R^{n}$
of ``finite energy''. At least formally the solutions are critical
points of $\mathcal{G}$ and according to the usual scheme of the
calculus of variations the existence proof is thus divided into two
distinct parts: 
\begin{itemize}
\item A coercivity (properness) estimate for $\mathcal{G}$, which yields
the existence of a maximizer $\phi$ (when the barycenter condition
on $P$ holds) 
\item An argument showing that any maximizer indeed satisfies the equation
in question 
\end{itemize}
The key new ingredient in our approach is the accomplishment of the
first point using a direct convexity argument. As for the second point
we will develop a real analog of the Kähler geometry setting considered
in \cite{bbgz,begz,bbegz} by introducing appropriate finite energy
spaces of convex functions and establishing a crucial differentiability
result for the ``energy of convexification'' (Prop \ref{pro:diff of composed energy convex}). 

Of course, our comparison with the Kähler geometry setting may appear
as an anachronism: the variational approach to real Monge-Ampère equations,
originating in Alexandrov's seminal work on the Minkowski problem
on the $n-$sphere (see the book \cite{bak} and references therein)
certainly precedes its complex analog. On the other hand, as far as
we know the precise convex analytical setting in $\R^{n}$ (as opposed
to the $n-$sphere) that we we need does not appear to have been developed
in the literature. Moreover, the analogy between the real and complex
settings gives a useful testing ground for conjectures in the Kähler
geometry setting.

It is also interesting to see that our variational approach is closely
related to the variational principles appearing in the theory of optimal
transport (see section \ref{sub:The-inhomogenous-Monge-Amp=0000E8re}).

\subsection{Setup}

Let $P$ be a\emph{ convex body} in an affine space of real dimension
$n$, i.e. $P$ is a compact convex subset with non-empty interior.
Identifying the affine space with the vector space $\R^{n},$ with
linear coordinates $p(=(p_{1},...,p_{n}),$ we may as well assume
that the origin $0$ is contained in the interior of $P.$ We will
identify the dual vector space with $\R^{n},$ with linear coordinates
$x(=(x_{1},...,x_{n}),$ 

A convex functions $\phi(x)$ on $\R^{n}$ is, by definition, convex
along affine lines, i.e. $\phi(tx+(1-t)y)\leq t\phi(x)+(1-t)\phi(y)$
and takes values in $]\infty,\infty]$ and we will exclude the case
when $\phi$ is identically $\infty.$ Note that such functions are
called\emph{ proper} \emph{convex functions }in \cite{r}, while we
will, to conform to more standard general terminology, say that a
function $\phi(x)$ is\emph{ proper} if $|\phi|\rightarrow\infty,$
as $|x|\rightarrow\infty.$ 

The \emph{subdifferential} $d\phi_{|x}$ of $\phi$ at $x$ is the
closed set of $\R^{n}$ consisting of all $p$ such that $f(y)\geq f(x)+\left\langle p,y-x\right\rangle $
for $y\in\R^{n}.$ In particular, $d\phi_{|x}$ is a equal to a point
(the usual differential of $\phi$ at $x)$ if $\phi$ is differentiable
at $x.$ The Monge-Ampère measure $MA_{\}(\phi)}$ of a finite convex
function $\phi$ is the (Borel) measure, which with our normalization
convention is defined by 
\begin{equation}
(MA(\phi)(E):=n!\int_{d\phi(E)}dp\label{eq:def of real ma}
\end{equation}
for any Borel set $E$ (see \cite{r-t,gu}); this is sometimes also
called the \emph{Monge-Ampère measure in the sense of Alexandrov.}
More generally, given any function $g$ in $L^{1}\R^{n})$ we can
define the ``$g-$Monge-Ampère measure'' $MA_{g}(\phi)$ by replacing
the measure $dp$ in the definition \ref{eq:def of real ma} by $gdp,$
so that $MA_{g}(\phi)=g(d\phi)MA(\phi)$ if $\phi$ is smooth, In
fact, all the results below for the operator $MA$ generalize word
for word to this more general setting, but for clarity of exposition
we will mainly stick to the case when $g=1.$ 
\begin{rem}
\label{rem:g-ma}The reason that $MA(\phi)$ (and more generally $MA_{g})$
indeed defines a bona fide measure is that the multivalued map from
$\R^{n}\rightarrow\R^{n}$ defined by $x\mapsto d\phi_{|x}$ (often
called the ``normal mapping'' in the literature) is invertible almost
everywhere on its image (wrt Lebesgue measure $dp$). This is a consequence
of the almost everywhere differentiability of the Legendre transform
(compare Lemma 1.1.12 in \cite{gu} or Lemma \ref{pro:var prop of leg}
below). 
\end{rem}
Let now $P$ be a given convex body in $\R^{n}$ containing $0$ in
its interior and of volume 
\[
V(P):=\mbox{Vol\ensuremath{(P)}}
\]
and denote by $\mathcal{P}(\R^{n})$ be the space of all convex functions
$\phi(x)$ on $\R^{n}$ such that 
\[
\phi(x)\leq\phi_{P}(x)+C,
\]
 where $\phi_{P}$ is the \emph{support function} \emph{of $P,$}
i.e. the following one-homogenous convex function 
\[
\phi_{P}(x):=\sup_{p\in P}\left\langle x,p\right\rangle 
\]

We let $\mathcal{P}_{+}(\R^{n})$ be the subspace of $\mathcal{P}(\R^{n})$
of elements of ``maximal growth'': 
\[
-C+\phi_{P}(x)\leq\phi(x)\leq\phi_{P}(x)+C
\]
 In particular, any $\phi$ in $\mathcal{P}_{+}(\R^{n})$ is \emph{proper}.
Standard examples of strictly convex and smooth elements in n $\mathcal{P}_{+}(\R^{n})$
are obtained by setting 
\begin{equation}
\phi_{P,k}:=\frac{1}{k}\log\int_{P}e^{k\left\langle x,p\right\rangle }\frac{dp}{V}\label{eq:phi P k smooth}
\end{equation}
for a given positive integer $k$ (note that $\phi_{P}=\phi_{P,\infty}).$ 

We equip the space $\mathcal{P}(\R^{n})$ with the topology defined
by point-wise convergence. Thanks to the uniform Lipschitz bound on
the elements in $\mathcal{P}(\R^{n})$ (coming from the boundedness
$\mathcal{P}(\R^{n}))$ this coincides with the topology defined by
local\emph{ uniform} convergence. 
\begin{lem}
\label{lem:(regularization).-For-any}(regularization). For any $\phi$
in $\mathcal{P}(\R^{n})$ there is a sequence of strictly convex smooth
functions $\phi_{j}$ in $\mathcal{P}_{+}(\R^{n})$ decreasing to
$\phi.$\end{lem}
\begin{proof}
Given $t\in\R^{n}$ and $\phi$ in $\mathcal{P}(\R^{n})$ we have
that $\phi(\cdot+t)$ is also in $\mathcal{P}(\R^{n})$ and hence
we may first define $\psi_{j}$ by convolutions in the usual way so
that $\psi_{j}$ is smooth and decreases to $\phi.$ Finally, we may
then set $\phi_{j}:=\max_{\epsilon}(\psi_{j},\phi_{P}-j)$ where $\max_{\epsilon}$
denotes a regularized max, which has the required properties a part
from the \emph{strict} convexity.. But this may be achieved by taking
suitable convex combinations of $\phi_{j}$ and $\phi_{0},$ where
$\phi_{0}$ is any fixed smooth and strictly convex function in $\mathcal{P}_{+}(\R^{n})$
such that $\phi\leq\phi_{0},$ for example $\phi_{0}=\phi_{P,1}+C$
for a sufficiently large constant $C.$ 
\end{proof}

\subsection{\label{sub:Relation-to-the-log}Relation to the complex setting:
the Log map}

Let $\mbox{Log }$be the map from $\C^{*n}$ to $\R^{n}$ defined
by $\mbox{Log}(z):=x:=(\log(|z_{1}|^{2},...,\log(|z_{n}|^{2}),$ so
that the real torus $T$ acts transitively on its fibers, We will
refer to $x$ as the (real) \emph{logarithmic coordinates} on $\C^{*n}.$
Given a psh $T-$invariant bounded function $\phi(z)$ on $\C^{*n}$
we will, abusing notation slightly, write $\phi(x)$ for the corresponding
convex function on $\R^{n},$ i.e. $\phi(x):=\phi(z)$ for any $z\in\mbox{(Log)}^{.1}\{x\}.$
The normalizing constant $n!$ in the definition of $MA(\phi)$ has
been chosen so that 
\[
\mbox{(Log)}_{*}MA_{\C}(\phi)=MA(\phi)
\]
where $MA_{\C}(\phi)$ is the Monge-Ampère measure on $\C^{*n}$ of
the locally bounded psh function $\phi(z),$ as defined in pluripotential
theory (compare \cite{r-t} and section \ref{sub:Pluricomplex-energy-and}).
Note however, that $MA_{g}$ for $g\neq1$ does not have any immediate
pluripotential analog.

\subsection{The Legendre transform from $\R^{n}$ to the convex body $P$}

Recall that th\emph{e Legendre(-Fenchel) transform} (also called the\emph{
conjugate} function in \cite{r}) of a convex function $\phi(x)$
is defined by 
\[
\phi^{*}(p):=\sup_{x\in\R^{n}}\left\langle p,x\right\rangle -\phi(x)
\]
which is a convex function on $\R^{n}$ with values in $]-\infty,\infty].$
Since the Legendre transform is an involution, i.e. $\phi^{^{**}}=\phi,$
we have 
\[
\phi_{P}^{*}(x)=0\,\mbox{on\,\ensuremath{P,\,\,\,\phi_{P}^{*}(x)=\infty\,\mbox{on\,\ensuremath{\R^{n}-P}}}}
\]
and 
\begin{equation}
\phi(x)\leq\psi(x)\Leftrightarrow\phi^{*}(p)\geq\psi^{*}(p)\label{eq:using monoton of legendre}
\end{equation}
It follows immediately that the following proposition holds:
\begin{prop}
\label{pro:leg is isometric} If $\phi$ is in $\mathcal{P}(\R^{n})$
then $\phi^{*}=\infty$ on the complement of $P.$ Moreover, the Legendre
transform induces a bijection between $\mathcal{P}_{+}(\R^{n})$ and
the space $\mathcal{H}(P)$ of bounded convex functions on $P$ such
that 
\[
\sup_{\R^{n}}(\phi-\phi_{P})=-\inf_{P}\phi^{*},\,\,\,\inf_{\R^{n}}(\phi-\phi_{P})=-\sup\phi^{*}
\]
and 
\[
\left\Vert \phi-\phi_{P}\right\Vert _{L^{\infty}(\R^{n})}=\left\Vert \phi^{*}\right\Vert _{L^{\infty}(P)}
\]

\end{prop}

\subsubsection{\label{sub:Functions-with-full}Functions with full Monge-Ampère
mass}

We will say that an element $\phi$ in $\mathcal{P}(\R^{n})$ has
\emph{full Monge-Ampère mass} if the total mass of $MA(\phi)$ on
$\R^{n}$ coincides with $n!$ times the volume $V(P)$ of $P.$ Following
the terminology in \cite{begz} in the Kähler geometry setting we
will denote this subclass of $\mathcal{P}(\R^{n})$ by $\mathcal{E}_{P}(\R^{n})$
(compare Remark \ref{rem:weighted energy classes}  below).
\begin{prop}
\label{pro:cont for full ma}If $\phi_{j}$ converges to $\phi$ in
$\mathcal{E}_{P}(\R^{n}),$ then $\int vMA(\phi_{j})\rightarrow\int vMA(\phi)$
for any bounded continuous function $v$ on $\R^{n}.$\end{prop}
\begin{proof}
If $v$ has compact support this is well-known to hold for\emph{ any}
sequence $\phi_{j}$ of convex functions converging locally uniformly
to $\phi$ \cite{r-t,gu}. Moreover, if $\phi_{j}$ converges $\phi$
in $\mathcal{E}_{P}(\R^{n}),$ then by assumption $\int MA(\phi_{j})=\int MA(\phi).$
Hence, writing $v$ as $v(\chi+(1-\chi))$ for $\chi$ a cut-off function
supported on a sufficiently large ball proves the proposition. 
\end{proof}
According to the following basic lemma any $\phi\in\mathcal{P}_{+}(\R^{n})$
has full Monge-Ampère mass, i.e. it is in the class $\mathcal{E}_{P}(\R^{n}).$
\begin{lem}
\label{lem:prop of image}The following properties of the image of
the subgradients of convex functions hold:
\begin{itemize}
\item If $\phi$ is a finite convex function $\phi$ on $\R^{n}$ then $d\phi(\R^{n})\subset\{\phi^{*}<\infty\}.$
In particular $d\phi(\R^{n})\subset P$ if $\phi\in\mathcal{P}(\R^{n})$
and $\phi\in\mathcal{E}_{P}(\R^{n})$ iff $d\phi(\R^{n})=P$ up to
a set of measure zero. 
\item If $\phi\in\mathcal{P}_{+}(\R^{n})$ then the interior of $P$ is
contained in $\{\phi^{*}<\infty\}$ and hence
\begin{equation}
\phi\in\mathcal{P}_{+}(\R^{n})\implies\int_{\R^{n}}MA_{\R}(\phi)=n!V(P)\label{eq:volume of convex f}
\end{equation}

\end{itemize}
\end{lem}
\begin{proof}
The first point follows immediately from the definition of a subgradient
and the second point follows from the fact that $\left\langle p,x\right\rangle -\phi(x)$
is clearly proper if $p$ is an interior point of $P$ and $\phi\in\mathcal{P}_{+}(\R^{n}).$
Indeed, then the sup defining $\phi^{*}(p)$ is attained, say at $x_{p},$
and it follows that $p\in d\phi_{|x_{p}}.$ The final statement \ref{eq:volume of convex f}
then follows from the well-known fact that the topological boundary
of $P$ is a nullset for Lebesgue measure.
\end{proof}
Before continuing it will be convenient to record the following property:
\begin{lem}
\label{lem:full ma is proper}Any $\phi$ in $\mathcal{P}(\R^{n})$
with full Monge-Ampère mass is proper. More precisely, there exists
a constant $C>0$ such that $\phi(x)\geq|x|/C-C$.\end{lem}
\begin{proof}
First note that if $\phi$ is an element in $\mathcal{E}(\R^{n})$
then $\{\phi^{*}<\infty\}$ is the closure of the interior of $P$
(the converse is trivial). Indeed, by the first point in \ref{lem:prop of image}
(and since the topological boundary of $P$ is a nullset) the interior
of $P$ has full measure in $\{\phi^{*}<\infty\}$ and in particular
is dense in the convex set $\{\phi^{*}<\infty\}.$ But then it follows
from a simple argument, using convexity, that all of the interior
has to be contained in $\{\phi^{*}<\infty\}.$ Finally, we note that
if $\phi$ is a convex function (finite) convex function $\phi$ on
$\R^{n}$ such that $0$ is contained in the interior of $\{\phi^{*}<\infty\}$
then $\phi$ is proper. Indeed, by assumption $u:=\phi^{*}$ is finite
on a closed small ball $B_{\epsilon}$ of radius $\epsilon$ centered
at $0$ and since $\phi^{*}$ is continuous there if follows that
$|\phi^{*}|\geq C$ on $B_{\epsilon}.$ Hence, $\phi(x)=u^{*}(x)\geq\sup_{p\in B_{\epsilon}}\left\langle p,x\right\rangle -C=\epsilon|x|-C$
which concludes the proof.
\end{proof}
We will also have great use for the following variational properties
of the Legendre transform:
\begin{lem}
\label{pro:var prop of leg}Let $\phi\in\mathcal{P}(\R^{n})$ and
$p$ an element in the convex set $\{\phi^{*}<\infty\}(\subset P)$ 
\begin{itemize}
\item $\phi^{*}$ is differentiable at $p$ iff the sup defining $\phi^{*}$
is attained a unique point $x_{p}$ and the differential at $p$ is
then given by $x_{p}=d\phi_{|p}^{*}$
\item If $\phi$ has full Monge-Ampère mass and $v$ is a continuous function
on $\R^{n}$ and $\phi^{*}$ is differentiable at $p$ then 
\[
\frac{d(\phi+tv)^{*}}{dt}_{t=0}=-v(d\phi_{|p}^{*})
\]

\item Moreover we then have, for any non-negative continuous function $v,$
\[
\int_{\R^{n}}MA(\phi)v=\int_{P}v(d\phi_{|p}^{*})dp
\]
(where the rhs is well-defined since the derivative of a convex function
exists a.e. wrt $dp)$ 
\end{itemize}
\end{lem}
Since, these content of the lemma above appears to be mostly well-known
in the case when $\phi$ is smooth and strictly convex, we have, for
completeness, provided a proof of the general case in the appendix.

\subsection{\label{sub:Compactness,-normalization-and}Compactness, normalization
and the action of the group $\R^{n}$}

We let $\mathcal{P}(\R^{n})_{0}$ be the subspace of all \emph{sup-normalized}
$\mbox{\ensuremath{\phi}:}$ 
\[
\sup_{\R^{n}}(\phi-\phi_{P})=0
\]

\begin{prop}
\textup{If $\phi$ is in $\mathcal{P}(\R^{n})$ then 
\[
\sup_{\R^{n}}(\phi-\phi_{P})=0\iff\phi(0)=0
\]
and hence the space $\mathcal{P}(\R^{n})_{0}$ is compact.}\end{prop}
\begin{proof}
Since, by definition, the gradient image of $\phi$ is in $P$ it
follows from the convexity of $\phi$ along the affine line $t\rightarrow tx$
that $\phi(x)\leq\phi(0)+\phi_{P}(x),$ i.e. $\sup(\phi-\phi_{P})\leq\phi(0).$
Since trivially, $\phi(0)=\phi(0)-\phi_{P}(0)\leq\sup_{\R^{n}}(\phi-\phi_{P})$
this proves the equivalence in the proposition. In particular, if
$\phi_{j}$ is a sequence in $\mathcal{P}(\R^{n})_{0}$ then $\phi_{j}(0)=0.$
Hence, since that gradient image of $\phi$ is in $P$ (and in particular
bounded) we deduce that $\sup_{K}|\phi(x)|\leq C_{K}$ on any given
compact subset of $K$ and that $\phi$ is Lipschitz continuous on
$K$ with a uniform Lipschitz constant. Applying the Arzelà\textendash{}Ascoli
theorem on $K$ thus concludes the proof of the compactness.
\end{proof}
We will say that $\phi$ is \emph{normalized }if it is sup-normalized
and $\phi\geq0,$ i.e. 
\begin{equation}
0=\phi(0)=\inf_{\R^{n}}\phi\label{eq:def of normalized}
\end{equation}
Given a function a strictly convex function $\phi\in\mathcal{P}_{+}(\R^{n})$
we define its \emph{normalization} $\tilde{\phi}$ by 
\begin{equation}
\tilde{\phi}(a):=\phi_{a}-\phi(a)\label{eq:normalization}
\end{equation}
where $a$ is the point where the infimum of $\tilde{\phi}$ is attained
and $\phi_{a}(x):=\phi(x+a)$ defines the action of the group $\R^{n}$
on $\mathcal{P}_{+}(\R^{n})$ by translations. Note that even if $\tilde{\phi}$
is not strictly convex we may always define its normalization $\tilde{\phi}$
by taking some point $a$ where the infimum of $\tilde{\phi}$ is
attained (but $a$ may not be uniquely determined). Also note that
under the Legendre transform 
\begin{equation}
(\phi_{a})^{*}(p)=\phi^{*}(p)-\left\langle a,p\right\rangle \label{eq:transl under legendre}
\end{equation}
for any $a\in\R^{n}.$

\subsection{The functional $\mathcal{E}$ and the finite energy class $\mathcal{E}_{P}^{1}(\R^{n})$}

Fix a reference element $\phi_{0}$ in $\mathcal{P}_{+}(\R^{n}).$
Then there is a unique functional $\mathcal{E}:=\mathcal{E}(\cdot,\phi_{0})$
on $\mathcal{P}_{+}(\R^{n})$ such that 
\begin{equation}
d\mathcal{E}_{|\phi}=MA(\phi)\label{eq:first deriv of energy in global conv}
\end{equation}
 normalized so that $\mathcal{E}(\phi_{0},\phi_{0})=0.$ To see this
we may first define 
\begin{equation}
\mathcal{E}(\phi,\phi_{0}):=\int_{0}^{1}(\phi-\phi_{0})MA(\phi_{0}(1-t)+t\phi)dt\label{eq:def of energy as mixed in global conv}
\end{equation}
and then verify that \ref{eq:first deriv of energy in global conv}
indeed holds. This could be shown using integration by parts, but
we will give a different proof in the course of the proof of the following
proposition. We can first extend the functional $\mathcal{E}$ to
be defined on $\mathcal{E}_{P}(\R^{n})$ by the formula \ref{eq:def of energy as mixed in global conv}
and define the class of all $\phi$ in $\mathcal{P}(\R^{n})$ of\emph{
finite energy} by 
\[
\mathcal{E}_{P}^{1}(\R^{n}):=\{\phi\in\mathcal{E}_{P}(\R^{n}):\,\,\mathcal{E}(\phi)>-\infty
\]
Note that the various spaces are related as follows:

\begin{equation}
\mathcal{P}_{+}(\R^{n})\subset\mathcal{E}_{P}^{1}(\R^{n})\subset\mathcal{E}_{P}(\R^{n})\subset\mathcal{P}(\R^{n})\label{eq:inclusions of spaces in convex setting}
\end{equation}
More generally, given a bounded function $g$ on $P$ replacing $MA$
with $MA_{g}$ we obtain a corresponding functional $\mathcal{E}_{g},$
precisely as before. Anyway, since $g$ is bounded the corresponding
finite energy space is independent of $g.$
\begin{prop}
\label{pro:prop of energy in global conv}We have that in the space
$\mathcal{E}_{P}^{1}(\R^{n})$ the functional $\mathcal{E}(\cdot,\phi_{0})$
defined by \ref{eq:def of energy as mixed in global conv} (which
by definition is finite) satisfies 
\[
d\mathcal{E}_{|\phi}=MA(\phi)
\]
In general the functional $\mathcal{E}(\cdot,\phi_{P})$ be extended
to an increasing (wrt the usual order relation) and upper semi-continuous
functional $\mathcal{P}(\R^{n})\rightarrow[-\infty,\infty[$ by setting 

\begin{equation}
\mathcal{E}(\phi,\phi_{P})=-n!\int\phi^{*}dp\label{eq:energy as legendre in global conv}
\end{equation}
In particular, an element $\phi$ in $\mathcal{P}(\R^{n})$ is in
$\mathcal{E}_{P}^{1}(\R^{n})$ iff $\phi^{*}$ is in $L^{1}(P,dp).$
For the energy functional associated to a function $g$ on $P$ the
formula \ref{eq:energy as legendre in global conv} holds with $dp$
replaced by $gdp.$\end{prop}
\begin{proof}
Denote by $\mathcal{H}(P)$ the space of all convex functions on $P$
and denote by $L$ the map from $\mathcal{E}_{P}(\R^{n})$ to $\mathcal{H}(P)$
induced by the Legendre transform. i.e $(L\phi)(p):=\phi^{*}(p).$
Let $\lambda$ be the linear functional on $\mathcal{H}(P)$ defined
by integration on $P$ against $dp$ and let $MA$ be the one-form
on $\mathcal{E}_{P}(\R^{n})$ such that the linear functional $MA_{|\phi}$
is defined by integration on $\R^{n}$ against the Monge-Ampère measure
$MA(\phi).$ Then we have the following key relation 
\begin{equation}
L^{*}(-n!\lambda)=MA\label{eq:key relation}
\end{equation}
Indeed, this follows immediately from the second point in Lemma \ref{pro:var prop of leg},
since by definition $\left\langle L^{*},v\right\rangle _{\phi}=-n!\left\langle \lambda,d(L(\phi+tv))/dt_{t=0}\right\rangle .$
Since $\lambda$ clearly defines a\emph{ closed} one-form on $\mathcal{H}(P)$
(and even exact with primitive $I_{P}:=\int_{P}(\cdot)dp),$ it follows
that the pull-back $MA$ is also closed and even exact. In particular
it has a primitive $\mathcal{E}$ (unique up to the normalization
$\mathcal{E}(\phi_{0})=0$) such that 
\[
\mathcal{E}(\phi,\phi_{0})=\int_{[0,1]}\gamma^{*}MA,\,\,\,\gamma:\,[0,1]\rightarrow\mathcal{H}(\R^{n})
\]
where $\gamma$ is any smooth curve connecting $\phi_{0}$ and $\phi.$
In particular, taking $\gamma$ to be an affine curve in $\mathcal{H}(\R^{n})$
gives the formula \ref{eq:def of energy as mixed in global conv},
while taking $\gamma$ to be the pull-back under $L$ of an affine
curve in $\mathcal{H}(P)$ gives the formula \ref{eq:energy as legendre in global conv}
on $\mathcal{E}_{P}^{1}(\R^{n})$ (using that $L\phi_{P}=0).$

Finally, we note that it follows immediately from the fact that the
Legendre transform is decreasing together with Fatou's lemma that
the functional on $\mathcal{P}(\R^{n})$ defined by \ref{eq:energy as legendre in global conv}
is increasing and upper-semicontinuous.
\end{proof}
We will often omit the explicit dependence on a reference $\phi_{0}$
in the definition of $\mathcal{E}(\cdot,\phi_{0}).$ Anyway, as a
consequence of the property \ref{eq:first deriv of energy in global conv}
\emph{differences} $\mathcal{E}(\phi)-\mathcal{E}(\psi)$ are independent
of the choice of $\phi_{0}.$ 
\begin{rem}
Since any $\phi$ in $\mathcal{P}(\R^{n})$ may be written as a decreasing
limit of elements $\phi_{j}$ in $\mathcal{P}_{+}(\R^{n})$ (just
set $\phi_{j}=\max\{\phi,\phi_{p}+1/j\})$ we could, by the monotonicity
and upper semi-continuity of $\mathcal{E}$ in the previous proposition,
equivalently have defined $\mathcal{E}$ by $\mathcal{P}(\R^{n})$
by setting 
\begin{equation}
\mathcal{E}(\phi)=\inf\left\{ \mathcal{E}(\psi):\,\psi\in\mathcal{P}_{+}(\R^{n}),\,\,\,\psi\geq\phi\right\} \label{eq:energie as inf}
\end{equation}
using formula \ref{eq:def of energy as mixed in global conv} and
then let $\mathcal{E}_{P}^{1}(\R^{n}):=\{\mathcal{E}>-\infty\}.$
More concretely, we could even assume that $\psi$ above is smooth
(by the approximation Lemma \ref{lem:(regularization).-For-any} ). \end{rem}
\begin{example}
The function $\phi(x)=0$ is \emph{not} in the space $\mathcal{E}_{P}^{1}(\R^{n}).$
Indeed, its Legendre transform is identically equal to infinity on
the complement of $0$ and hence $\phi^{*}$ is not in $L^{1}(P,dp).$
This also shows that the formula \ref{eq:def of energy as mixed in global conv}
is not valid in general on the complement of $\mathcal{E}_{P}(\R^{n}).$\end{example}
\begin{rem}
\label{rem:weighted energy classes}The reason to use the notation
$\mathcal{E}_{P}(\R^{n})$ for the space of all $\phi$ with full
Monge-Ampère mass is that, just as in the Kähler geometry setting
\cite{begz}, $\phi\in\mathcal{E}_{P}(\R^{n})$ iff $\phi$ has finite\emph{
$\chi-$weighted energy} for some convex positive weight $\chi,$
i.e. $\phi$ is in the class $\mathcal{E}_{P}^{\chi}(\R^{n})$ defined
as in \ref{eq:energie as inf}, but with $\mathcal{E}(\psi)$ replaced
by 
\begin{equation}
\int_{0}^{1}\chi(\phi-\phi_{0})MA(\phi_{0}(1-t)+t\phi)dt\label{eq:def of weighted energy}
\end{equation}

\end{rem}

\subsection{The projection $Pr$ on the convexification}

Fix an element $\phi$ in $\mathcal{P}(\R^{n})$ and a bounded continuous
function $v,$ i.e. $v\in\mathcal{C}_{b}^{0}(\R^{n}).$ If $\phi$
is in the ``boundary'' of $\mathcal{P}(\R^{n}),$ i.e. $\phi$ is
not strictly convex, then some perturbation $\phi+v$ will leave the
space $\mathcal{P}(\R^{n}).$ As a remedy for this we define the projection
operator $Pr$ from $\{\phi\}+\mathcal{C}_{b}^{0}(\R^{n})$ onto $\mathcal{P}(\R^{n})$
by 
\[
Pr(\phi+v)(x):=\sup_{\psi\in\mathcal{P}(\R^{n})}\{\psi(x):\,\,\,\psi\leq\phi+v\}
\]
Noting that $Pr(\phi+v)^{*}=(\phi+v)^{*}$ we could also use 
\[
Pr(\phi+v)=(\phi+v)^{**}
\]
as the definition of $Pr(\phi+v).$
\begin{prop}
\label{pro:diff of composed energy convex}Fix an element $\phi$
in $\mathcal{E}_{P}^{1}(\R^{n}).$ Then the functional 
\[
\mathcal{C}_{b}^{0}(\R^{n}):\,\,\, v\mapsto(\mathcal{E}\circ Pr)(\phi+v)(=\int_{P}(\phi+v)^{*}dp)
\]
 is Gateaux differentiable and the differential at $v$ is given by
$MA(Pr(\phi+v)).$ \end{prop}
\begin{proof}
Let us first consider the differential of the functional at $v=0.$
Observe that, for $p$ fixed, $t\mapsto(\phi+tv)^{*})(p)$ is convex
and hence its right and left derivatives $d_{\pm}(p)$ exist everywhere,
defining functions which are in $L_{loc}^{\infty}(P)$ and $d_{+}(p)=d_{-}(p)$
for almost every $p$ (by the first point of Prop \ref{pro:var prop of leg}).
Moreover, by convexity $d_{\pm}(p)$ are both defined as monotone
limit. Hence it follows from the Lebesgue monotone convergence theorem
that 
\[
d(\int_{P}(\phi+tv)^{*}dp)/dt_{t=0^{\pm}}=\int_{P}d_{\pm}(p)dp
\]
 and since $d_{+}(p)=d_{-}(p)$ a.e. wrt $dp$ this gives the desired
differentiability property of the functional $\mapsto(\mathcal{E}\circ Pr).$
Moreover, by Lemma \ref{pro:var prop of leg} we have if $u:=\phi^{*}$
that 
\[
d(\int_{P}(\phi+tv)^{*}dp)/dt_{t=0^{\pm}}=-\int_{P}v(du_{p})dp
\]
 and the proof if concluded by invoking the formula in the third point
of Lemma \ref{pro:var prop of leg}. Finally, to obtain the differential
at any $v$ we note that the previous argument gives that 
\[
d(\int_{P}(\phi+tv)^{*}dp)/dt_{t=1^{\pm}}=-\int_{P}v(d(\phi+v)^{*}(p))dp.
\]
But since $(\phi+v)^{*}=\psi^{*}$ for $\psi=Pr(\phi+v)$ we can now
apply the formula in the third point of Lemma \ref{pro:var prop of leg}
to $\psi.$
\end{proof}

\subsection{Geodesics }

Given two strictly convex and smooth functions $\phi_{0}$ and $\phi_{1}$
in $\mathcal{P}_{+}(\R^{n})$ the geodesic $\phi_{t}$ in $\mathcal{P}_{+}(\R^{n})$
from $\phi_{0}$ to $\phi_{1}$ is defined as the the map $[0,1]\rightarrow\mathcal{P}_{+}(\R^{n})$
defined by 
\[
\phi_{t}=((1-t)\phi_{0}^{*}+t\phi_{1}^{*})^{*}
\]
 i.e. under the Legendre transform (for $t$ fixed) $\phi_{t}$ corresponds
to an affine curve in $\mathcal{H}(P).$ In particular, $\phi_{t}(x)$
is smooth and convex in $(t,x).$

\subsection{\label{sub:Variational-principles-and}Variational principles and
a coercive Moser-Trudinger type inequality}

Given $(P,g)$ we consider the following Moser-Trudinger type functional
on $\mathcal{P}_{+}(\R^{n}):$ 
\[
\mathcal{G}_{(P,g)}(\phi):=\frac{1}{V(P,g)}\mathcal{E}_{g}(\phi,\phi_{P})-\mathcal{I}(\phi),\,\,\,\,\,\mathcal{I}(\phi):=-\log\int e^{-\phi}dx
\]
To simplify the notation we will set $g=1$ and write $\mathcal{G}_{(P,g)}=\mathcal{G},$
but the proofs in the general case are the same. Note that $\mathcal{G}$
is a well-defined as a functional on $\mathcal{E}_{P}^{1}(\R^{n})$
taking values in $]-\infty,\infty].$ The normalization of the energy
term above is made so that $\mathcal{G}(\phi+c)=\mathcal{G}(\phi).$
Moreover, we have the following simple, but crucial 
\begin{lem}
\label{lem:inv of m-t functional}The functional $\mathcal{G}$ is
invariant under the action of $\R^{n}$ by translations, $\phi\mapsto\phi_{a},$
iff $0$ is the barycenter of $P.$ Moreover, if $\mathcal{G}(\phi)$
is bounded from above, then $0$ is the barycenter of $P.$\end{lem}
\begin{proof}
Since the volumes form $dx$ is invariant under translations so is
the functional $\mathcal{I}$ and hence $\mathcal{G}$ is invariant
under translations iff the function $\mathcal{E}(\cdot,\phi_{P})$
is. The proof of the first statement is thus concluded by noting that
\[
\mathcal{E}(\phi_{a},\phi_{P})-\mathcal{E}(\phi,\phi_{P})=-\int_{P}((\phi_{a})^{*}-\phi^{*})dp=\int_{P}\left\langle a,p\right\rangle dp,
\]
 where we have used \ref{eq:transl under legendre} in the last equality.
Finally, if $0$ is not the barycenter of $P$ we can take a curve
$\phi_{t}$ in $\mathcal{P}_{+}(\R^{n})$ such that $\phi_{t}^{*}(0)=tp_{i_{0}}$
where $p$ is the barycenter of $P$ and $p_{i_{0}}\neq0.$ Then it
follows as above that 
\[
\mathcal{G}(\phi_{t})=\log\int e^{-\phi_{P}}dx+\mathcal{E}(\phi_{t},\phi_{P})=\log\int e^{-\phi_{P}}dx-tp_{i_{0}}
\]
 which is unbounded from above when either $t\rightarrow\infty$ or
$t\rightarrow-\infty.$
\end{proof}
Next, we have the following key concavity property:
\begin{prop}
\label{prop:The-functional-G is conc}The functional $\mathcal{G}$
is concave along geodesics in $\mathcal{P}_{+}(\R^{n})$ and strictly
concave modulo the action of $\R^{n}$ by translations. In particular,
any solution $\phi$ as in the statement of Theorem \ref{thm:existe real ma intr}
maximizes $\mathcal{G}$ on the space $\mathcal{P}_{+}(\R^{n}).$\end{prop}
\begin{proof}
Let $\phi_{t}$ be a geodesic in $\mathcal{P}_{+}(\R^{n}).$ By the
Prekopa-Leindler inequality 
\begin{equation}
t\mapsto-\log\int_{\R^{n}}e^{-\phi_{t}}dx\label{eq:log int functional in convex case}
\end{equation}
 is convex, since $\phi_{t}$ is convex in $(x,t)$ (see \cite{b-b2}
for complex geometric generalizations of this inequality). Moreover,
by Prop \ref{pro:prop of energy in global conv} $\mathcal{E}(\phi_{t},\phi_{0})$
is affine wrt $t$ and hence $\mathcal{G}(\phi_{t})$ is convex as
desired. The last statement follows from the equality case for the
Prekopa-Leindler inequality giving that if the function in \ref{eq:log int functional in convex case}
is affine in $t$ then $\phi_{t}(x)=\phi(x+ta)$ for some vector $a.$
The final statement now follows by connecting a given element $\phi$
(which by approximation may be assumed smooth) and the solution $\phi_{0}$
with a geodesic and using that the differential of $\mathcal{G}$
vanishes at $\phi_{0}.$ 
\end{proof}
We are now in the position to prove one of the main results in the
present paper:
\begin{thm}
\label{thm:coerciv k-e}Let $P$ be a convex body containing $0$
in its interior. For any $\delta>0$ there is a constant $C_{\delta}$
such that 
\[
\mathcal{G}(\phi)\leq(1-\delta)\mathcal{E}(\phi,\phi_{P})+C_{\delta}
\]
for any normalized $\phi\in\mathcal{P}_{+}(\R^{n}),$ i.e. $\phi(0)=0$
and $\phi\geq0.$ Moreover, \textup{$\mathcal{\mathcal{G}}$ is bounded
from above on} $\mathcal{P}_{+}(\R^{n})$ iff it is invariant under
the action of $\R^{n}$ by translations iff $0$ is the barycenter
of $P.$ \end{thm}
\begin{proof}
\emph{Step one} \emph{(a crude M-T type inequality):} there is a positive
constant $C$ such that 
\[
\log\int e^{-\phi}dx\leq-C\mathcal{E}(\phi,\phi_{P})\text{}+C
\]
for any $\phi\in\mathcal{P}_{+}(\R^{n})$ such that $\phi(0)=0$ (or
equivalently such that $\sup(\phi-\phi_{P})=0).$ 

We first fix a reference $\phi_{0}$ in $\mathcal{P}_{+}(\R^{n})$
such that $\int e^{-\phi_{0}}dx=1.$ Given $\phi$ in $\mathcal{P}_{+}(\R^{n})$
there is a \emph{geodesic} $\phi_{t}$ in $\mathcal{P}_{+}(\R^{n})$
starting at $\phi_{0}$ such that $\phi_{1}=\phi.$ By the previous
proposition $\mathcal{G}(\phi_{t})$ is concave giving $\mathcal{G}(\phi_{1})\leq\mathcal{G}(\phi_{0})+d\mathcal{G}(\phi_{t})/dt_{t=0},$
where 
\[
d\mathcal{G}(\phi_{t})/dt_{t=0}=\int_{\R^{n}}(-d\phi_{t}/dt_{t=0})(e^{-\phi_{0}}dx-MA(u_{0}))
\]
By the invariance of $\mathcal{G}$ under $\phi\mapsto\phi+c$ we
may assume that $\sup(\phi_{1}-\phi_{0})=0$ and hence by the convexity
of $\phi_{t}$ wrt $t$ we have $d\phi_{t}/dt_{t=0}\leq0.$ Next we
note that we may take the fixed reference $\phi_{0}$ so that 
\begin{equation}
e^{-\phi_{0}}\leq C\cdot MA(\phi_{0})\label{eq:u 0 with ricci controle}
\end{equation}
 for some constant $C.$ Accepting, for the moment, the existence
of such a $\phi_{0}$ we deduce that there is a constant $C$ such
that 
\[
\mathcal{G}(\phi)-\mathcal{G}(\phi_{0})\leq C\int-d\phi_{t}/dt_{t=0})MA(\phi_{0})=-C\mathcal{E}(\phi,\phi_{0}),
\]
using \ref{eq:first deriv of energy in global conv} and that $\mathcal{E}(\phi_{t},\phi_{0})$
is affine wrt $t.$ Finally, since $\mathcal{E}(\phi,\phi_{0})=\mathcal{E}(\phi,\phi_{P})+\mathcal{E}(\phi_{P},\phi_{0})$
this concludes the proof up to the existence of $\phi_{0}\in\mathcal{P}_{+}(\R^{n})$
satisfying \ref{eq:u 0 with ricci controle}. An explicit choice of
$\phi_{0}$ may be obtained by setting $\phi_{0}=\phi_{P,1}$ as in
\ref{eq:phi P k smooth}. The property \ref{eq:u 0 with ricci controle}
can then be checked by straightforward, but somewhat tedious calculations.
Alternatively, we may set $\phi_{0}=\phi$ for a solution $\phi\in\mathcal{P}_{+}(\R^{n})$
of the inhomogeneous Monge-Ampère equation $MA(\phi)=e^{-\psi}dx,$
where $\psi$ is any given element in $\mathcal{P}_{+}(\R^{n}),$
e.g. $\psi=\phi_{P}$ or even in $\mathcal{E}_{P}(\R^{n})$ (see Cor
\ref{cor:inhomo m-a for e to - psi}). Since, $\phi\in\mathcal{P}_{+}(\R^{n})$
we have $\psi\leq\phi+A$ and hence the property \ref{eq:u 0 with ricci controle}
follows with $C=e^{-A}.$ 

\emph{Step two:} \emph{refinement by scaling}. Let now $\phi$ be
a normalized function in $\mathcal{P}_{+}(\R^{n})$ and in particular
$\phi\geq0.$ We will improve the inequality in the previous step
by a scaling argument. To this end fix $t\in]0,1[.$ Since $\phi\geq t\phi$
we have that 
\[
\log\int e^{-\phi}dx\leq\log\int e^{-t\phi}dx=\log\int e^{-\phi_{t}}d(x/t)
\]
where $\phi_{t}(x):=t\phi(x/t).$ Note that $\phi_{t}=(tu)^{*},$
where $u=\phi^{*}$ and in particular $\phi_{t}\in\mathcal{P}_{+}(\R^{n})_{0}.$
Hence, applying the previous step gives 
\[
\log\int e^{-\phi_{t}}dx\leq-C\mathcal{E}(\phi_{t},\phi_{P})\text{}+C=Ct\int_{P}u+C
\]
All in all this means that 
\[
\log\int e^{-\phi}dx\leq Ct\int_{P}udp+C+n\log t
\]
and hence setting $\delta=t/C$ concludes the proof of the first statement
of the theorem. 

Finally, if $0$ is the barycenter of $P$ then we have, by the previous
lemma, and the definition of the normalization $\tilde{\phi}$ that
\[
\mathcal{G}(\phi)=\mathcal{G}(\tilde{\phi})\leq0+C
\]
for any $\mathcal{P}_{+}(\R^{n}).$ The converse was proved in the
previous lemma.\end{proof}
\begin{rem}
The scaling argument in the previous proof is somewhat analogous to
a scaling argument used by Donaldson \cite{d0} for the Mabuchi functional
on a toric manifold.
\end{rem}
Interestingly, the boundedness of $\mathcal{G}$ under the moment
condition on $(P,g),$ say for $g=1,$ is also a consequence of the
functional form of the\emph{ }Santaló\emph{ }inequality \cite{a-k-m}.
Indeed, the latter inequality says that, for any convex function $\phi(x)$
in $\R^{n}$ the following inequality holds after perhaps replacing
$\phi$ by $\phi_{a}$ for some $a\in\R^{n}:$ $\int e^{-\phi(x)}dx\int e^{-\phi^{*}(p)}dp\leq(2\pi)^{n}$
and hence the boundedness of $\mathcal{G}$ follows from Jensen's
inequality. However, for the proof of Theorem \ref{thm:existe real ma intr}
we do need the stronger coercivity inequality obtained in the previous
theorem.

\subsection{Proof of Theorem \ref{thm:existe real ma intr}}

\emph{Step 1: }the sup of $\mathcal{G}$ is attained, i.e. there exists
a maximizer $\phi$ of finite energy.
\begin{proof}
By Prop \ref{pro:prop of energy in global conv} $\mathcal{E}$ is
upper semi-continuous (usc) on the space $\mathcal{P}(\R^{n}).$ .
By the invariance under translations (see Lemma \ref{lem:inv of m-t functional})
we can take a sequence of normalized functions $\phi_{j}$ in $\mathcal{P}_{+}(\R^{n})$
such that $\mathcal{G}(\phi_{j})\rightarrow\sup\mathcal{G}.$ But
then it follows from the coercivity inequality in Theorem \ref{thm:coerciv k-e}
that there is constant $C$ such that $\mathcal{E}(\phi_{j})\geq-C.$
By the compactness of $\mathcal{P}_{0}(\R^{n}).$ we may, after perhaps
passing to a subsequence, assume that $\phi_{j}\rightarrow\phi$ in
$\mathcal{P}(\R^{n}),$ where $\mathcal{E}(\phi)\geq-C,$ since $\mathcal{E}(\phi)$
is usc. The proof of step 1 is now concluded by noting that the functional
$-\mathcal{I}$ is usc along $\phi_{j},$ or more precisely that 
\begin{equation}
\int e^{-\phi_{j}}dx=\lim_{j\rightarrow\infty}\int e^{-\phi_{j}}dx\label{eq:conv of exp integral in proof}
\end{equation}
Indeed, since $\mathcal{E}(\phi_{j})\geq-C$ it follows from the coercivity
inequality in Theorem \ref{thm:coerciv k-e} (and a simple scaling
argument) that 
\[
\int e^{-p(\phi_{j}-\phi_{P})}\mu_{P}\leq C_{p},\,\,\,\mu_{P}:=e^{-\phi_{P}}dx
\]
 for some positive number $p>1.$ But since $\phi_{j}\rightarrow\phi$
uniformly on any compact set the desired upper semi-continuity of
$-\mathcal{I}$ then follows from Hölder's inequality. Indeed, integrating
the lhs in \ref{eq:conv of exp integral in proof} over the complement
of a ball $B_{R}$ of radius $R$ gives 
\[
\int_{\R^{n}-B_{R}}e^{-(\phi_{j}-\phi_{P})}\mu_{P}\leq(\int_{\R^{n}-B_{R}}e^{-p(\phi_{j}-\phi_{P})}\mu_{P})^{1/p}\mu_{P}(\R^{n}-B_{R})
\]
 and since $\mu_{P}$ has finite mass on $\R^{n}$ the rhs above can
be made arbitrary small by taking $R$ sufficiently large.

Step 2: the maximizer $\phi$satisfies the equation $MA(\phi)=e^{-\phi}dx/\int e^{-\phi}dx$
in the weak sense

Fix a smooth function $v$ of compact support and and consider the
following functional on the real line: $t\mapsto f(t):=\frac{1}{V}\mathcal{E}(Pr(\phi+tv)-\mathcal{I}(\phi).$
Since $\mathcal{I}$ is increasing we have that $f(t)\leq\mathcal{G}(Pr(\phi+tv))\leq\mathcal{G}(\phi)$
(since $\phi$ is a maximizer), i.e. $f(t)\leq f(0).$ But by Prop
\ref{pro:diff of composed energy convex} $f(t)$ is differentiable
and hence $df/dt=0$ at $t=0,$ which by Prop \ref{pro:diff of composed energy convex}
gives the desired equation, since $Pr(\phi)=\phi.$ 

Step 3: local regularity of solutions

We will give the argument for the more general case when $\phi$ is
a finite energy solution of $MA_{g}(\phi)=Ce^{-F_{1}(\phi)}$ where
$g(p)=e^{-F_{2}(p)}$ for $F_{1}$ a continuous function on $\R^{n}$
and $F_{2}$ a bounded function on $P$ (see Remark \ref{rem:g-ma}
for the definition of $MA_{g}).$ By Lemma \ref{lem:full ma is proper}
$\phi$ is proper and hence the sublevel sets $\Omega_{R}:=\{\phi<R\}$
are bounded convex domains exhausting $\R^{n}.$ Fixing $R,$ writing
$\Omega:=\Omega_{R}$ and replacing $\phi$ with $\phi-R$ we then
have that $\phi=0$ on $\partial\Omega$ and $1/Cdx\leq MA(\phi)\leq Cdx$
on $\Omega$ for some positive constant $C.$ Hence, it follows from
the first point in Theorem \ref{thm:(Caffarelli-).-Assume} below
that $\phi$ is in the Hölder class $\mathcal{C}_{loc}^{1,\alpha}$
for some $\alpha>1.$ In particular, the gradient $d\phi$ is a single-valued
continuous function and hence $MA(\phi)=f$ for a continuous function
$f$ such that $1/C'\leq f\leq C'$ in $\Omega.$ Applying the second
point in Theorem \ref{thm:(Caffarelli-).-Assume} thus shows that
$\phi$ is in the Sobolev space $W_{loc}^{2,p}(\Omega)$ for any $p>1.$
Finally, by general Evans-Krylov theory for non-linear PDEs we deduce
that $\phi\in\mathcal{C}^{\infty}(\R^{n}).$ 

Step 4: global regularity

Applying Cor \ref{cor:inhomo m-a for e to - psi} below with $\psi=\phi$
shows that $\phi-\phi_{P}$ is globally bounded on $\R^{n}.$ More
precisely, the Legendre transform of $\phi$ is Hölder continuous
up to the boundary of $P$ for any Hölder exponent $\gamma<1.$

Uniqueness: Let $\phi_{0}$ and $\phi_{1}$ be two solutions and let
$\phi_{t}$ be the geodesic segment connecting them. By the strict
concavity in Prop \ref{prop:The-functional-G is conc} $\phi_{t}=\phi_{0}(x+ta)$
for some vector $a$ which concludes the proof. 
\end{proof}

\subsection{\label{sub:The-invariant-R of convex bod}The invariant $R$ of a
convex body}

Let $P$ be a convex body containing $0$ and define the following
invariant $R_{P}\in]0,1],$ which is a measure of the failure of $P$
having the property that its barycenter $b$ coincides with $0:$
\begin{equation}
R_{P}:=\frac{\left\Vert q\right\Vert }{\left\Vert q-b\right\Vert },\label{eq:inv R of conv bod}
\end{equation}
 where $q$ is the point in $\partial P$ where the line segment starting
at $b$ and passing through $0$ meets $\partial P.$ In the case
when $P$ is the canonical polytope associated to a smooth Fano variety,
the invariant $R_{P}$ was introduced in \cite{li1}, where it was
shown to coincide with another seemingly analytical invariant. A slight
modification of the proof of Theorem \ref{thm:existe real ma intr}
gives the following theorem, which - when translated to toric geometry
- generalizes the main result of \cite{li1} (see section \ref{sub:The-invariant-R of log Fano})
:
\begin{thm}
\label{thm:R for ma}Let $P$ be a convex body containing $0$ in
its interior and fix an element $\phi_{0}\in\mathcal{P}_{+}(\R^{n}).$
Then the invariant $R_{P}$ coincides with the following two numbers,
defined as the sup over all  $r\in[0,1[$ such that 
\begin{itemize}
\item there is a solution $\phi\in\mathcal{P}_{+}(\R^{n})$ to the equation
\begin{equation}
\det(\frac{\partial^{2}\phi}{\partial x_{i}\partial x_{j}})=e^{-r\phi}e^{-(1-r)\phi_{0}}\label{eq:eq in statement of thm r for ma}
\end{equation}

\item The following functional on $\mathcal{P}_{+}(\R^{n})$ is bounded
from above:
\end{itemize}
\[
G_{r,\phi_{0}}(\phi):=\frac{1}{V(P)}\mathcal{E}(\phi,\phi_{P})+\frac{1}{r}\log\int e^{-r\phi}e^{-(1-r)\phi_{0}}dx
\]
\end{thm}
\begin{proof}
First observe that $R_{P}$ is the sup over all $r\in[0,1]$ such
that 
\begin{equation}
(1-r)\phi_{P}(x)+r\left\langle b,x\right\rangle \geq0\label{eq:bound in proof of R for ma}
\end{equation}
Accepting this for the moment we will start by showing that if $r$
is a positive number such that \ref{eq:bound in proof of R for ma}
holds for $r+\delta$, for some $\delta>0,$ then $G_{r+\delta}$
is bounded from above. By assumption $-(1-r)\phi_{P}(x)\leq r\left\langle b,x\right\rangle $
and hence 
\[
G_{r,\phi_{P}}(\phi)\leq\frac{1}{V(P)}\mathcal{E}(\phi,\phi_{P})+\frac{1}{r}\log\int e^{-r\phi}e^{r\left\langle b,x\right\rangle }dx
\]
Applying the boundedness statement in Theorem \ref{thm:coerciv k-e}
to the translated and scaled convex body $rP-\{b\}$ (which has its
barycenter at $0)$ thus shows that $G_{r,\phi_{P}}(\phi)\leq C_{r}$
and, by the same argument, $G_{r+\delta,\phi_{P}}(\phi)\leq C_{r+\delta}.$
Moreover, since $\phi_{0}-\phi_{P}$ is bounded the corresponding
inequalities also hold when $\phi_{P}$ is replaced by $\phi_{0}.$

Next, we show that the inequality the bound $G_{r+\delta,\phi_{0}}(\phi)\leq C_{r+\delta}$
implies the existence of a solution to equation \ref{eq:eq in statement of thm r for ma}
for the parameter $r.$ First, by a simple scaling argument, it follows
that the functional $G_{r,\phi_{0}}$ is coercive, i.e. there exists
positive numbers $\delta$ and $C_{\delta}$ such that 
\[
G_{r,\phi_{0}}(\phi)\leq\delta\mathcal{E}(\phi,\phi_{P})+C_{\delta}
\]
 for any sup-normalized $\phi.$ But then it follows, exactly as in
the proof of Theorem \ref{thm:existe real ma intr}, that any sup-normalized
maximizing sequence of $G_{r,\phi_{0}}$ converges to a solution to
the equation \ref{eq:eq in statement of thm r for ma}. 

Conversely, let $r$ be such that there is solution to the equation
\ref{eq:eq in statement of thm r for ma}. It then follows, just like
in the proof of Prop \ref{prop:The-functional-G is conc} (now using
the Prekopa inequality for the convex function $(t,x)\mapsto$ $r\phi_{t}(x)+(1-r)\phi_{0}(x))$
that $G_{r,\phi_{P}}$ (and hence also $G_{r,\phi_{0}})$ is bounded
from above. Now fix a point $a\in\R^{n}.$ By definition $\phi_{P}(x+a)\leq\phi_{P}(x)+\phi_{P}(a)$
and hence, for any $\phi\in\mathcal{P}_{+}(\R^{n})$ we have 
\[
-\frac{(1-r)}{r}\phi_{P}(a)+\frac{1}{r}\log\int e^{-r\phi}e^{-(1-r)\phi_{P}}dx\leq\frac{1}{r}\log\int e^{-r\phi(x)}e^{-(1-r)\phi_{P}(x+a)}dx
\]
Making the change of variables $x\mapsto x+a$ in the latter integral
and applying the boundedness of $G_{r,\phi_{P}}$ thus gives that
\[
-\frac{(1-r)}{r}\phi_{P}(a)+\frac{1}{r}\log\int e^{-r\phi}e^{-(1-r)\phi_{P}}dx\leq-\frac{1}{V(P)}\int(\phi_{-a})^{*}dp+C
\]
But, using \ref{eq:transl under legendre} and rearranging we deduce
that 
\[
G_{r,\phi_{P}}(\phi)-C\leq(1-r)\phi_{P}(a)+r\left\langle b,a\right\rangle 
\]
Assume for a contradiction that the rhs above is negative $(=-\delta).$
Then replacing $a$ with $ta$ for $t$ a positive number $t$ shows
that $G_{r,\phi_{P}}(\phi)\leq-\delta t+C$ and hence letting $t\rightarrow\infty$
yields a contradiction.

Finally, we come back to the first claim concerning the inequality
\ref{eq:bound in proof of R for ma}. First observe that the inequality
holds for $r\leq R_{P}.$ Indeed, by definition, the lhs in \ref{eq:bound in proof of R for ma}
is equal to the sup over $\left\langle (1-r)p+rb,x\right\rangle ,$
as $p$ ranges over all points in $P.$ In particular, for $r=R_{P}$
and $p=q$ we have by definition $(1-r)p+rb=0$ and hence \ref{eq:bound in proof of R for ma}
holds for $r\leq R_{P}.$ Conversely, suppose that $r>R_{p}$ and
let $a$ be the vector defining a supporting hyperplane of $P$ at
$q,$ i.e. $\left\langle p,a\right\rangle \leq\left\langle q,a\right\rangle $
for any $p\in P.$ Hence, the lhs in \ref{eq:bound in proof of R for ma}
for $x=a$ is bounded from above by $\left\langle (1-r)q+rb,a\right\rangle :=f(r).$
Finally, by assumption $f(r)=0$ at $r=R_{P}$ and $df(r)/dr=\left\langle b-q,a\right\rangle <0,$
since $b$ is in the interior of $P$ and hence $f(r)<0$ for any
$r>R_{P}.$ 
\end{proof}

\subsection{\label{sub:The-inhomogenous-Monge-Amp=0000E8re}The inhomogeneous
Monge-Ampère equation}

In this section we will establish some local and global regularity
properties for the inhomogeneous Monge-Ampère equation used above. 
\begin{thm}
\label{thm:inhomo}Let $P$ be a convex body in $\R^{n}$ and let
$\mu$ be measure on $\R^{n}$ of total mass $V(P).$ Then there exists
a unique (mod $\R)$ convex function $\phi$ on $\R^{n}$ such that
\begin{equation}
MA(\phi)=\mu\label{eq:inhomo ma}
\end{equation}
with $\phi\in\mathcal{E}_{P}(\R^{n}),$ i.e. the closure of the subgradient
image is $P:$ 
\[
\overline{d\phi(\R^{n})}=P
\]

\begin{itemize}
\item If moreover \textup{
\[
\int_{\R^{n}}|x|^{q}\mu<\infty
\]
 for some number $q>n$ then $\phi-\phi_{P}$ is bounded and if the
finiteness holds for any $q>0$ then }the Legendre transform $\phi^{*}$
of $\phi$ is Hölder continuous up to the boundary of $P$ for any
Hölder exponent in $[0,1[.$ 
\item If $\mu=fdx$ for $f$ smooth and strictly positive a solution $\phi$
is unique modulo constants and smooth. In particular, the gradient
$d\phi$ then maps $\R^{n}$ diffeomorphically onto the image of the
interior of $P.$ 
\end{itemize}
\end{thm}
\begin{proof}
Uniqueness modulo constant follows from a comparison principle argument
as in section 16.2 in \cite{bak}. 

\emph{Existence: }First observe that it will be enough to prove the
result when $\mu$ has finite first moments, i.e. 
\begin{equation}
\int\mu|x|<\infty.\label{eq:moment cond}
\end{equation}
 Indeed, any measure $\mu$ can be written as $\mu=f\nu$ where $\nu$
has finite first moments (e.g. take $f=(1+|x|)$ and $(1+|x|)^{-1}\mu)$
and we can solve the $MA(\phi_{i})=\mu_{i}$ where $\mu_{i}=\max(f,i)\nu$
with $\phi_{i}\in\mathcal{P}(\R^{n})_{0}.$ Finally, by compactness
we have, after perhaps passing to a subsequence, that $\phi_{i}\rightarrow\phi\in\mathcal{P}(\R^{n})_{0}$
and by the local continuity of $MA$ acting on $\mathcal{P}(\R^{n})_{0}$
(see section \ref{sub:Functions-with-full}) $MA(\phi)=\mu$ 

Assume now that $\mu$ has finite first moments. Since, the gradient
image of any $\phi$ in $\mathcal{P}(\R^{n})_{0}$ is uniformly bounded
there is a constant $C$ such that $|\phi(x)|\leq C|x|$ and hence
the functional 
\[
I_{\mu}(\phi):=\int\phi\mu
\]
 is finite on $\mathcal{P}(\R^{n})_{0}.$ In fact, it is even \emph{continuous.
}Indeed, if $\phi_{j}\rightarrow\phi$ in $\mathcal{P}(\R^{n})_{0}$
then the convergence is uniform on any large ball $B_{R}$ of radius
$R,$ so that the desired continuity is obtained by decomposing $\mu=1_{B_{R}}\mu+1_{\R^{n}-B_{R}}$
for a large ball $B_{R}$ of radius $R$ and using the uniform bound
$|\phi(x)|\leq C|x|$ on $\R^{n}-B_{R}$ together with \ref{eq:moment cond}
and finally letting $R\rightarrow\infty.$ As a consequence the functional
\[
\mathcal{G}_{\mu}(\phi):=\frac{1}{V}\mathcal{E}(\phi,\phi_{P})-\mathcal{I}_{\mu}(\phi)
\]
 is upper semi-continuous on the compact space $\mathcal{P}(\R^{n})_{0}.$
In particular it has a maximizer $\phi_{\mu}$ of finite energy and
the proof is concluded by noting that $\phi_{\mu}$ satisfies the
equation \ref{eq:inhomo ma}. Indeed, this is shown precisely as in
the end of the proof of Theorem \ref{thm:existe real ma intr}, using
the projection operator $Pr.$

\emph{Regularity:} this is proved exactly as in the proof of Theorem
\ref{thm:existe real ma intr}, using Caffarelli's interior regularity
results (see below). For the global $C^{0}-$bound and the Hölder
regularity see Prop \ref{pro:sob type ineq}.
\end{proof}
It should be emphasized that the existence of a (weak) solution $\phi$
for a general measure $\mu$ in the previous theorem is essentially
well-known (for example, this is shown in \cite{bak} when $\mu$
has an $L^{1}-$density). The result is also closely related to the
theory of the problem of \emph{optimal transportation of measures}.
Briefly, the problem, in the original formulation of Monge, is to
find, given two probability measure measure $\mu$ and $\nu$ on $\R^{n},$
a map $T$ such that $T_{*}\mu=\nu$ where $T$ minimizes a certain
cost function $c(x,p).$ In the present setting $\nu=1_{P}gdp$ and
$c(x,p)=\left\langle x,p\right\rangle $ and then $T=d\phi,$ for
$\phi$ the solution in the previous theorem, under suitable regularity
assumptions. Interestingly, the variational principle used in our
proof is equivalent to the duality formula for the minimum of the
Kantorovich problem for optimal transport (see Theorem 6.1.1 in \cite{a-g-s}
and references therein).
\begin{cor}
\label{cor:inhomo m-a for e to - psi}Let $P$ be a convex body in
$\R^{n}.$ For any given $\psi\in\mathcal{E}_{P}(\R^{n})$ there is
unique (mod $\R)$ function $\phi\in\mathcal{P}_{+}(\R^{n})$ such
that 
\[
MA(\phi)=V(P)e^{-\psi}dx/\int e^{-\psi}dx
\]
\end{cor}
\begin{proof}
We may assume that $0$ is contained in the interior of $P.$ By lemma
\ref{lem:full ma is proper} the moment condition in the previous
theorem is satisfied, so that the previous theorem furnishes the desired
solution.\end{proof}
\begin{rem}
\label{rem:integr and lelong in kahler}In the Kähler geometry setting
the analog of the previous corollary is known to hold for a Kähler
class $[\omega]$ on a smooth manifold $X.$ The point is that if
$v\in\mathcal{E}(X,\omega)$ then $v$ has no Lelong numbers and hence
$e^{-pv}\in L^{1}(X,dV)$ for any $p>0,$ so that $v$ is bounded
by Kolodziej's estimate.
\end{rem}

\subsubsection{A global $C_{0}-$estimate}

In the arguments above we used the following 
\begin{prop}
\label{pro:sob type ineq}If the $q:$th moment of $\mu$ is finite
for some $q>n,$ then any solution $\phi\in\mathcal{P}(\R^{n})$ to
equation \ref{eq:inhomo ma} is in $\mathcal{P}_{+}(\R^{n}).$ More
precisely, the following inequality holds for any $\phi\in\mathcal{P}_{0}(\R^{n})$
with full Monge-Ampere mass: 
\[
\left\Vert \phi-\phi_{P}\right\Vert _{\mathcal{C}^{0}(\R^{n})}\leq C_{q}(-\mathcal{E}(\phi,\phi_{P})+(\int_{\R^{n}}MA(\phi)|x|^{q}))^{1/q}
\]
for any $q>n.$ More generally, the Legendre transform $u:=\phi^{*}$
is in the Hölder space $\mathcal{C}^{\gamma}(P)$ for $\gamma=1-n/q$
if the $q$th moments of $MA(\phi)$ are finite.\end{prop}
\begin{proof}
By assumption $P$ is the closure of an open convex domain $D$ and
in particular the boundary of $P$ is Lipschitz (cf. \cite{e-e} Sec.
V.4.1{]}). We will deduce the desired inequality from the Sobolev
inequality on the Lipschitz domain $D$ which says that $W^{q,1}$
injects in $\mathcal{C}^{0}(D)$ in a continuous manner if $q>n$
and 
\[
\left\Vert u\right\Vert _{\mathcal{C}^{0}(D)}\leq C_{q}(\int_{D}|u|dp+(\int_{D}|du|^{q}))^{1/q}
\]
Indeed, setting $f(x):=|x|^{q}$ and taking $\phi\in\mathcal{P}_{0}(\R^{n})$
we let $u:=\phi^{*}$ so that $u\geq0$ (since $\phi(0)=0).$ But
then the inequality to be proved follows immediately from combining
Prop \ref{pro:prop of energy in global conv}, Prop \ref{pro:leg is isometric}
and the last point in Lemma \ref{pro:var prop of leg}. The last claim
follows from the general formulation of the Sobolev embedding theorem
for Hölder spaces.\end{proof}
\begin{rem}
\label{rem:The-moment-condition in kahler}The moment condition on
$\mu$ in the previous proposition may also equivalently formulated
in the following way: $\int\mu f_{\delta}(\phi-\phi_{P})<\infty$
for any $\phi\in\mathcal{P}(\R^{n})$ and some $\delta>0,$ where
$f_{\delta}(x):=x^{n+\delta}.$ In the Kähler geometry setting it
is known, as a consequence of Kolodziej's estimates that a measure
$\mu$ with the stronger integrability property obtained by setting
$f(x)=e^{\delta x}$ for any $\delta>0,$ has a bounded Monge-Ampère
potential $v_{\mu},$ i.e. a bounded solution to $MA(v)=\mu.$ Moreover,
it has been conjectured \cite{d-g-pk-z} that the latter property
is equivalent to $v_{\mu}$ being Hölder continuous. Comparing with
the real setting it seems hence natural to ask if boundedness of $v_{\mu}$
holds for $f(x)=x^{n+\delta}?$ 
\end{rem}

\subsection{Caffarelli's interior regularity results }

Let $\Omega$ be a bounded open convex set in $\R^{n}$ and $f$ a
function on $\Omega.$ Consider the following boundary value problem
for a convex function $\phi$ in $\Omega,$ continuous up to the boundary:
\[
MA(\phi)=fdx\,\,\mbox{in\,}\Omega,\,\,\,\,\phi=0\,\,\mbox{on\,\ensuremath{\partial}\ensuremath{\ensuremath{\Omega}}}
\]

\begin{thm}
\label{thm:(Caffarelli-).-Assume}(Caffarelli \cite{ca1,ca0}). Assume
that $f>0.$ Then any (convex) solution $\phi$ on $\Omega$ of the
previous equation is
\begin{itemize}
\item strictly convex and locally $\mathcal{C}^{1,\alpha}$ for some $\alpha>0$
if there exists a constant $C$ such that $1/C\leq f\leq C.$
\item in the class $W_{loc}^{2,p}$ for any $p>1,$ if $f$ is continuous
\item smooth if $f$ is smooth
\end{itemize}
\end{thm}
\begin{proof}
For the proof of the first point we first recall the following special
case of Cor 2 in\cite{ca1}: if the following holds in the viscosity
sense; 
\[
1/Cdx\leq MA(\phi)\leq Cdx,\,\,\,\phi=0\,\,\mbox{on}\,\partial\ensuremath{\Omega}
\]
in a bounded set $\Omega,$ then $\phi$ is strictly convex in $\Omega.$
Moreover, as pointed out in \cite{ca1} if the previous inequalities
hold weakly (i.e. in the sense of Alexandrov) then they hold in the
viscosity sense (see also Prop 1.3.4 in \cite{gu} where it is assumed
that $f$ is continuous, which anyway will always be the case in this
paper). Hence $\phi$ is strictly convex in our case. But then, as
shown in \cite{ca}, if follows from the previous differential inequalities
that $\phi$ is in fact locally $\mathcal{C}^{1,\alpha}$ for some
$\alpha>0.$ As for the second point it is contained in the main result
in \cite{ca0} and the final point then follows from Evans-Krylov
theory for fully non-linear elliptic operators follows by standard
linear bootstrapping.
\end{proof}

\section{\label{sec:Toric-varieties,-polytopes}Toric log Fano varieties,
polytopes and Kähler-Ricci solitons}

\subsection{Log \label{sub:Log-Fano-varieties}Fano varieties }

Let $X$ be an $n-$dimensional normal compact projective variety.
Recall that a (Weil-) divisor $D$ on $X$ is a formal sum of prime
divisors, i.e. codimension one irreducible subvarieties. As usual
we will often identify divisors up to linear equivalence: $D\sim D'$
if $D-D'$ is principal, i.e. equal to the zero divisor of a rational
function on $X.$ A divisor $D$ is a \emph{Cartier divisor }if it
is locally principal and we can hence identify Cartier divisors on
$X$ with line bundles on $X.$ In case $X$ is regular these two
notions of divisors coincide. We will use additive notation for tensor
products of line bundles on $X.$ 

When $X$ is smooth the \emph{canonical line bundle} $K_{X}$ is defined
as the top-exterior power of the cotangent bundle of $X.$ When $X$
is singular $K_{X}$ is well-defined on the regular locus $X_{reg}$
of $X$ and $K_{X}$ is said to be \emph{$\Q-$Cartier} (also called
a $\Q-$line bundle) if there is a positive number $r$ such that
$rK_{X}$ extends from $X_{reg}$ to a line bundle defined on all
of $X.$ The minimal such integer $r$ is called the \emph{Gorenstein
index} of $X$ and $X$ is said to be Gorenstein if $r$ is equal
to one. 

A normal variety $X$ is said to be \emph{Fano }if $-K_{X}$ is \emph{$\Q-$Cartier}
and ample (in the literature such a variety is sometimes said to be
a $\Q-$Fano variety). When $X$ is toric any Fano variety has\emph{
log-terminal} singularities (also called \emph{Kawamata log-terminal,
}or\emph{ klt} for short); see Remark \ref{rem:toric is klt}. Such
singularities play a key role in the Minimal Model Program (MMP);
see \cite{bbegz} and references therein. From the algebro-geometric
point of view klt singularities are defined in terms of discrepancies
on resolutions of $X,$ but there is also a direct analytical definition
on $X$ that is the one that we will use here (see below). 

Similarly, if $(X,\Delta)$ is a\emph{ log pair }in the sense of MMP,
i.e. $X$ is a normal variety and $\Delta$ is a $\Q-$divisor on
$X$ such that the log-canonical line bundle 
\[
K_{(X,\Delta)}:=K_{X}+\Delta
\]
is \emph{$\Q-$line bundle.} Here will also assume that $\Delta$
has coefficients $<1$ (in particular we do allow negative coefficients).
We will write $X_{0}$ for the complement of $\Delta$ in $X_{reg}.$
A log pair $(X,\Delta)$ is said to be a \emph{log Fano variety} if
$-K_{(X,\Delta)}$ is ample. We will also fix a section $s_{\Delta}$
on $X_{reg}$ whose zero divisor is $r\Delta$ for some integer $r.$
There is also a notion of klt singularities for log pairs $(X,\Delta)$
(see below).

\subsection{\label{sub:Metrics-on-line}Metrics on line bundles, $\omega-$psh
functions and the klt condition}

Given a holomorphic $L\rightarrow X$ we let $\mathcal{H}(X,L)$ be
the space of all (possibly singular) metrics on $L$ with positive
curvature current and denote by $\mathcal{H}_{b}(X,L)$ the subspace
consisting of the locally bounded metrics (see below). We will use
additive notation for (Hermitian) metrics on $L.$ This means that
a metric $\left\Vert \cdot\right\Vert $ on $L$ is represented by
a collection of local functions $\phi(:=\{\phi_{U}\})$ defined as
follows: given a local generator $s$ of $L$ on an open subset $U\subset X$
we define $\phi_{U}$ by the relation 
\[
\left\Vert s\right\Vert ^{2}=e^{-\phi_{U}},
\]
 where $\phi_{U}$ is upper semi-continuous (usc). It will convenient
to identify the additive object $\phi$ with the metric it represents.
Of course, $\phi_{U}$ depends on $s$ but the curvature current 
\[
dd^{c}\phi:=\frac{i}{2\pi}\partial\bar{\partial}\phi_{U}
\]
 is globally well-defined on $X$ and represents the first Chern class
$c_{1}(L),$ which with our normalizations lies in the integer lattice
of $H^{2}(X,\R).$ By definition the metric $\phi$ is \emph{smooth}
if $\phi_{U}$ can be chosen smooth, i.e. it is the restriction of
a smooth function for some local embedding $U\rightarrow\C^{m}.$
When $X$ is smooth a smooth metric $\phi$ is said to be\emph{ strictly
positively curved} if $dd^{c}\phi>0$ (as a $(1,1)-$form) and for
a general variety $X$ the metric $\phi$ is said to be smooth if
it is locally the restriction of a positively curved metric on some
ambient space, i.e. $dd^{c}\phi$ is a Kähler form on $X.$ Continuous
and bounded (also called \emph{locally }bounded) metrics etc are defined
in a similar manner and then $dd^{c}\phi$ is a well-defined positive
\emph{current} on $X.$ Fixing $\phi_{0}\in\mathcal{H}_{b}(X,L)$
and setting $\omega_{0}:=dd^{c}\phi_{0}$ the map $\phi\mapsto v:=\phi-\phi_{0}$
thus gives an isomorphism between the space of all metrics on $L$
with positive current and the space $PSH(X,\omega_{0})$ of all $\omega_{0}-$psh
functions \cite{g-z}, i.e. the space of all usc functions on $X$
such that $dd^{c}v+\omega_{0}\geq0.$

\subsubsection{The measure $\mu_{\phi}$ and the klt condition}

In the particular case when $L=-K_{X}$ a locally bounded metric $\phi$
determines a measure $\mu_{\phi}$ on $X$ defined as follows on $X_{0}$
(and extended by zero to all of $X):$ 
\begin{equation}
\mu_{\phi}=i^{n^{2}}e^{-\phi_{U}}\chi^{1/r}\wedge\bar{\chi}^{1/r}\label{eq:mu of phi}
\end{equation}
where $\chi$ is the local $(n,0)-$form which is dual to a given
local generator $s$ of $-rK_{X},$ i.e. $\chi=s^{-1}$ and $\left\Vert s\right\Vert _{\phi}^{2}=e^{-r\phi_{U}}.$
In particular, if $\chi^{1/r}$ is taken as $dz_{1}\wedge\cdots\wedge dz_{n}$
wrt some local coordinates $z_{i}$ then we will, abusing notation
slightly, write $\mu_{\phi}=1_{X_{reg}}e^{-\phi}dz\wedge d\bar{z}.$
Now the analytical definition of $X$ having klt singularities amounts
to $\mu_{\phi}$ having finite total mass for some (and hence any)
locally bounded metric $\phi.$ It is sometimes convenient to represent
$\mu_{\phi}$ globally as follows: if $s\in H^{0}(X,-rK_{X})$ is
such that $X_{0}$ is contained in the Zariski open set $Y:=\{s\neq0\}$
we can write 
\[
\mu_{\phi}=1_{Y}i^{n^{2}}s^{-1/r}\wedge\overline{s^{-1/r}}\left\Vert s\right\Vert _{\phi}^{2/r}
\]
More generally, if $(X,\Delta)$ is a log pair then any metric $\phi$
on $-(K_{X}+\Delta)$ determines a measure $\mu_{\phi}$ defined as
above, but replacing $s$ with a local generator of $-r(K_{X}+\Delta)$
and finally taking the tensor product with $(s_{\Delta}^{r}\otimes\overline{s_{\Delta}^{r}})^{1/r},$
where $s_{\Delta}$ is a section with zero-divisor $\Delta$ on $X_{reg}.$
As before, we can also write $\mu_{\phi}=1_{X_{reg}}e^{-(\phi+\psi_{\Delta})}dz\wedge d\bar{z},$
where $\psi_{\Delta}:=\log(|s_{\Delta}|^{2})/r$ is the singular metric
defined by $\Delta$ on the line bundle $L_{\Delta}\rightarrow X_{reg}.$
The log pair $(X,\Delta)$ is then said to be\emph{ klt} if $\mu_{\phi}$
has finite mass for any locally bounded metric $\phi.$ See \cite{bbegz}
for the equivalence with the algebro-geometric definition.

\subsection{\label{sub:Pluricomplex-energy-and}Pluricomplex energy and the Monge-Ampère
measure}

Let us first recall the definition of the energy type-functional $\mathcal{E}$
on $\mathcal{H}_{b}(X,L)$ for a given ample line bundle $L\rightarrow X$
over a variety $X$ \cite{bbegz}. It depends on the choice of a reference
metric $\phi_{0}$ in $\mathcal{H}_{b}(X,L):$ 
\[
\mathcal{E}(\phi,\phi_{0}):=\frac{1}{(n+1)}\sum_{j=0}^{n}\int_{X}(\phi-\phi_{0})(dd^{c}\phi)^{n-j}\wedge(dd^{c}\phi_{0})^{j}
\]
where the integration pairing $\int_{X}$ refers, as usual, to integration
along the regular locus $X_{0}$ of $X$ and the wedge products are
defined in the usual sense of pluripotential theory a la Bedford-Taylor
(see \cite{bbegz} and references therein). In particular, we will
write 
\[
MA(\phi):=(dd^{c}\phi)^{n}
\]
 for the Monge-Ampère measure of $\phi\in\mathcal{H}_{b}(X,L).$ We
will often omit the explicit dependence of $\mathcal{E}$ on the reference
$\phi_{0}.$ The functional $\mathcal{E}$ is, up to an additive normalizing
constant, uniquely determined by the variational property 
\[
d\mathcal{E}_{|\phi}=MA(\phi)
\]
(viewed as one-forms on $\mathcal{H}_{b}(X,L)).$ As a consequence
the differences $\mathcal{E}(\phi)-\mathcal{E}(\psi)$ are independent
on the choice of fixed reference metric. Now for any arbitrary positively
curved singular metric $\phi$ on $L$ we define, following \cite{bbegz},
\[
\mathcal{E}(\phi)=\inf\left\{ \mathcal{E}(\psi):\,\psi\in\mathcal{H}_{b}(X,L),\,\,\,\psi\geq\phi\right\} 
\]
and let $\mathcal{E}^{1}(X,L):=\{\mathcal{E}>-\infty\}.$ We point
out, even though this fact will not be used here, that the Monge-Ampère
measure can be defined for any $\phi\in\mathcal{H}(X,L)$ in terms
of non-pluripolar products and one then denotes by $\mathcal{E}(X,L)$
the space of all $\phi$ with \emph{full Monge-Ampère mass, }i.e.
$\int_{X}MA(\phi)=c_{1}(L)^{n}.$ In particular, we then have the
relations 
\[
\mathcal{H}_{b}(X,L)\subset\mathcal{E}^{1}(X,L)\subset\mathcal{E}(X,L)\subset\mathcal{H}(X,L)
\]
 (see \cite{begz,bbegz})

\subsection{Kähler-Einstein metrics on log Fano varieties}

Following \cite{bbegz} an element $\phi\in\mathcal{E}^{1}(X,-K_{X})$
is said to be a \emph{(singular) Kähler-Einstein metric} if 
\begin{equation}
(dd^{c}\phi)^{n}=C\mu_{\phi}\label{eq:k-e equation for phi on general fano}
\end{equation}
 for some constant $C,$ where $\mu_{\phi}$ is the canonical measure
\ref{eq:mu of phi} associated to $\phi.$ Similarly, on a (log) Kähler-Einstein
metric $\phi\in\mathcal{E}^{1}(X,-(K_{X}+\Delta))$ on the log Fano
variety on $(X,\Delta)$ is defined by the same equation as above,
using the corresponding measure $\mu_{\phi}.$ By the regularity result
in \cite{bbegz} $\phi$ is in fact automatically smooth on the complement
$X_{0}$ of $\Delta$ in the regular locus of $X$ and continuous
on all of $X.$ In particular, the curvature current $\omega:=dd^{c}\phi$
is a bona fide Kähler-Einstein metric on $X_{0},$ i.e. $\mbox{Ric \ensuremath{\omega=\omega}}$
on $X_{0}$ and globally on $X$ the equation  $\mbox{Ric \ensuremath{\omega=\omega}}+[\Delta]$
holds in the sense of currents.

\subsection{\label{sub:Geodesics,-convexity-and ding}Geodesics, convexity and
Ding type functionals }

As explained in \cite{b-b} any two metrics $\phi_{0}$ and $\phi_{1}$
in $\mathcal{H}_{b}(X,L)$ can be connected by a\emph{ geodesic} $\phi_{t}$
defined as the point-wise sup over all \emph{subgeodesics} $\psi_{t}$
connecting $\phi_{0}$ and $\phi_{1},$ where such a curve of metrics
$\psi_{t}$ on $L$ is defined as follows: complexifying $t$ to take
values in the strip $T:=[0,1]+i\R$ the corresponding metric $\psi(z,t):=\psi_{t}(z)$
is an $i\R-$invariant continuous semi-positively curved metric $\psi$
on the pull-back of $L$ to $X\times T.$ This is a weak analog of
bona fide geodesics defined wrt the Mabuchi metric on the space of
Kähler metrics (see \cite{b-b} and references therein). We recall
the complex version of the Prekopa theorem \cite{bern1,bern2,bbegz}:
If $X$ is a Fano variety and $\psi_{t}$ a subgeodesic in $\mathcal{H}(X,-K_{X}),$
then the functional 
\begin{equation}
t\mapsto-\log\int_{X}\mu_{\psi_{t}}\label{eq:complex prekopa statement}
\end{equation}
is convex in $t.$ We note that, since $\psi_{t}+\psi_{\Delta}$ is
a subgeodesic in $\mathcal{H}(X,-K_{X})$ the result also applies
if $(X,\Delta)$ is a Fano variety and $\psi_{t}$ is a subgeodesic
in $\mathcal{H}(X,-(K_{X}+\Delta).$ Following \cite{bbegz} (up to
a sign difference) we define the Ding type functional. 
\begin{equation}
\mathcal{G}_{(X,\Delta)}:=\mathcal{E}(\phi)+\log\int_{X}\mu_{\psi_{t}}\label{eq:ding type functional for general log pair}
\end{equation}

\begin{prop}
\cite{bern2,bbegz}\label{prop:k-e max ding funtional for general log fano}
Let $(X,\Delta)$ be a log Fano variety. Then any Kähler-Einstein
metric $\phi_{KE}$ for $(X,\Delta)$ maximizes the functional $\mathcal{G}_{(X,\Delta)}$
on $\mathcal{E}(X,-(K_{X}+\Delta)).$
\end{prop}
For future reference we also note that there is a ``twisted'' variant
of the previous setting obtained by replacing $\mu_{\phi}$ with $\mu_{r\phi+(1-r)\phi_{0}}$
for any given $\phi_{0}\in\mathcal{H}_{b}(X,-(K_{X}+\Delta))$ and
$r\in[0,1].$ Then the previous proposition still holds (with the
same proof) when $\mathcal{G}_{(X,\Delta)}$ is replaced by the corresponding
functional $\mathcal{G}_{(X,\Delta,\phi_{0},r)}$ and $\phi_{KE}$
with the corresponding twisted Kähler-Einstein metric (see also the
even more general setting of mean field type equations in \cite{berm2}).

\subsection{Toric varieties and polytopes}

Let $T$ be the unit-torus in $\C^{n}$ of real dimension $n$ and
denote by $T_{c}:=(\C^{*})^{n}$ its compactification, with its standard
group structure. A $n-$dimensional algebraic variety $X$ is said
to be\emph{ toric} if it admits an effective holomorphic action of
the complex torus $T_{c}$ with an open dense orbit. In practice,
we will fix such an embedding and identify $T_{c}$ with its image
in $X.$

We will next briefly recall the well-known correspondence between\emph{
$T_{c}-$equivariant polarizations $(X,L)$} and \emph{convex lattice
polytopes $P.$ }In fact, using the scaling $L\mapsto rL$ and $P\mapsto rP$
this will give a correspondence between polarizations by $\Q-$line
bundles and\emph{ rational} polytopes. First recall that there are
two equivalent ways of defining a polytope $P$ in a vector space,
say in $\R^{n}:$
\begin{enumerate}
\item $P$ is the convex hull of a finite set of points $A$ (and $P$ is
called a \emph{lattice (rational) polytope} if $A\in\Z^{n}$ $(\Q^{n})).$ 
\item $P$ is the intersection of a finite number of half spaces $\left\langle l_{F},\cdot\right\rangle \geq-a_{F},$
where $l_{F}$ is a vector in the dual vector space and the label
$F$ thus runs over the facets $F$ of $P.$
\end{enumerate}
In the following all polytopes $P$ will assumed to be full-dimensional.
Let us first consider the correspondence referred to above using the\emph{
first description} of a lattice polytope above. Starting with a $T_{c}-$equivariant
ample line bundle $L$ on $X$ one considers the induced action of
the group $T_{c}$ on the space $H^{0}(X,kL)$ of global holomorphic
sections of $kL\rightarrow X$ (for $k$ a given positive integer).
Decomposing the action of $T_{c}$ according to the corresponding
one-dimensional representations $e^{m},$ labeled by $m\in\Z^{n}:$
\[
H^{0}(X,kL)=\oplus_{m\in B_{k}}\C e^{\alpha}
\]
one then defines the lattice polytope $P_{(X,L)}$ as the convex hull
of $B_{k}$ in $\R^{n}.$ Note that, from an abstract point of view,
$\R^{n}$ thus arises as $M\otimes_{\Z}\R,$ where $M$ is the character
lattice of the group $T_{c}$ (compare \cite{c-l-s}).

Conversely, given a convex lattice polytope $P$ one obtains a pair
$(X_{P},kL_{P})$ by letting $X_{P}$ be the closure of the image
of $X_{P}$ under the following map: 
\begin{equation}
T_{c}\rightarrow\P(\oplus_{m\in kP\cap\Z^{n}}\C e^{m}),\,\,\,\, z\mapsto[z^{m_{1}}:\cdots z^{m_{N}}]\label{eq:emb of torus in proj}
\end{equation}
 equipped with its standard action of $T_{c},$ taking $kL$ as the
restriction of the line bundle $\mathcal{O}(1)$ on $\P^{N-1}$ (it
is well-known \cite{c-l-s} that this is an embedding for $k$ sufficiently
large, where in fact $k=1$ will do if $X$ is smooth). 

Next, we will briefly recall how the\emph{ second description} of
$P_{(X,L)}$ above arises from the toric point of view. After perhaps
twisting the action on $T_{c}$ (which corresponds to translating
the polytope) we may as well assume that $0$ in an interior point
of $P_{(X,L).}$ Now any rational polytope $P$ containing zero in
its interior may be uniquely represented as 
\begin{equation}
P=\{p\in M_{\R}:\,\,\left\langle l_{F},p\right\rangle \geq-a_{F}\}\label{eq:repres of rational pol}
\end{equation}
 for primitive dual lattice vectors $l_{F}$ and strictly positive
rational numbers $a_{F},$ where the index $F$ runs over all facets
of $P.$ Let $s_{0}$ be the equivariant element in $H^{0}(X,L)$
corresponding to $0$ and denote by $D_{0}$ its zero-divisor. By
the orbit-cone correspondence (or rather orbit-\emph{face} correspondence)
\cite{c-l-s} the facets $F$ of $P_{(X,L)}$ correspond to the $T_{c}-$invariant
prime divisors $D_{F}$ on $X$ and hence any $T_{c}-$invariant divisor
$D$ on $X$ can be written uniquely as 
\begin{equation}
D=\sum_{F}c_{F}D_{F}\label{eq:expans of divisor}
\end{equation}
for some integers $c_{F}.$ When $D=D_{0}$ the integers $c_{F}$
are precisely the positive numbers $a_{F}$ appearing in \ref{eq:repres of rational pol}
(note that the divisor $D_{0}$ is referred to as $D_{P}$ in \cite{c-l-s}).
In other words, $a_{F}=\nu_{D_{F}}(s_{o});$ the order of vanishing
of $s_{0}$ along the corresponding prime divisor $D_{F}.$ As a consequence,
if $s_{m}$ is an arbitrary equivariant element in $H^{0}(X,L)$ then
\[
\nu_{D_{F}}(s_{m})=\nu_{D_{F}}(\chi^{m}s_{o})=\nu_{D_{F}}(\chi^{m})+\nu_{D_{F}}(s_{0})=\left\langle l_{F},m\right\rangle +\alpha_{F}\geq0,
\]
 where $\chi^{m}$ is the character corresponding to $m$ which may
be identified with a rational function on $T_{c}$ and where we have
used the basic fact that $\nu_{D_{F}}(\chi^{m})=\left\langle l_{F},m\right\rangle .$
Hence $m$ is a lattice point in $P_{(X,L)}$ (which was the starting
point of the previous correspondence).

\subsubsection{\label{sub:The-canonical-divisor-log fano}The canonical divisor
and toric (log) Fano varieties }

Let $X$ be a toric variety. Then $\pm K_{X}$ exists as a divisor
on $X$ (but in general not as $\Q-$line bundles) and 
\begin{equation}
-K_{X}\sim\sum_{F}D_{F},\label{eq:exp of canonical div}
\end{equation}
where the index $F$ ranges over all $T_{c}-$invariant prime divisors
of $X.$ As we will next explain there is correspondence between \emph{toric
log Fano varieties $(X,\Delta)$} on one hand and \emph{rational convex
polytopes $P$ containing $0$ in the interior}. First we note that
in the particular case when $X$ is a\emph{ Fano }variety there is
a canonical lift of $T_{c}$ to the line bundle $rK_{X},$ which thus,
as explained above, gives rice to a lattice polytope $rP_{X}.$ Similarly,
if $(X,\Delta)$ is a log Fano variety then we can get a canonical
rational polytope $P_{(X,\Delta)}$ by setting $P_{(X,\Delta)}=P_{(X,L)}$
for $L=-(K_{X}+\Delta)$ with the $T_{c}$ action induced by the one
from $-K_{X},$ i.e. the action is compatible with the natural isomorphism
between $L$ and $-K_{X}$ on the embedding of$T_{c}$ (using that
$\Delta$ has a canonical trivialization on $T_{c}).$
\begin{prop}
Let $(X,\Delta)$ be a log-Fano variety and let $P_{(X,\Delta)}$
be the corresponding rational polytope in $\R^{n}.$ Then $P_{(X,\Delta)}$
is a rational polytope containing $0$ in its interior and the coefficients
$c_{F}$ of $\Delta$ in \ref{eq:expans of divisor} are given by
$1-a_{F}.$ Conversely, if $P$ is a rational polytope containing
zero in its interior then $P=P_{(X,\Delta)}$ for a log Fano variety
$(X,\Delta).$ In particular, $\Delta$ is effective iff $a_{F}\leq1$
and $X\mapsto P_{(X,0)}$ gives a correspondence between Fano varieties
and polytopes $P$ as above with $a_{F}=1.$ \end{prop}
\begin{proof}
Let us start with the case when $\Delta=0,$ where a proof can be
found in \cite{c-l-s} and hence we just sketch the proof. First,
assume that $X$ is a Fano variety and $-rK_{X}$ is an ample line
bundle, for $r$ large. Then there is a section $s\in H^{0}(X,-rK_{X})$
with zero divisor $r\sum D_{F}$ (as follows form the linear equivalence
\ref{eq:exp of canonical div}). Under an equivariant embedding of
$T_{c}$ in $X$ the section $s$ pulls back to a multiple of the
$r$ th tensor power of the holomorphic $n-$vector field $z_{1}\frac{\partial}{\partial z_{1}}\wedge\cdots\wedge z_{n}\frac{\partial}{\partial z_{n}}$
on $T_{c}$ (using that the latter $n-$vector field is naturally
defined on $X_{reg}$ with the right zero locus). As a consequence,
$s$ is invariant under the $T_{c}-$ action, i.e. in the notation
above $s=s_{0}$ and hence $a_{F}(rP_{X})=\nu_{D_{F}}(s_{0})=r,$
i.e. $a_{F}(P_{X})=1$ as desired. Conversely, if $a_{F}(P)=1$ and
we take $r$ such that $rP$ is a lattice polytope. Then $s_{0}\in H^{0}(X,rL)$
is such that $\nu_{D_{F}}(s_{0})=a_{F}(rP)=r\cdot1,$ i.e. the zero
divisor of $s_{0}\in H^{0}(X,rL_{P})$ is equal to $r\sum D_{F}$
and hence (by \ref{eq:exp of canonical div}) $-K_{X}\sim L_{P}$
is ample as desired.

The proof for a general $\Delta$ is similar: if $-r(K_{X}+\Delta)$
is an ample line bundle, for $r$ large, we can take $s\in H^{0}(X,-r(K_{X}+\Delta)$
with zero divisor $r(\sum_{F}D_{F}-\Delta)$ which is indeed effective
iff $\Delta$ has coefficients $<1.$ We then deduce that $s=s_{0}$
as before (since $\Delta$ is trivial on $T_{c})$ and hence $a_{F}(P_{(X,\Delta)})=\nu_{D_{F}}(s_{0})/r$
is a rational positive number. The converse is then obtained as before.
\end{proof}

\subsection{\label{sub:Toric-metrics-as}Toric metrics as convex functions on
$\R^{n}$ and Legendre transforms }

Let now $(X,L)$ be a toric variety with corresponding polytope $P$
and assume that $0\in P.$ As before we denote by $s_{0}$ the corresponding
element in $H^{0}(X,L).$ Given any metric $\left\Vert \cdot\right\Vert $
on $L$ we obtain a function $\phi(x)$ on $\R^{n}$ by setting 
\begin{equation}
\phi(x):=-\log\left\Vert s_{0}\right\Vert ^{2}(z),\label{eq:correspondence metric convex f}
\end{equation}
 where $x=\log z$ (see section \ref{sub:Relation-to-the-log}) wrt
the fixed embedding of $T_{c}$ in $X,$ where $s_{0}$ is non-vanishing.
\begin{prop}
The correspondence \ref{eq:correspondence metric convex f} gives
a bijection between the space of $\mathcal{H}_{b}(X,L)^{T}$ of $T-$invariant
locally bounded metrics on $L\rightarrow X$ with positive curvature
current and the space $\mathcal{P}_{+}(\R^{n}).$ In particular, the
Legendre transform then induces a bijection between $\mathcal{H}_{b}(X,L)^{T}$
and the space $\mathcal{H}_{b}(P)$ of bounded convex functions on
$P.$\end{prop}
\begin{proof}
First note that, by definition, $\left\Vert \cdot\right\Vert $ has
positive curvature iff $\phi(z)$ is psh iff $\phi(x)$ is convex.
Next, let $h_{0}(=\left\Vert \cdot\right\Vert ^{2})$ be a fixed element
in $\mathcal{H}_{b}(X,L)^{T}$ with curvature current $\omega_{0}.$
Writing an arbitrary metric on $L$ as $h=e^{-v}h_{0}$ gives a bijection,
$h\mapsto v,$ between $\mathcal{H}_{b}(X,L)$ and the space $PSH_{b}(X,\omega_{0})^{T}.$
Moreover, since $T_{c}$ is embedded as a Zariski open set in $X$
it follows from the basic fact that any psh function, which is bounded
from above, extends uniquely over an analytic set, that we may as
well replace $PSH_{b}(X,\omega_{0})^{T}$ with its restriction to
$T_{c}.$ Now, by definition, the space of all $\phi(x)$ in $\mathcal{P}_{+}(\R^{n})$
may be identified with $PSH_{b}(T_{c},\omega_{P})^{T},$ where $\omega_{P}:=dd^{c}\phi_{P}.$
To conclude the proof it will thus be enough to show that $\omega_{P}=F^{*}\omega_{0}$
for some $h_{0}\in\mathcal{H}_{b}(X,L),$ where $F$ is the embedding
of $T_{c}$ in $X.$ To this end we fix $k>0$ such that $kL$ is
very ample, i.e the map \ref{eq:emb of torus in proj} is an embedding.
Let $h_{0}$ be the locally bounded (in fact continuous) metric with
positive curvature on $\mathcal{O}(1)\rightarrow\P^{N-1}$ induced
by the continuous two-homogenous psh function $\Phi(Z):=\log\max_{i=1,..,N}|Z_{i}|^{2}$
on $\C^{N}-\{0\}$ (the total space of $\mathcal{O}(1)^{*}\rightarrow\P^{N-1}).$
By definition the restriction of $h_{0}$ to the image of $X$ in
$\P^{N-1}$ is an element in $\mathcal{H}_{b}(X,L)$ and $-\log h_{0}(s_{0})(z)=\phi_{P}(z)$
and hence $F^{*}\omega_{0}=dd^{c}\phi_{P}$ as desired. \end{proof}
\begin{rem}
\label{rem:guilllem}As shown by Guillemin smooth strictly positively
metrics correspond, under the Legendre transform, to smooth functions
on the interior of $P$ with a particular boundary singularity (see
\cite{b-g-l} for the extension to singular toric varieties).
\end{rem}

\subsection{\label{sub:Toric-K=0000E4hler-Einstein-metrics}Toric Kähler-Einstein
metrics and solitons (proofs of Theorems \ref{thm:-existence of ke intro},
\ref{thm:existe of k-r sol intro})}

Here we will prove Theorems \ref{thm:-existence of ke intro} and
\ref{thm:existe of k-r sol intro} apart from the statements concerning
K-stability and Futaki invariants which will be considered in section
\ref{sub:Futaki-invariants-and}.

Let $X$ be a toric log Fano variety $(X,\Delta)$ and denote by $P(=P_{(X,\Delta)})$
the corresponding rational polytope. As explained above $P$ contains
$0$ as an interior point and the corresponding invariant element
$s_{0}\in H^{0}(X,-r(K_{X}+\Delta))$ is such that $T_{c}=\{s_{0}\neq0\}.$
Moreover, under the canonical identification of $K_{X}$ with $K_{X}+\Delta$
on $T_{c}$ we may identify the dual of an $r$th root of $s_{0}$
with the standard invariant $(n.0)-$form $dz$ on $T_{c}$ and hence
under the correspondence in section the precious section the canonical
measure on $X$ defined by a metric $\phi$ on $-(K_{X}+\Delta)$
satisfies

\begin{equation}
\mbox{Log }_{*}\mu_{\phi}=e^{-\phi(x)}dx\label{eq:toric can measure}
\end{equation}
Hence, the Kähler-Einstein equation \ref{eq:k-e equation for phi on general fano}
is equivalent to the equation 
\begin{equation}
MA_{\R}(\phi)=Ce^{-\phi}dx\label{eq:real k-e}
\end{equation}
for a convex function $\phi\in\mathcal{E}_{P}^{1}(\R^{n}),$ where
$dx$ denotes the usual Euclidean measure. Recall that geometrically
$e^{-\phi}$ is the point-wise norm of $s_{0}^{1/r}$ for the given
metric on $-(K_{X}+\Delta).$ We can now deduce the equivalence between
the first two points in Theorem \ref{thm:-existence of ke intro}
from Theorem \ref{thm:existe real ma intr} with $g=1$ (the regularity
is shown in \cite{bbegz}).

More generally, given a toric holomorphic vector field $V=\sum a_{i}\frac{\partial}{\partial z_{i}}$
we may define the corresponding (singular) Kähler-Ricci soliton metric
$\phi\in\mathcal{E}(X,-(K_{X}+\Delta))$ by the equation 
\begin{equation}
MA_{\R}(\phi)=Ce^{-\phi+\left\langle a,d\phi\right\rangle }dx\label{eq:real k-r-s}
\end{equation}
for $\phi\in\mathcal{E}_{P}^{1}(\R^{n}).$ By Theorem \ref{thm:existe of k-r sol intro}
$\phi$ is in fact automatically smooth on $\R^{n},$ i.e. the corresponding
metric on $-(K_{X}+\Delta)$ is smooth on the complex torus $T_{c}$
in $X.$ Note that, for any smooth $\phi$ in $\mathcal{P}(\R^{n})$
the function 
\begin{equation}
H_{V}(\phi):=\left\langle a,d\phi\right\rangle \label{eq:hamiltonian toric}
\end{equation}
is globally bounded on $\R^{n}.$ Indeed, $d\phi$ takes values in
the bounded set $P$ and hence 
\begin{equation}
\left|H_{V}(\phi)\right|\leq C\label{eq:hamiltonian toric-2}
\end{equation}
for a constant $C$ independent of $\phi.$ To see the relation to
the usual Kähler-Ricci soliton equation \ref{eq:k-r soliton eq intro}
we note that a simple computation gives, 
\begin{equation}
dd^{c}H_{V}(\phi):=-d(i_{V}\omega)\label{eq:hamiltonian toric-1}
\end{equation}
where $\omega=dd^{c}\phi,$ where the rhs by Cartan's formula equals
$-L_{V}\omega$ and hence $\omega$ indeed satisfies the equation
\ref{eq:k-r soliton eq intro} on the complex torus $T_{c}.$ Finally,
applying Theorem \ref{thm:existe real ma intr} with $g(p)=e^{\left\langle a,p\right\rangle }$
and using that $0$ is the barycenter of $(P,e^{\left\langle a,p\right\rangle })$
iff $a$ is the unique critical point of the strictly convex function
$a\mapsto\log\int_{P}e^{\left\langle a,p\right\rangle }dp,$ gives
Theorem \ref{thm:existe of k-r sol intro} up to the global regularity
statement on $X.$ The global continuity of the metric on $-(K_{X}+\Delta)\rightarrow X$
defined by $\phi$ follows from the bound \ref{eq:hamiltonian toric-2},
which implies that the Monge-Ampère measure of the finite energy metric
has a density in $L^{p}(X,\mu_{P})$ for any $p>1$ (compare ...)
and hence the continuity follows immediately as in the case $a=0$
considered in \cite{bbegz}. As for the global smoothness on the complement
of $\Delta$ in the regular locus of $X$ it will be established in
section \ref{sub:Regularity-of-singular}.
\begin{example}
\label{ex: on the sphere}When $n=1$ we have $P=[a_{1},a_{2}]$ and
$X$ is the Riemann sphere $\C\cup\{\infty\}$ with $\Delta=(1-a_{1})[0]+(1-a_{2})[\infty].$
The barycenter condition for the existence of a log Kähler-Einstein
metric forces $a_{1}=a_{2}=t$ for some positive number $t.$ For
any $t$ a direct calculation reveals that $\phi(x)=\log(e^{tx}+e^{-tx})$
gives a solution to equation \ref{eq:real k-e} and hence, by the
uniqueness statement in Theorem \ref{thm:existe of k-r sol intro},
any other solution is given by $\log(e^{t(x+s)}+e^{-t(x+s)})$ for
some $s\in\R.$ Geometrically, the corresponding Kähler metrics $\omega_{a}$
on the two-sphere are thus ``foot-balls'' with a cone angle $2\pi t:$
\begin{equation}
\omega_{t}:=dd^{c}\phi(z)/t=\frac{1}{2\pi}\frac{|z|^{2(t-1)}}{(1+|z|^{2t})^{2}}dz\wedge d\bar{z}\label{eq:log ke on sphere}
\end{equation}
and the parameter $s$ just comes from the action of the automorphism
group of $(X,\Delta),$ which does not change the isometry class of
the corresponding Riemannian metrics on the two-sphere. However, for
some $t>1,$ there may be different mutually non-isometric Kähler
metrics $\omega$ solving the log Kähler-Einstein equation for $(X,\Delta).$
In fact, as shown in \cite{tr}, this happens precisely when $\Delta$
has negative integer coefficients. \end{example}
\begin{rem}
\label{rem:toric is klt}Any toric log Fano variety $(X,\Delta)$
in fact has klt singularities See \cite{c-l-s} for an algebraic proof,
but from the analytical point this follows almost immediately. Indeed,
letting $\phi$ be a locally bounded metric on $X$ represented by
the function $\phi\in\mathcal{P}_{+}(\R^{n})$ the mass of $\mu_{\phi}$
on $X$ coincides, according to formula \ref{eq:toric can measure}
with $\int_{P}e^{-\phi(x)}dx.$ But, since $0$ is in the interior
of $P$ and $\phi-\phi_{P}$ is bounded we have $\phi(x)\geq|x|/C-C$
and hence the integral is indeed finite.
\end{rem}

\subsection{\label{sub:Relations-to-complete}Relations to complete Kähler-Ricci
solitons}

Now assume for simplicity that $X$ is \emph{smooth} toric variety
and consider a family of toric $\Q-$divisors $\Delta_{t}$ with coefficients
$<1$ for $t\in]0,1],$ such that $\Delta_{t}$ is affine wrt $t$
and converges to a reduced divisor $\Delta_{0}$ as $t\rightarrow0,$
i.e. $\Delta_{t}=\Delta_{0}+O(t).$ More precisely, the coefficients
$c_{F}(t)$ are assumed to tend to either zero or one, as $t\rightarrow0.$
Let $\omega_{t}$ be the corresponding curve of log Kähler-Ricci solitons
associated to $(X,\Delta_{t}),$ which, by Theorem \ref{thm:existe of k-r sol intro},
is uniquely determined modulo toric automorphisms. It seems natural
to conjecture that, as $t\rightarrow0,$ the scaled metrics $\tilde{\omega}_{t}:=\omega_{t}/t$
converges towards a \emph{complete} (translating) Kähler-Ricci soliton
on the quasi-projective variety $Y:=X-\Delta_{0},$ i.e. 
\begin{equation}
\mbox{Ric \ensuremath{\omega=}}L_{V}\omega\label{eq:translating k-r soliton}
\end{equation}
on $Y$ for some holomorphic vector field $V$ on $Y,$ which is the
limit of $\tilde{V_{t}}:=tV_{t}.$ Of course, the notion of convergence
needs to be made precise, but the least one could ask for is that
- modulo toric automorphisms - the convergence holds in the weak topology
of currents on $Y.$ The rational for this conjecture is that, on
$X,$ we have $\mbox{Ric \ensuremath{\tilde{\omega}_{t}=}}t\tilde{\omega}_{t}+L_{\tilde{V_{t}}}\tilde{\omega}_{t}+[\Delta_{t}]$
and hence, when $t\rightarrow0,$ at least heuristically, one obtains
a limiting Kähler current $\omega$ such that $\mbox{Ric \ensuremath{\omega=L_{V}\omega+}}[\Delta_{0}],$
which indicates that $\omega$ is asymptotic to a Euclidean cylinder
in the normal directions close to the ``boundary'' $\Delta_{0}$
of $Y.$ For example, taking $X$ to be $\P^{n}$ and $\Delta_{0}$
the hyperplane at infinity, so that $Y=\C^{n},$ should give the the
complete Kähler-Ricci soliton on $\C^{n}$ constructed by Cao \cite{cao1},
generalizing Hamilton's ``cigar soliton'' in $\C.$ Similarly, taking
$X$ to be the total space of the bundle $\P(\mathcal{O}(k)\oplus\mathcal{O}(0))\rightarrow\P^{n-1}$
and $\Delta_{0}$ the ``section at infinity'' should give the complete
Kähler-Ricci soliton in the total space of $\mathcal{O}(k)\rightarrow\P^{n-1}$
found in \cite{cao1}. For similar limit considerations, with very
precise converge results, see \cite{f-i-k}.

\subsection{\label{sub:The-invariant-R of log Fano}The invariant $R(X,\Delta)$
of a log Fano variety and lower bounds on the Ricci curvature}

Let $(X,\Delta)$ be a log Fano variety and fix a smooth semi-positive
form $\omega_{0}\in c_{1}(-(K_{X}+\Delta).$ Given $r\in[0,1]$ we
consider the following ``twisted Kähler-Einstein equation'' for
a Kähler current $\omega\in c_{1}(-(K_{X}+\Delta)):$ 
\begin{equation}
\mbox{Ric \ensuremath{\omega-\Delta=r\omega+(1-r)\omega_{0}}}\label{eq:aubins equ for log fano}
\end{equation}
for $\omega$ smooth on $X_{0}(:=X_{reg}-\Delta)$ and with continuous
local potentials on $X.$ In the case when $X$ is smooth and $\Delta=0$
the equation was introduced by Aubin as a continuity method to produce
Kähler-Einstein metrics. The following theorem generalizes the main
result of \cite{li1} (which concerned the case when $\Delta=0$ and
$X$ is smooth):
\begin{thm}
\label{thm:R for lof fano}Let $(X,\Delta)$ be a toric log Fano variety
and $\omega_{0}$ a smooth semi-positive form in $c_{1}(-(K_{X}+\Delta).$
Then the supremum over all $r$ such that the equation \ref{eq:aubins equ for log fano}
admits a solution coincides with the invariant $R_{P}$ (formula \ref{eq:inv R of conv bod})
of the canonical polytope $P$ associated to $(X,\Delta).$ \end{thm}
\begin{proof}
Given the ``toric dictionary'' above the theorem follows immediately
from Theorem \ref{thm:R for ma}, apart from the global regularity
of the solutions, which in turn follows from Theorem 1.5 in \cite{bbegz}.
\end{proof}
As shown in \cite{sz} when $X$ is any smooth (and not necessary
toric) Fano manifold and $\Delta=0$ the sup over all $r\in[0,1[$
such that the equation \ref{eq:aubins equ for log fano} admits a
solution coincides with the geometric invariant $R(X)$ defined as
the the sup of all numbers $r\in[0,1[$ such that there exists a Kähler
metric $\omega\in c_{1}(-K_{X})$ with $\mbox{Ric \ensuremath{\omega\geq r\omega.}}$
Here we note that one can similarly define an invariant $R(X,\Delta)$
of any log Fano variety $(X,\Delta),$ as the sup over all $r\in[0,1[$
such that there exists a Kähler current $\omega\in c_{1}(-(K_{X}+\Delta),$
smooth on $X_{0}$ and such that $\mbox{Ric \ensuremath{\omega-\Delta-r\omega}}$is
a smooth positive form. Then the following generalization of the main
result of \cite{sz} holds:
\begin{thm}
\label{thm:general R as ricci curv}Let $(X,\Delta)$ be a log Fano
variety with klt singularities. If $\Delta$ is an effective divisor
and $\omega_{0}$ a given semi-positive form in $c_{1}(-(K_{X}+\Delta)),$
then the invariant $R(X,\Delta)$ coincides with the sup over all
$r$ such that the equation \ref{eq:aubins equ for log fano} admits
a solution. Moreover, in the case when $(X,\Delta)$ is toric the
divisor $\Delta$ need not be effective.\end{thm}
\begin{proof}
Let us start by noting that the sup over all $r$ such that the equation
\ref{eq:aubins equ for log fano} admits a solution coincides with
the sup over all $r\in[0,1[$ such that the Ding type functional $\mathcal{G}_{(X,\Delta,\phi_{0},r)}(=:\mathcal{G}_{r,\phi_{0}})$
is bounded from above. First, if $\mathcal{G}_{r+\delta,\phi_{0}}\leq C$
then a simple scaling argument gives that $\mathcal{G}_{r,\phi_{0}}$
is coercive and by the variational approach in \cite{bbegz} there
hence exists a solution $\omega$ to the equation \ref{eq:aubins equ for log fano}.
Conversely, if the latter equation admits a solution, then it follows
from Prop \ref{prop:k-e max ding funtional for general log fano}
and the subsequent discussion that the functional $\mathcal{G}_{r,\phi_{0}}$
is bounded from above (note that in the toric case the convexity argument
uses Prekopa's theorem in $\R^{n},$ or its generalization in \cite{b-b2},
and hence does not rely on the positivity of the current). Finally,
we note that since $\phi_{0}-\phi_{0}'$ is bounded the upper boundedness
of $\mathcal{G}_{r,\phi_{0}}$ holds for one choice of $\phi_{0}$
precisely one it holds for\emph{ any }choice of $\phi_{0}$ and hence
the invariants above both coincide with $R(X,\Delta).$
\end{proof}

\subsection{\label{sub:Existence-of-K=0000E4hler-Einstein singular along div}
Relations to the work of Song-Wu and Li-Sun}

We start by rephrasing the existence results in Theorem \ref{thm:-existence of ke intro}
in terms of a given polarized toric variety $(X,L),$ where $L$ is
 thus an ample toric $\Q-$line bundle over $X.$ As explained in
section \ref{sub:The-canonical-divisor-log fano} the rational polytope
$P:=P_{(X,L)}$ may be written as $P_{(X,\Delta_{L})}$ for a toric
(Weil) $\Q-$divisor $\Delta_{L}$ such that $L$ is linearly equivalent
to $-(K_{X}+\Delta_{L}).$ Next, after replacing $P$ by $P':=P-\{b\},$
where $b$ is the barycenter of $P,$ we obtain a new polytope $P'$
with barycenter in the origin. This amounts to replacing $\Delta_{L}$
with another toric divisor $\Delta,$ linearly equivalent to $\Delta_{L},$
such that $P'=P_{(X,\Delta)}.$ Hence applying Theorem \ref{thm:-existence of ke intro}
we deduce the following
\begin{cor}
\label{cor:existe of log ke if delta changes}Let $X$ be a toric
variety and $L$ an ample toric $\Q-$line bundle over $X.$ Then
there exists a toric $\Q-$divisor $\Delta$ with coefficients in
$]-\infty,1[$ and a Kähler current $\omega\in c_{1}(L)$ with continuous
potentials on $X,$ such that $\omega$ is Kähler-Einstein on $X-\Delta,$
satisfying $\mbox{Ric \ensuremath{\omega-[\Delta]=\omega}}$ in the
sense of currents on $X.$ 
\end{cor}
Note however that the divisor $\Delta$ may not be effective, i.e.
its coefficents may be negative. In particular if $X$ is a Fano variety
and $L=-rK_{X}$ for some $r<1$ then it is natural to ask for which
$r$ the corresponding divisor $\Delta$ above is effective? We next
observe that for $r\leq R_{P},$ where $R_{P}$ defined as in section
\ref{sub:The-invariant-R of convex bod}, the corresponding divisor
$\Delta$ is indeed effective. To see this we take $\Delta_{L}$ to
be the canonical divisor scaled by $r$ so that $P_{(X,L)}=rP_{X},$
where we recall that $P_{X}$ is the set where $\left\langle l_{F},\cdot\right\rangle \geq-\text{1. }$Accordingly,
$P'$ is the set where $\left\langle l_{F},\cdot\right\rangle \geq-r(-1-\left\langle l_{F},b\right\rangle (:=-a_{F}(P').\text{ }$
Since the coefficents of $\Delta$ are given by $c_{F}=1-a_{F}=1-r-\left\langle rl_{F},b\right\rangle $
it follows that $\Delta$ is effective iff 
\[
1-r-r\left\langle l_{F},b\right\rangle \geq0,
\]
 for any $F.$ But when $r=R_{P}$ the previous inequality follows
immediately from relation \ref{eq:bound in proof of R for ma} applied
to $x:=-l_{F}$ and we thus deduce the following
\begin{cor}
\label{cor:a la song-wang}Let $X$ be a toric Fano variety and denote
by $P$ the canonical rational polytope associated to $X.$ Then,
for any $r\in]0,R_{P}],$ there exists an effective toric $\Q-$divisor
$D_{r},$ linearly equivalent to $-K_{X}$ and a singular Kähler metric
$\omega_{r}\in c_{1}(-K_{X})$ with continuous potentials, such that
\[
\mbox{Ric }\ensuremath{\omega_{r}=r\omega_{r}+(1-r)[D_{r}],}
\]
 More precisely, the coefficent $c_{F}$ of $D_{r}$ along the invariant
divisor $D_{F},$ defined by the facet $F$ of $P,$ is given by $c_{F}=1-\left\langle l_{F},b\right\rangle r/(1-r),$
where $l_{F}$ is the primitive lattice vector defining an inward
normal of the facet $F.$
\end{cor}
In the case when $X$ is smooth it is shown in \cite{so-j}, using
a method of continuity, that the metric $\omega_{r}$ in the previous
corollary in fact has edge-cone singularities along the divisor $D_{r}.$
As explaind in \cite{so-j} this latter result is closely related
to a conjecture of Donaldson \cite{do3} concerning the invariant
$R(X)$ (i.e. the greatest lower bound on the Ricci curvature) of
a smooth Fano variety. According to Donaldson's conjecture, if one
replaces the metric $\omega_{0}$ in equation \ref{eq:aubins equ for log fano}
with a current $[D],$ where $D$ is a given smooth divisor linearly
equivalent to $-K_{X},$ then the corresponding equation is still
solvable for any $r\in]0,R_{X}[.$ In other words, for any such $r$
there exists a log Kähler-Einstein metric $\omega_{r}$ associated
to the pair $(X,\frac{1-r}{r}D).$ It was moreover conjectured by
Donaldson that the metric has edge-cone singularities. Very recently,
Li-Sun \cite{li-s} confirmed a variant of this conjecture on a smooth
toric Fano variety, by using the result of Song-Wang (see Cor \ref{cor:a la song-wang}
and the subsequent discussion). More precisely, it was shown that
for a ``generic'' divisor $D_{\lambda}$ linearly equivalent to
$-\lambda K_{X},$ for $\lambda$ a sufficently divisible integer,
Donaldson's conjecture holds for $D:=D_{\lambda}/\lambda.$ Let us
briefly recall their elegant argument. By the result of Song-Wang
the equation in question can be solved for $r=R_{P}$ if one replaces
$D$ with $D_{R_{P}}.$ Next, Li-Sun use a $\C^{*}-$action to produce
a deformation of $D_{\lambda}$ to $D_{r}$ and deduce, by the convexity
results in \cite{bern2} (compare Prop \ref{prop:k-e max ding funtional for general log fano}),
that the log Ding functional of $(X,(1-R_{P})D_{\lambda})$ is bounded
from below. To conclude the proof they then need to show that the
log Ding functional of $(X,(1-r)D_{\lambda})$ is\emph{ proper }for
any $r<R_{P}$ (so that the existence results in \cite{j-m-r} can
be invoked). To this end Li-Song use a result from \cite{berm2} which
gives that the properness holds for $r$ sufficently small and then
finally conclude by an interpolation argument. It may be worth comparing
with the singular situation considered here. In case $X$ is a singular
Fano variety Cor \ref{cor:a la song-wang} can be used as a starting
point and by the generalized convexity results in \cite{bbegz} the
same argument as in the smooth case gives that the log Ding functionals
of $(X,(1-r)D)$ are bounded for any $r\leq R_{P}.$ However, to deduce
the properness (so that the existence results in \cite{bbegz} can
be invoked) one would need to further study the regularity properties
of the log pairs $(X,D_{\lambda}).$

\section{\label{sec:K-energy-type-functionals}K-energy type functionals and
K-stability}

Let us start by recalling the definition of the Mabuchi K-energy functional
$\mathcal{M}$ in Kähler geometry. This functional was first introduced
in the case when $X$ is smooth and $L\rightarrow X$ is an ample
line bundle. Then $\mathcal{M}$ is defined by the property that its
differential at $\phi\in\mathcal{H}(X,L)$ is equal to $-(S_{\phi}-\bar{S})(dd^{c}\phi)^{n}$,
where $S_{\phi}$ is the (suitably normalized) scalar curvature of
the Kähler metric $dd^{c}\phi.$ In the case when $L=-K_{X}$ and
$X$ is a Fano variety with log-terminal singularities it was shown
in \cite{berm2,bbegz} how to extend the definition of $\mathcal{M}$
to a singular setting (see also \cite{d-t} for related results).
In case $\phi$ is smooth and positively curved the formula reads
\begin{equation}
\mathcal{M}(\phi)=F(MA(\phi)),\,\,\,\,\, F(\mu):=-E(\mu)+D(\mu,\mu_{\phi_{0}})\label{eq:def of mab}
\end{equation}
where $E(\mu)$ is the\emph{ pluricomplex energy} of the measure $\mu$
(relative to $dd^{c}\phi_{0})$ and $D(\mu,\mu')$ denotes the classical\emph{
relative entropy} of $\mu$ wrt to $\mu':$ 
\[
D(\mu,\mu')=:D_{\mu'}(\mu):=\int_{X}\log(\mu/\mu')\mu(\geq0)
\]
 if $\mu$ is absolutely continuous wrt $\mu'$ and $D(\mu,\mu')=\infty$
otherwise. When $\mu=MA(\phi)$ for $\phi\in\mathcal{H}(X,L)$ we
have, by definition, that 
\[
E(MA(\phi))=\mathcal{E}(\phi,\phi_{0})-\int_{X}(\phi-\phi_{0})MA(\phi)
\]
It should be pointed out that in the case when $X$ is smooth the
corresponding formula \ref{eq:def of mab} is equivalent to a previous
formula of Tian and Chen \cite{ti0}.

We next come back to the setting of convex functions in $\R^{n}$
associated to a convex body $P,$ taking $\phi_{0}=\phi_{P}$ as the
reference. We also equip $P$ with a smooth positive density $g(p).$
For any function $\phi$ in $\mathcal{P}_{+}(\R^{n})$ we define the
following Mabuchi type functional associated to $(P,g):$ 
\[
\mathcal{M}_{(P,g)}(\phi)V(P,g)=-\mathcal{E}_{g}(\phi,\phi_{P})+\int\phi MA_{g}(\phi)+D(MA_{g}(\phi),dx),
\]
Note that $\mathcal{M}(\phi+c)=\mathcal{M}(\phi)$ and hence $\mathcal{M}$
is determined by its restriction to the sub space of all sup-normalized
elements. In this case when $P$ is the canonical rational polytope
associated to Fano variety $X$ $\mathcal{M}_{P}(\phi)$ coincides
with the Mabuchi functional of the $T-$invariant metric on $-K_{X}$
corresponding to $\phi.$ Indeed, the push-forward from $T_{c}$ to
$\R^{n}$ of $\mu_{\phi_{P}}$ may be written as $e^{-\phi_{P}}dx$
and hence $D(MA(\phi),\mu_{P})=D(MA(\phi),dx)-\int\phi_{P}MA(\phi).$ 

More generally, in the setting of a log Fano variety $(X,\Delta)$
with canonical rational polytope $P,$ with $\phi$ denoting a positively
curved metric on $-(K_{X}+\Delta),$ we will write $\mathcal{M}_{(X,\Delta,V)}$
for the Mabuchi type functional corresponding to $\mathcal{M}_{(P,g)}$
for $g(p)=e^{\left\langle a,p\right\rangle },$ where $V$ is the
holomorphic toric vector field $V$ with components $a_{i}.$ In the
case when $\Delta=0$ and $X$ is a Fano manifold the functional $\mathcal{M}_{(X,\Delta,V)}$
essentially coincides with the ``modified Mabuchi functional'' appearing
in \cite{t-z2}. Similarly, we will write $\mathcal{G}_{(X,\Delta,V)}$
for the functional corresponding to $\mathcal{G}_{g}.$

\subsection{Variational principles and coercivity}

We will say that a functional $\mathcal{F}$ on $\mathcal{P}_{+}(\R^{n})$
is\emph{ relatively coercive }if there exists a positive constant
$C$ such that 
\[
\mathcal{F}(\phi)\geq-\mathcal{E}(\phi,\phi_{P})/C-C
\]
 on the subspace of all normalized $\phi.$ In particular, $\mathcal{F}$
is then bounded from below on the latter subspace. In order to relate
this notion to other equivalent notions of coercivity (sometimes also
called strong properness) in Kähler geometry we recall the definition
of Aubin's $J-$functional, which is the scale invariant analog of
$-\mathcal{E}:$
\[
J(\phi,\phi_{0}):=-(\mathcal{E}(\phi,\phi_{0})+\int(\phi-\phi_{0})MA(\phi_{0})
\]
In particular, in the toric setting, $J(\phi,\phi_{P})=-\mathcal{E}(\phi,\phi_{P})$
if $\phi$ is sup-normalized, since $\phi-\phi_{P}=0$ on the support
of $MA(\phi_{P}).$ Fixing a\emph{ smooth }positively curved metric
$\phi_{0}$ we will simply write $J(\phi):=J(\phi,\phi_{0}).$
\begin{lem}
\label{lem:j for sup-normalized}Let $L\rightarrow X$ be a semi-positive
line bundle over a projective variety $X$ and let $\mathcal{H}_{0}$
denote the space of all smooth positively curved metrics on $L$ such
that $\sup_{X}(\phi-\phi_{0})=0$ for a fixed reference $\phi_{0}\in\mathcal{H}_{0}.$
Then there is a constant $C$ (only depending on the reference $\phi_{0})$
such that 
\[
|J(\phi,\phi_{0})-\left|\mathcal{E}(\phi,\phi_{0})\right||\leq C
\]
for any $\phi\in\mathcal{H}_{0}.$\end{lem}
\begin{proof}
The lemma follows immediately from the following estimate: there is
a constant $C$ such that, if $\mu_{0}:=MA(\phi_{0})$ 
\[
\int(\phi-\phi_{0})MA(\phi_{0})\leq C
\]
When $X$ is smooth the lemma is well-known \cite{g-z} and holds
more generally for any measure $\mu_{0}$ such that $\phi-\phi_{0}$
is in $L^{1}(X,\mu)$ for any $\phi\in\mathcal{H}.$ Taking a smooth
resolution $Y\rightarrow X$ and pulling back $L$ thus proves the
general case.
\end{proof}
The following proposition reveals the close connections between the
two functionals $\mathcal{G}_{P}$ and $\mathcal{M}_{P}:$ 
\begin{prop}
\label{pro:comparison mab and ding for P }Let $P$ be a convex body
containing $0$ in its interior. Then 

\begin{equation}
\inf_{\mathcal{P}_{+}(\R^{n})}-\mathcal{G}_{(P,g)}=\inf_{\mathcal{P}_{+}(\R^{n})}\mathcal{M}_{(P,g)},\label{eq:inf -G equal inf of M}
\end{equation}
 the minimizers of the two functionals coincide and $-\mathcal{G}_{P}$
is relatively coercive iff $\mathcal{M}_{P}$ is relatively coercive.\end{prop}
\begin{proof}
This is proved using Legendre transforms in infinite dimension, following
the argument in \cite{berm2}. First note that, up to a trivial scaling,
we may assume that $V(P,g)=1$ and to simplify the notation we will
omit the subindex $g$ in the following. To conform to the sign conventions
for the Legendre transforms used in the present paper it is convenient
to introduce the functional $\mathcal{I}_{-}(v):=\log\int_{\R^{n}}e^{-v}\mu_{P}$
and $\mathcal{E}_{-}(v):=-(\mathcal{E}\circ Pr)(\phi_{P}+v)$ and
set $\mathcal{G}_{-}=\mathcal{E}_{-}-\mathcal{I}_{-}$ which is thus
a difference of two\emph{ convex} functionals defined on the vector
space $\mathcal{C}_{b}(\R^{n})$ of all bounded continuous functions
$v$ on $\R^{n}.$ By definition, $-\mathcal{G}(\phi)=\mathcal{G}_{-}(\phi_{P}+v),$
for $v:=\phi-\phi_{P}$ if $\phi\in\mathcal{P}_{\text{+ }}(\R^{n}).$
Then, just as in Step 2 in the proof of Thm \ref{thm:existe real ma intr},
the infimum of $-\mathcal{G}$ over $\mathcal{P}_{+}(\R^{n})$ coincides
with the infimum of $\mathcal{G}_{-}$ over $\mathcal{C}_{b}(\R^{n}).$
We will also use the pairing $(v,\mu):=-\int_{\R^{n}}v\mu$ between
$\mathcal{C}_{b}(\R^{n})$ and the space $\mathcal{M}(\R^{n})$ of
all signed measures on $\R^{n}.$ The sign conventions have been chosen
so that, if $\mu$ is a probability measure, then we can write the
energy of a measure $\mu$ as a Legendre transform: 
\[
E(\mu)=(\mathcal{E}_{-})^{*}(\mu),
\]
where the Legendre transform of a functional $\mathcal{F}$ on the
vector space $\mathcal{M}(\R^{n})$ is defined by $\mathcal{F}^{*}(\mu):=\sup_{v\in\mathcal{C}_{b}(\R^{n})}((v,\mu)-\mathcal{F}(u)).$
Next, one notes that, since, by Prop \ref{pro:diff of composed energy convex},
the gradient of $\mathcal{E}_{-}$ takes values in the subspace $\mathcal{M}_{1}(\R^{n})$
of all probability measures in $\mathcal{M}(\R^{n}),$ it follows
that $(\mathcal{E}_{-})^{*}(\mu)=\infty,$ unless $\mu$ is a probability
measure. Similarly, it well-known that $D(\mu)=\mathcal{I}_{-}^{*}(\mu).$
Hence, $\mathcal{M}(\phi)=-(\mathcal{E}_{-})^{*}(\mu)+\mathcal{I}_{-}^{*}(\mu)$
for $\mu=MA(\phi).$ With these preparations in place the proof of
the equality \ref{eq:inf -G equal inf of M} follows immediately from
the monotonicity of the Legendre transform and the fact that it is
an involution (compare formula \ref{eq:using monoton of legendre}).
Finally, the last two statement are proved exactly as in \cite{berm2}.
\end{proof}
From the results in section \ref{sub:Variational-principles-and}
concerning $\mathcal{G}_{(P,g)}$ we then deduce the following variational
principle:
\begin{prop}
\label{prop:var prop wrt G and M}The following is equivalent for
$\phi\in\mathcal{P}_{+}(\R^{n}):$
\begin{itemize}
\item $MA_{g}(\phi)=e^{-\phi}$
\item $\phi$ minimizes the functional $-\mathcal{G}_{(P,g)}$
\item $\phi$ minimizes the functional $\mathcal{M}_{(P,g)}$
\end{itemize}
\end{prop}
Combining Proposition \ref{pro:comparison mab and ding for P } with
Theorem \ref{thm:coerciv k-e} also immediately gives the following
analog of the latter theorem (and Theorem \ref{thm:existe real ma intr}):
\begin{thm}
\label{thm:mab function is coerc for body}Let $P$ be a convex body
such that $0$ is in the interior of $P.$ Then there is a constant
$C$ such that $\mathcal{M}_{(P,g)}(\phi)$ is relatively coercive.
Moreover, \textup{$\mathcal{M}_{(P,g)}$ is bounded from below on
all of} $\mathcal{P}_{+}(\R^{n})$ iff $0$ is the barycenter of $(P,g)$
iff \textup{$\mathcal{M}_{(P,g)}$ admits an absolute minimizer $\phi$
solving the Monge-Ampère equation in Theorem \ref{thm:existe real ma intr}.}
\end{thm}
In the setting of toric varieties the previous results give the following
\begin{thm}
\label{thm:mab functional is coerc for toric is equi to ke}Let $(X,\Delta)$
be a toric log Fano variety with canonical polytope $P$ and $V$
a toric holomorphic vector field on $X$ with components $a_{i}.$
Then the following is equivalent: 
\begin{itemize}
\item For any $\delta>0$ there is a constant $C_{\delta}$ such that for
any $T-$invariant locally bounded metric on $-(K_{X}+\Delta)$  with
positive curvature 
\[
\mathcal{M}_{(X,\Delta,V)}(\phi)\geq(1-\delta)\inf_{t\in T_{c}}J(t^{*}\phi)-C_{\delta}
\]
(and similarly for the functional $\mathcal{G}_{(X,\Delta,V)})$
\item $0$ is the barycenter of $(P,e^{\left\langle a,p\right\rangle })$
(i.e. $a$ is the critical point of the Laplace transform of $1_{P}dp)$ 
\item $(X,\Delta)$ admits a (singular) Kähler-Ricci soliton with vector
field $V$ 
\end{itemize}
\end{thm}
\begin{proof}
By Prop \ref{pro:comparison mab and ding for P } it will be enough
to consider the functional $\mathcal{G}_{(X,\Delta,V)}).$ If the
inequality in the first point holds then $\mathcal{G}_{(X,\Delta,V)}(\phi)$
is bounded from below and hence the second point holds, by the previous
theorem. Conversely, if the second point above holds, then $\mathcal{G}_{(X,\Delta,V)}(\phi)$
is invariant under normalizations $\phi\mapsto\tilde{\phi}$ (by Lemma
\ref{lem:inv of m-t functional}) and hence by the previous theorem
\ref{lem:j for sup-normalized}
\[
\mathcal{G}_{(X,\Delta,V)}(\phi)\geq(1-\delta)(-\mathcal{E}(\tilde{\phi},\phi_{0}))-C_{\delta}
\]
Now, by the definition of normalization $-\mathcal{E}(\tilde{\phi},\phi_{0})\geq\inf_{t\in T_{c}}(-\mathcal{E})(t^{*}\phi,\phi_{0}).$
Finally, using that $\mathcal{G}_{(X,\Delta,V)}(\phi)$ is invariant
under $\phi\mapsto\phi+c$ and invoking Lemma \ref{lem:j for sup-normalized}
concludes the proof of the equivalence between the first and the second
point (which we already know is equivalent to the third point).
\end{proof}

\subsection{\label{sub:The-Mabuchi-funtional}The Mabuchi functional expressed
in terms of the Legendre transform on $P$}

In this section we will consider the ``Kähler-Einstein case'' when
$g=1.$ Following Donaldson \cite{d0} we denote by $\mathcal{C}^{\infty}$
the space of all functions $u$ on $P$ which are smooth and strictly
convex in the interior and continuous up to the boundary and let 
\[
\mathcal{F}(u):=\mathcal{M}_{P}(u^{*})V(P)/n!
\]
(note that if $u\in\mathcal{C}^{\infty}$ then $\phi:=u^{*}$ is a
smooth and strictly convex element in $\mathcal{P}(\R^{n})_{+}).$
We will show how to express the functional $\mathcal{F}$ in terms
of the following linear functional: 
\[
\mathcal{L}_{\sigma_{P}}(u):=\int_{\partial P}u\sigma_{P}-n\int_{P}udp,
\]
 where $\sigma_{P}$ is the canonical measure on $\partial P$ defined
by 
\begin{equation}
\sigma_{P}:=\frac{d}{dt}_{|t=1^{+}}(1_{tP}dp)\label{eq:canonical measure text}
\end{equation}
Equivalently, a simple argument shows that $P$ is absolutely continuous
wrt the standard measure $\lambda_{\partial P}$ on $\partial P$
induced by the Euclidean structure $\R^{n}$ and 
\begin{equation}
\sigma_{P}=\lambda_{\partial P}/\left\Vert d\rho\right\Vert \label{eq:canonical measure text in temds of lesb}
\end{equation}
a.e. on $\partial P,$ where $\rho$ is the Minkowski functional of
$P,$ i.e. the one-homogenous defining convex function such that $P=\{\rho<1\}.$
The next proposition can be seen as a generalization of a formula
of Donaldson \cite{d0} concerning the case when $P$ is a rational
simple polytope, i.e. there are precisely $n$ facets meeting a given
vertex. One virtue of the present approach is that it avoids any integration
by parts on $P$ (which seem rather complicated in the case of a non-simple
polytope). See section \ref{sub:Comparison-with-Donaldson's} for
a comparison with Donaldson's notation.
\begin{prop}
\label{pro:form for F in terms of linear}The following formula holds:
\begin{equation}
\mathcal{F}(u):=-\int_{P}\log\det(u_{ij})dp+\mathcal{L}_{\sigma_{P}}(u)\label{eq:formula for F in terms of linear}
\end{equation}
\end{prop}
\begin{proof}
We start by noting that 
\[
D(MA(\phi),dx)=-n!\int_{P}\log\det(u_{ij})dp,
\]
 which follows from making the change of variables $p=d\phi_{|x}$
in the integral defining the lhs above and using that, by duality,
$\det(u_{ij})\det(\phi_{ij})=1.$ The rest of the proof then follow
from combining formula \ref{eq:energy as legendre in global conv}
with the following lemma. \end{proof}
\begin{lem}
Let $\phi\in\mathcal{P}_{+}(\R^{n}).$ Then 
\[
\frac{1}{n!}\int\phi MA(\phi)=\int_{\partial P}ud\sigma_{P}-(n+1)\int_{P}udp
\]
where $u=\phi^{*},$ the Legendre transform of $\phi.$\end{lem}
\begin{proof}
By definition $\phi(x)=\left\langle p,x\right\rangle -v(p)$ for $x=dv(p).$
Hence, making the change of variables $p=d\phi(x)$ in the integral
\[
\frac{1}{n!}\int\phi MA(\phi)=(\int_{P}\left\langle p,dv\right\rangle -v(p))dp=\left(\int_{P}\left\langle p,dv\right\rangle +nv(p)dp\right)-\int_{P}(n+1)\int_{P}vdp,
\]
where we have rearranged the rhs in order to identify the first integral
with $\int_{\partial P}vd\sigma.$ To see this set $\sigma(t):=\int_{tP}vdp$
for $t>0.$ On one hand, by definition, $d\sigma(t)/dt_{t=1}=\int_{\partial P}v\sigma_{P}.$
On the other making the change of variables $p\rightarrow tp$ in
the integral defining $\sigma(t)$ and using Leibniz rule gives an
integral over $P$ which is precisely the one in the bracket above.\end{proof}
\begin{thm}
\label{thm:donaldson conj in text}Let $P$ be a convex body containing
$0$ in its interior. Then the following is equivalent: 
\begin{itemize}
\item The functional $\mathcal{F}$ (formula \ref{eq:formula for F in terms of linear})
admits a minimizer $u$ in $\mathcal{C}^{\infty}$ 
\item $0$ is the barycenter of $P$
\item For any convex function on $P$ we have $\mathcal{L}_{\sigma_{P}}(v)\geq0$
with equality iff $v$ is linear.
\end{itemize}
Moreover, the minimizer (when it exists) is unique modulo the addition
of affine functions and satisfies Abreu's equation 
\begin{equation}
S(u)=1\label{eq:abreus equa in thm}
\end{equation}
\end{thm}
\begin{proof}
As explained above, we have, up to normalization, that $\mathcal{F}(u)=\mathcal{M}_{P}(\phi)$
for $\phi=u^{*}$ and hence the equivalence between the first two
points follows from Theorem \ref{thm:mab function is coerc for body}
combined with Proposition \ref{eq:formula for F in terms of linear}.
The equivalence between the second and third point follows from Lemma
\ref{lem:lower bound on l when bary} below. To see that equation
\ref{eq:abreus equa in thm} holds we recall that if $u$ minimizes
$\mathcal{F}$ then, by Prop \ref{prop:var prop wrt G and M}, $\phi$
satisfies the corresponding real Monge-Ampère equation with $g=1.$
But then the corresponding Kähler metric on $T_{c}$ has constant
Ricci curvature and in particular constant scalar curvature so that
the equation \ref{eq:abreus equa in thm} follows from Abreu's formula
\cite{ab}.
\end{proof}
As explained in the introduction the previous theorem confirms a special
case of a conjecture of Donaldson in \cite{d0}. In the proof of the
previous theorem we used the following
\begin{lem}
\label{lem:lower bound on l when bary}Let $P$ be a convex body containing
$0$ in its interior. Then 
\[
\mathcal{L}_{\sigma_{P}}(u)=\int_{\partial P}u\sigma_{P}-n\int_{P}udp\geq\int_{P}udp
\]

for any convex function $v$ on $P$ such that $u(0)=0.$ Moreover,
equality holds above iff $u$ is linear. In particular, $\mathcal{L}_{\sigma_{P}}(u)>0$
for any non-affine convex function $v$ iff $0$ is the barycenter
of $P.$ \end{lem}
\begin{proof}
The lemma could be proved exactly as in Lemma 4.1 and Lemma 4.2 in
\cite{z-z}, which applies to any convex polytope. But for completeness
we give a simple alternative proof which works direct for any convex
body $P.$ Fix a convex function $u$ on $P$ (by a simple approximation
argument we may assume that $u$ is smooth on $\bar{P})$ and set
$\sigma(t):=\int_{tP}udp$ for $t>0.$ By definition $d\sigma(t)/dt_{t=1}=\int_{\partial P}ud\sigma_{P}.$
Since $u(0)=0$ and $u$ is convex $u(tp)/t$ is increasing in $t,$
Hence, $\sigma(t)/t^{n+1}$ is also increasing in $t$ (using the
change of variables $p\rightarrow tp$ in the integral), i.e. $d(\sigma(t)/t^{n+1})/dt\geq0.$
Evaluating the previous derivative at $t=1$ then proves the desired
inequality (using Leibniz rule) and the equality case also follows
since $u(tp)/t$ is constant if $u$ is linear. 
\end{proof}

\subsubsection{\label{sub:Comparison-with-Donaldson's}Comparison with Donaldson's
setting}

In \cite{d0}  Donaldson associates to any rational polytope $P$
another measure on $\partial P,$ that we will here denote by $\sigma_{P}'.$
It is induced from the integral lattice in $\R^{n}$ and defined as
$\sigma_{P}':=d\lambda/\left\Vert d\rho\right\Vert ,$ where now $\rho$
is given by $\rho(p):=\max_{F}(-\left\langle l_{F},p\right\rangle -a_{F})$
of $P$ (compare formula \ref{eq:repres of rational pol}), i.e. $\rho$
is a defining one-homogenous function of $P$ such that $d\rho$ is
a primitive integral vector on any facet. Hence, on any facet $F$
of $P$ 
\begin{equation}
\sigma_{P}=\sigma_{P}'/a_{F}\label{eq:relation between measures}
\end{equation}
and $\sigma_{P}=\sigma_{P}'$ iff $P$ is the canonical polytope of
a Fano variety. As shown by Donaldson, when $P$ is a Delzant polytope,
i.e. $P$ corresponds to a polarized toric manifold $(X,L)$ and the
boundary of $P$ is equipped with the measure $\sigma_{P},$ the solutions
$u$ as in Theorem \ref{thm:donaldson type intro} (which moreover
satisfy Guillemin's boundary conditions) are precisely the Legendre
transforms of toric metrics on $L$ whose curvature form $\omega\in c_{1}(L)$
has constant scalar curvature on $X.$ On the other hand, writing,
as in section \ref{sub:The-canonical-divisor-log fano}, 
\[
L=-(K_{X}+\Delta),\,\,\,\Delta=\sum_{F}(1-a_{F})D_{F}
\]
 the solutions obtained here, i.e. those induced by the measure $\sigma_{P},$
satisfy the following equation on $X:$ 
\begin{equation}
\mbox{Ric \ensuremath{\omega=\omega}}+\sum_{F}(1-a_{F})D_{F},\label{eq:k-r soliton eq intro-1}
\end{equation}
Accordingly they have constant scalar curvature on the complement
of the toric divisor ``at infinity'' $D$ with singularities along
the components $D_{F}$ of $D$ determined by the numbers $a_{F}.$ 

For future reference we also record the following consequence of the
relation \ref{eq:relation between measures}: 
\begin{equation}
\mathcal{L}_{\sigma_{P}}(u)-\mathcal{L}_{\sigma_{P}'}(u)=\sum_{F}(1-a_{F})(b_{F}\int_{P}udp-\int_{F}u\sigma_{P}'),\,\,\,\, b_{F}=\int_{F}\sigma_{P}'/\int_{P}dp,\label{eq:formula for F' in terms of F}
\end{equation}
 where $\mathcal{L}_{\sigma}(u)$ is defined by Donaldson's general
formula \ref{eq:def of general  linear functional intro}.

\subsection{\label{sub:Futaki-invariants-and}Futaki invariants and $K-$stability}

\subsubsection{Futaki invariants}

The Futaki invariant was originally defined for $X$ a smooth Fano
manifold as a Lie algebra character. Here we will follow the approach
of Ding-Tian \cite{d-t} which applies to any irreducible normal Fano
variety $X.$ Given a holomorphic vector field $W$ on the regular
locus $X_{0}$ the corresponding Futaki invariant $f(W)\in\R$ may
be defined as 
\[
f_{X}(W):=\frac{d}{dt}\mathcal{M}_{X}(\phi_{t})
\]
 where $\phi_{0}$ is a fixed metric, invariant under the corresponding
$S^{1}-$action and $\phi_{t}$ is the curve obtained by pull-back
$\phi_{0}$ under the flow of $W$ (strictly, speaking in \cite{d-t}
there is an extra extension condition on $V$ but as observed in \cite{bbegz}
the condition is always satisfied). As shown in \cite{d-t} $f_{X}(W)$
thus defined is independent of the reference $\phi_{0}$ and the time
$t.$ More generally, given a log Fano variety $(X,\Delta)$ and a
holomorphic vector field $W$ whose flow preserves the log regular
locus $X_{0}(=:X_{reg}-\Delta)$ we may define the log Futaki invariant
$f_{(X,\Delta)}(W)$ as above, by replacing $\mathcal{M}_{X}$ with
$\mathcal{M}_{(X,\Delta)}.$ Even more generally, given a holomorphic
vector fields $V$ and $W$ as above we define the modified log Futaki
invariant $f_{(X,\Delta,V)}(W)$ as above, by replacing $\mathcal{M}_{X}$
with $\mathcal{M}_{(X,\Delta,V}).$ The independence on $\phi_{0}$
and $t$ can then be checked as before.

In the toric log Fano case we have the following result, well-known
in the smooth case \cite{mab,d0} (when $\Delta$ is trivial):
\begin{prop}
\label{pro:toric futaki inv}Let $(X,\Delta)$ be a toric log Fano
variety and $W$ the invariant vector field on $X$ with components
$a\in\R^{n}.$ Then 
\[
f_{(X,\Delta)}(W):=-\mathcal{L}_{\sigma_{P}}(\left\langle a,p\right\rangle )
\]
In particular, $f_{(X,\Delta)}(W)=0$ for all $W$ iff $0$ is the
barycenter in the corresponding polytope $P_{(X,\Delta)}.$\end{prop}
\begin{proof}
Letting $\phi_{0}$ be a $T-$invariant metric we note that $\phi_{t}(x)=\phi(x+at).$
Setting $u_{t}:=(\phi_{t})^{*}$ this means that $u_{t}=u_{0}-\left\langle a,p\right\rangle t$
and hence the previous formula follows immediately from Lemma \ref{lem:lower bound on l when bary}.
\end{proof}

\subsubsection{\label{sub:K-stability}K-stability}

Let us start by recalling Donaldson's general definition \cite{d0}
of K-stability of a polarized variety $(X,L),$ generalizing the original
definition of Tian \cite{ti1}. First, a\emph{ test configuration}
for $(X,L)$ consists of a polarized projective scheme $\mathcal{L}\rightarrow\mathcal{X}$
with a $\C^{*}-$action and a $\C^{*}-$equivariant map $\pi$ from
$\mathcal{X}$ to $\C$ (equipped with its standard $\C^{*}-$action)
such that any polarized fiber $(X_{t},L_{t})$ is isomorphic to $(X,rL)$
for $t\neq0,$ for some integer $r.$ The corresponding \emph{Donaldson-Futaki
invariant} $f(\mathcal{X},\mathcal{L})$ is defined as follows: consider
the $N_{k}-$dimensional space $H^{0}(X_{0},kL_{0})$ over the central
fiber $X_{0}$ and let $w_{k}$ be the weight of the $\C^{*}-$action
on the complex line $\det H^{0}(X_{0},kL_{0}).$ Then the Donaldson-Futaki
invariant of $f(\mathcal{X},\mathcal{L})$ is defined as the sub-leading
coefficient in the expansion of $w_{k}/kN_{k}$ in powers of $1/k.$
More precisely, expanding 
\[
w_{k}=a_{0}k^{n+1}+a_{1}k^{n}+O(k^{n-1}),\,\,\,\,\, N_{k}:=b_{0}k^{n}+O(k^{n-1})
\]
gives 
\[
f(\mathcal{X},\mathcal{L})=\frac{1}{b_{0}^{2}}(a_{1}b_{0}-a_{0}b_{1})
\]
The polarized variety $(X,L)$ is said to be\emph{ K-polystable} if,
for any test configuration, $f(\mathcal{X},\mathcal{L})\leq0$ with
equality iff $(\mathcal{X},\mathcal{L})$ is a product test configuration.
Following \cite{l-x} we will also assume that the total space $\mathcal{X}$
of the test configuration is normal, to exclude some pathological
phenomena observed in \cite{l-x}. 

Similarly, if one also fixes a $\Q-$divisor $\Delta$ on $X,$ with
normal crossings, one can more generally define the \emph{log K-polystability}
of $(X,L)$ wrt $\Delta$ as before \cite{do3,li2,o-s}, but phrased
in terms of the corresponding \emph{log Donaldson-Futaki invariants}
defined by 
\[
f(\mathcal{X},\mathcal{L},\Delta):=f(\mathcal{X},\mathcal{L})+a_{0}\frac{\tilde{b}_{0}}{b_{0}}-\tilde{a}_{0},
\]
 where $\tilde{a}_{0}$ is the leading coefficient of the weight of
$\det H^{0}(\Delta_{0},kL_{0})$\emph{ }and $a_{0}$ is the leading
coefficient of the dimension of $H^{0}(\Delta_{0},kL_{0})$ (in the
definition we first assume that $\Delta$ is an irreducible divisor
and then extend by linearity). In particular, if $(X,\Delta)$ is
a log Fano variety then we say that $(X,\Delta)$ is log K-stable
if $L:=-(K_{X}+\Delta)$ is log K-stable wrt $\Delta.$
\begin{rem}
As explained in \cite{d0}, in the case when $X_{0}$ smooth the equivariant
Riemann-Roch theorem shows that the Donaldson-Futaki invariant $f(\mathcal{X},\mathcal{L})$
is proportional (with a sign difference) to the Futaki-invariant $f_{X_{0}}(W),$
where $W$ is the generator of the induced $\C^{*}-$action on $X_{0}$
(and a similar relation holds in the log setting \cite{li2}). 
\end{rem}
In the case when $X$ is a general polarized toric variety it was
shown by Donaldson \cite{d0} how to obtain toric test configurations
from a convex piece-wise linear rational function $u$ on a polytope
(called toric degenerations). Briefly, $(\mathcal{X},\mathcal{L})$
is the polarized toric variety such that the corresponding rational
polytope $Q$ is defined as one side of the graph of $u$ over $P$
with the projection $\pi$ defined so that the ``roof'' of $Q$
corresponds to the central fiber $X_{0}.$ 
\begin{prop}
\label{pro:(Donaldson).-toric test}(Donaldson \cite{d0}). Let $(X,L)$
be a polarized toric variety, $P$ the corresponding polytope and
$u$ a piece-wise affine convex function on $P.$ Then $u$ determines
a test configuration such that the corresponding Donaldson-Futaki
invariant is given by $\mathcal{L}_{\sigma_{P'}}(u)$ (up to a numerical
factor). 
\end{prop}
Combining the previous proposition with formula \ref{eq:formula for F' in terms of F}
we then arrive at the following 
\begin{prop}
\label{prop:log don-fut inv for toric log fano}(same notation as
in the previous proposition). Write $L=-(K_{X}+\Delta)$ for a toric
divisor $\Delta.$ Then the log Donaldson-Futaki invariant of $(X,L,\Delta)$
is given by $\mathcal{L}_{\sigma_{P}}(u).$\end{prop}
\begin{proof}
By linearity we may as well assume that $\Delta=D_{F}$ for a fixed
facet $F$ of $P.$ As explained in the proof of Prop \ref{pro:(Donaldson).-toric test}
given in \cite{d0} $a_{0}=-\int_{P}udp$ and since we may apply the
same result to the polarized toric variety $(D_{F},L_{|D_{F}})$ we
also have $\tilde{a}_{0}=-\int_{F}ud\sigma_{P}'.$ Moreover, since
$b_{0}=\int_{P}dp$ and hence similarly $\tilde{b}_{0}=\int_{F}\sigma_{P}'$
combining the previous proposition with formula \ref{eq:formula for F' in terms of F}
concludes the proof.
\end{proof}

\subsubsection{Enf of proof of Theorem \ref{thm:-existence of ke intro}}

By Theorem \ref{thm:-existence of ke intro} we just have the verify
the equivalence between the last three points above. But this follows
immediately from Lemma \ref{lem:lower bound on l when bary} combined
with Prop \ref{pro:toric futaki inv} and Prop \ref{prop:log don-fut inv for toric log fano},
at least for Futaki invariants defined with respect to toric vector
fields $V.$ Finally, if there is a Kähler-Einstein metric on for
$(X,\Delta)$ and $\Delta$ is \emph{effective,} then, by Prop \ref{prop:k-e max ding funtional for general log fano}
the corresponding Ding type functional is bounded from above and hence
so is the corresponding Mabuchi type functional $\mathcal{M}_{(X,\Delta)}$
(by the analogue of Prop \ref{pro:comparison mab and ding for P };
see \cite{berm2,bbegz}). But $\mathcal{M}_{(X,\Delta)}(\phi_{t})$
is linear in $t$ if $\phi_{t}$ comes from the flow of $V$ and hence
it must be that it is actually constant, i.e. its derivative $f_{(X,\Delta)}(V)$
vanishes.

\section{Convergence of the \label{sec:Convergence-of-the}Kähler-Ricci flow}

Recall that the Kähler-Ricci flow on a Fano manifold $X$ is defined
by 
\begin{equation}
\frac{d\omega_{t}}{dt}=-\mbox{Ric }\omega_{t}+\omega_{t}\label{eq:k-r flow for k=0000E4hler metrics}
\end{equation}
for a given initial Kähler form $\omega_{0}.$ It may be equivalently
formulated as the following flow of positively curved metrics $\phi_{t}$
on $-K_{X}:$ 
\[
\frac{d\phi_{t}}{dt}=\log\frac{MA(\phi_{t})}{\tilde{\mu}_{\phi_{t}}},
\]
 where $\tilde{\mu}_{\phi_{t}}$ is the measure defined by the metric
$\phi_{t}$ (formula \ref{eq:mu of phi}), normalized by its mass.
As shown by Song-Tian \cite{so-t} the latter flow can also be given
a meaning on any Fano variety $X$ with log terminal singularities.
In particular, the corresponding flow of currents $\omega_{t}:=dd^{c}\phi_{t}$
restricts to the usual Kähler-Ricci flow \ref{eq:k-r flow for k=0000E4hler metrics}
on the regular locus $X_{0}.$ In fact, all the constructions and
results in this section carries over immediately to the general setting
of a log Fano variety $(X,\Delta)$ with klt singularities, with $X_{0}$
denoting the complement of $\Delta$ in the regular locus of $X,$
but to simplify the notation we will assume that $\Delta=0.$
\begin{thm}
\label{thm:conv of k-r flow}Let $X$ be a toric Fano variety and
let $\omega_{t}$ evolve according to the corresponding Kähler-Ricci
flow. Then there exists a family $A_{t}$ of toric automorphisms of
$X$ such that $A_{t}^{*}\omega_{t}$ converges weakly towards a (singular)
Kähler-Ricci soliton $\omega$ on $X.$ 
\end{thm}
Given the coercivity estimate for the modified Mabuchi functional
in Theorem \ref{thm:mab functional is coerc for toric is equi to ke}
the proof of the previous theorem is a rather straight-forward adaptation
of the proof in \cite{bbegz} of the convergence of Kähler-Ricci flow
on a Fano manifold for which the ordinary Mabuchi K-energy functional
is proper. 

Turning to the details of the proof we let $\psi_{t}$ be defined
as the pull-back of the metric $\phi_{t}$ under the time $t$ flow
$\exp(tV)$ of the holomorphic vector field $V,$ where $V$ is the
unique toric vector field with components $a_{i}$ determined by the
canonical polytope $P$ corresponding to $X.$ Then $\psi_{t}$ satisfies
the following modified Kähler-Ricci flow (compare \cite{t-z2}): 
\begin{equation}
\frac{d\psi_{t}}{dt}=\log\frac{MA_{g}(\psi_{t})}{\tilde{\mu}_{\psi_{t}}},\label{eq:mod k-r flow}
\end{equation}
 where, $g(p)=e^{\left\langle a,p\right\rangle }$ and $MA_{g}$ is
the corresponding Monge-Ampère type operator. Now, a direct computation
reveals that, along the latter flow, 
\[
\frac{d\mathcal{G}(\psi_{t})}{dt}=D(MA_{g}(\psi_{t}),\tilde{\mu}_{\psi_{t}}),
\]
 where we recall that $D$ denotes, as before, the relative entropy.
In particular, $\mathcal{G}_{g}(\psi_{t})$ is increasing in $t.$
Strictly, speaking the previous computation is only valid in the smooth
setting, but it can easily be justified by regularizing precisely
as in the proof of Lemma 6.4 in \cite{bbegz}.

Now, by Theorem \ref{thm:coerciv k-e}, $\mathcal{G}_{g}(\psi_{t})$
is bounded from above and hence there is a subsequence $t_{j}$ such
that the rhs above tends to zero. But then it follows from the Pinsker
inequality that 
\begin{equation}
\left\Vert \tilde{\mu}_{\psi_{t_{j}}}-MA_{g}(\psi_{t_{j}})\right\Vert \rightarrow0,\label{eq:abs var norm}
\end{equation}
 in the absolute variation norm of the measures (i.e. the $L^{1}-$norm
between the densities wrt any fixed background measure). Let now $\tilde{\psi}_{t}$
be the normalization of $\psi_{t},$ obtained by applying an appropriate
toric automorphism $B_{t}$ and denote by $\tilde{\psi}$ a weak limit
point of $\tilde{\psi}_{t_{j}}.$ By invariance the convergence \ref{eq:abs var norm}
still holds when $\psi_{t}$ is replaced with its normalization $\tilde{\psi}_{t}$
and $\mathcal{G}_{g}(\tilde{\psi}_{t})$ is still increasing in $t$
(by the invariance of $\mathcal{G}_{g}$ under toric automorphism,
which holds as in the proof of Theorem \ref{thm:existe real ma intr}).
It thus follows from the relative coercivity inequality in Thm \ref{thm:coerciv k-e}
that
\begin{equation}
\mathcal{E}(\tilde{\psi}_{t})\geq-C\label{eq: pf k-r lower bound on energy}
\end{equation}
 and hence $\mathcal{E}(\tilde{\psi})\geq-C.$ In particular, $\tilde{\psi}_{t}$
and $\tilde{\psi}$ have full Monge-Ampère mass and hence it follows
from Prop \ref{pro:cont for full ma} and \ref{eq:abs var norm}  that
\begin{equation}
MA_{g}(\tilde{\psi})=\tilde{\mu}_{\tilde{\psi}},\label{eq:soliton eq in k-r proof}
\end{equation}
All we have to do now is to verify the following 
\[
\mbox{Claim:\,\,}\,\,\,\mathcal{E}_{g}(\tilde{\psi}_{t_{j}})\rightarrow\mathcal{E}_{g}(\tilde{\psi})
\]
Indeed, accepting the latter claim for the moment we note that, since,
$\tilde{\psi}$ satisfies the equation \ref{eq:soliton eq in k-r proof}
and hence (by Prop \ref{prop:var prop wrt G and M}) maximizes the
functional $\mathcal{G}_{g}$ it follows, using that $\mathcal{G}(\tilde{\psi}_{t})$
is increasing in $t,$ that any subsequence of $\tilde{\psi}$ is
an asymptotic maximizer of the functional $\mathcal{G}_{g}.$ Hence,
by the proof of Theorem \ref{thm:existe real ma intr}, it converges
to the unique normalized finite energy minimizer of $\mathcal{G}_{g}$
(which thus coincides with $\tilde{\psi}).$ 

All in all, setting $A_{t}=B_{t}\circ\exp(tV)$ concludes the proof
of the theorem up to the claim above to whose proof we finally turn.
First note that since the modified Mabuchi functional $\mathcal{M}_{g}$
is bounded from below it follows from \ref{eq: pf k-r lower bound on energy}
that 
\[
D_{\mu_{0}}(MA_{g}(\tilde{\psi}_{t}))\leq C'
\]
 At this point we can invoke the following crucial compactness property
(see Theorem 3.10 in \cite{bbegz}):
\begin{lem}
Let $\mu_{0}$ be a probability measure with locally Hölder potentials
and let $\phi_{j}\rightarrow\phi$ be a weakly convergent sequence
such that $\mathcal{E}(\phi_{j})\geq-C.$ For each probability measure
$\nu$ with finite relative entropy, i.e. $D_{\mu_{0}}(\nu)<\infty,$
we then have 
\[
\int_{X}(\phi_{j}-\phi)\nu\rightarrow0,
\]
 uniformly wrt $D_{\mu_{0}}(\nu).$
\end{lem}
Applying the previous lemma to $\phi_{j}:=\tilde{\psi}_{t_{j}}$ and
$\nu=MA_{g}(\tilde{\psi}_{t_{j}})$ gives, after perhaps passing to
a subsequence, that 
\[
\int_{X}(\tilde{\psi}_{t_{j}}-\tilde{\psi})MA_{g}(\tilde{\psi}_{t_{j}})\rightarrow0
\]
 But then it follows, since $\tilde{\psi}_{t_{j}}$ is sup-normalized,
that the convergence in the claim indeed holds (compare the proof
of Lemma 2.4 in \cite{bbegz}).

\subsection{\label{sub:Regularity-of-singular}Regularity of singular Kähler-Ricci
solitons}

Here we will use the Kähler-Ricci flow to show that any toric (singular)
Kähler-Ricci soliton $(\omega,V)$ on a toric Fano variety $X$ is
such that $\omega(=dd^{c}\psi)$ is smooth on the regular locus $X_{0}.$
As explained in section \ref{sub:Toric-K=0000E4hler-Einstein-metrics}
we already know that $\psi$ is continuous, viewed as a metric on
$-K_{X}.$ We take $\psi:=\psi_{0}$ as the initial data for the modified
Kähler-Ricci flow $\psi_{t}$ (formula \ref{eq:mod k-r flow}). By
the work of Song-Tian \cite{so-t} the usual Kähler-Ricci flow $\phi_{t}$
is smooth om $X_{0}$ for $t>0$ and hence so is $\psi_{t},$ since
the two flows coincide up to conjugation by the flow of $V.$ Now,
by Prop \ref{prop:var prop wrt G and M} $\psi_{0}$ is a maximizer
for $\mathcal{G}_{g}$ and, since, as explained above, the corresponding
functional $\mathcal{G}_{g}(\psi_{t})$ is increasing $\psi_{t}$
is also a maximizer for $\mathcal{G}_{g}$ for any $t>0$ (more precisely,
as explained above $\mathcal{G}_{g}(\psi_{t})$ is increasing for
$t>0,$ which is enough since it is also continuous up to $t=0).$
But then it follows from Prop \ref{prop:var prop wrt G and M} that
for any $t>0,$ $\psi_{t}$ satisfies the corresponding Kähler-Ricci
soliton equation and is smooth on $X_{0}.$ By the uniqueness of solutions
modulo automorphisms we deduce that $\psi_{0}$ is also smooth on
$X_{0}.$ Actually, we do not need to use the uniqueness: since the
time-derivative of the flow vanishes for $t>0$ it follows, by continuity,
that $\psi_{0}=\psi_{t}$ for any $t>0$ and hence $\psi_{0}$ is
also smooth on $X_{0},$ as desired.

\section{Appendix: proof of Lemma \ref{pro:var prop of leg}}

A proof of the first point in Lemma \ref{pro:var prop of leg} can
be found in \cite{r}; but it is also a special case of the following
slightly more general claim that we will use to prove the second point:
\emph{let $G_{0}$ be a proper upper semi-continuous function on $\R^{n}$
with a unique maximizer $x_{0}$ and let $G_{t}(x):=G_{0}(x)+tv(x)$
for a bounded continuous function $v.$ Then $g(t):=\sup_{x\in\R^{n}}G_{t}(x)$
is differentiable at $t=0$ and 
\[
\frac{dg(t)}{dt}_{t=0}=v(x_{0})
\]
}This is without doubt a well-known fact but for completeness we include
the proof. First note that $G_{0}$ is bounded from above (since it
is usc and hence, by properness, $G_{0}(x)\rightarrow-\infty$ as
$|x|\rightarrow-\infty.$ Since $v$ is bounded it then follows that,
for $t$ sufficiently small, the sup of $G_{t}$ is attained at some
(but not necessarily unique) point $x_{t}.$ Hence, $g(t)-g(0)=$
\[
=G_{t}(x_{t})-G_{0}(x_{0})=\left(G_{t}(x_{0})-G_{0}(x_{0})\right)+\left(G_{0}(x_{t})-G_{0}(x_{0})\right)+t(v(x_{t})-v(x_{0}))
\]
Next we will show that 
\begin{equation}
v(x_{t})-v(x_{0})=o(t).\label{eq:proof of claim diff}
\end{equation}
 By the continuity of $v$ it will be enough to establish that $x_{t}=x_{0}+o(t).$
To this end we first note that since $v$ is bounded and $G_{0}$
is proper it follows that the $x_{t}$ stay in a compact subset $K$
and $\limsup_{t\rightarrow0}G(x_{t})\geq G(x_{0})(=\sup_{x\in\R^{n}}G_{t}(x)).$
Hence, if $x_{*}$ is a limit point of $x_{t}$ then the upper-semicontinuity
of $G_{0}$ implies that $x_{*}$ is a maximizer for $G_{0}.$ By
the uniqueness assumption this means that $x_{*}=x_{0}$ and hence
$x_{t}=x_{0}+o(t)$ as desired, thus proving \ref{eq:proof of claim diff}. 

If $G_{0}$ were differentiable at $x_{0}$ we could use the maximization
property of $x_{0}$ to deduce that $\left(G_{0}(x_{t})-G_{0}(x_{0})\right)=o(t)$
and hence that $\frac{dg(t)}{dt}_{t=0}=v(x_{0})+0+0.$ But in general
we only know, a priori, that that $\left(G_{0}(x_{t})-G_{0}(x_{0})\right)\leq0$
with equality at $t=0,$ so that $\frac{dg(t)}{dt}_{t=0^{+}}\leq v(x_{0}).$
Moreover, by symmetry (i.e. replacing $t$ by $-t)$ we also have
$\frac{dg(t)}{dt}_{t=0^{-}}\geq v(x_{0}).$ On the other hand $g_{t}$
is convex in $t$ (as it is defined as a sup of affine functions)
and hence its right and left derivatives exist and satisfy the inequality
$\frac{dg(t)}{dt}_{t=0^{-}}\leq\frac{dg(t)}{dt}_{t=0^{+}}.$ Thus
it must be that the right and left derivatives both coincide with
$v(x_{0})$ which concludes the proof of the claim above.

To prove the first point set $G_{t}(x)=\left\langle p+ta,x\right\rangle -\phi(x)$
for given vectors $p$ and $a$ and for the second point one set $G(x,t):=\left\langle p,x\right\rangle -(\phi(x)+tv(x)).$
As for the last point we first assume that $\phi$ is smooth and strictly
convex and that $f$ is bounded. Then, making the change of variables
$p=d\phi_{|x}$ gives 
\[
\int vMA(\phi)=\int v(x)d(p(x))=\int v(x_{p})dp,
\]
 where $x_{p}$ is uniquely determined by $p=d\phi_{|x_{p}},$ which
by duality and the first point above means that $x_{p}=d\phi_{|p}^{*},$
proving the desired formula in the case. Finally, we take smooth and
strictly $\phi_{j}$ decreasing to a given $\phi$ and hence $u_{j}:=\phi_{j}^{*}$
increase to $u:=\phi^{*}.$ By convexity $du_{j|p}\rightarrow du_{|p}$
for any $p\in S$ where $S$ is the set of points $p$ where $u$
is finite and differentiable. By assumption $S=P-N$ where $N$ has
measure zero. Finally, letting $j\rightarrow\infty$and using Prop
\ref{pro:cont for full ma} in the lhs and dominated convergence in
the rhs concludes the proof for $v$ bounded. But writing $v$ as
an increasing limit $v_{j}$ of non-negative bounded continuous functions
and then using the Lebesgue monotone convergence theorem then proves
the general case.


\begin{thebibliography}{10}
\bibitem{ab}M. Abreu, K\textasciidieresis{}ahler geometry of toric
manifolds in symplectic coordinates, Symplectic and contact topology:
interactions and perspectives (Toronto, ON/Montreal, QC, 2001)

\bibitem{a-g-s}Ambrosio, L; Gigli, N; Savaré, G: Gradient flows in
metric spaces and in the space of probability measures. Second edition.
Lectures in Mathematics ETH Zürich. Birkhäuser Verlag, Basel, 2008.
x+334 pp.

\bibitem{a-k-m}S. Artstein-Avidana; B. Klartag; V. Milman: The Santaló
point of a function, and a functional form of the Santaló inequality.
Mathematika (2004), 51 : pp 33-48 

\bibitem{bak}Bakeman, I.J: Convex analysis and nonlinear geometric
elliptic equations. Springer-Verlag, Berlin, 1994. 

\bibitem{bar}F. Barthe, On a reverse form of the Brascamp-Lieb inequality.
Invent. Math. 134 (1998), no. 2, 335\textendash{}361.

\bibitem{ba}Batyrev, V.V.: Dual polyhedra and mirror symmetry for
Calabi-Yau hypersurfaces in toric varieties. J. Algebr. Geom. 3, 493-535
(1994)

\bibitem{berm2}Berman, R.J: A thermodynamical formalism for Monge-Ampere
equations, Moser-Trudinger inequalities and Kahler-Einstein metrics.
arXiv:1011.3976 

\bibitem{berm3}Berman, R.J: K-polystability of Q-Fano varieties admitting
Kähler-Einstein metrics. Preprint.

\bibitem{b-b}Berman, R.J.; Berndtsson, B: Moser-Trudinger type inequalities
for complex Monge-Ampère operators and Aubin's ``hypothèse fondamentale''.
Preprint in 2011 at arXiv:1109.1263

\bibitem{b-b2}Berman, R.J.; Berndtsson, B: The volume of Kähler-Einstein
varieties and convex bodies. arXiv:1112.4445 

\bibitem{bbgz}Berman, R.J.: Boucksom, S; Guedj, V; Zeriahi, A: A
variational approach to complex Monge-Ampère equations. arXiv:0907.4490

\bibitem{bbegz}Berman\emph{; }R.J: Eyssidieu, P: Boucksom, S; Guedj,
V; Zeriahi, A: Convergence of the Kähler-Ricci flow and the Ricci
iteration on Log-Fano varities. arXiv:1111.7158

\bibitem{bern1}Berndtsson, B: Curvature of vector bundles associated
to holomorphic fibrations. Annals of Math. Vol. 169 (2009), 531-560 

\bibitem{bern2}Berndtsson, B: A Brunn-Minkowski type inequality for
Fano manifolds and the Bando- Mabuchi uniqueness theorem , arXiv:1103.0923.

\bibitem{begz}Boucksom, S; Eyssidieux, P; Guedj, V; Zeriahi, A: Monge-Ampère
equations in big cohomology classes. Acta Math. 205 (2010), no. 2,
199\textendash{}262. 

\bibitem{b-g-l}Burns, D; Guillemin, Vi; Lerman, E: Kähler metrics
on singular toric varieties. Pacific J. Math. 238 (2008), no. 1, 27\textendash{}40.

\bibitem{ca0}Caffarelli, L.A: Interior \$W\textasciicircum{}\{2,p\}\$
estimates for solutions of the Monge-Ampère equation. Ann. of Math.
(2) 131 (1990), no. 1, 135\textendash{}150.

\bibitem{ca}Caffarelli, L.A: Some regularity properties of solutions
of Monge Ampère equation. Comm. Pure Appl. Math. 44 (1991), no. 8-9,
965\textendash{}969. 

\bibitem{ca1}Caffarelli, L. A. A localization property of viscosity
solutions to the Monge-Ampère equation and their strict convexity.
Ann. of Math. (2) 131 (1990), no. 1, 129\textendash{}134. 

\bibitem{ca-2}Caffarelli, L.A: The regularity of mappings with a
convex potential. J. Amer. Math. Soc. 5 (1992), no. 1, 99\textendash{}104,

\bibitem{cao1}Cao, H-D: Existence of gradient Kähler-Ricci solitons.
Elliptic and parabolic methods in geometry (Minneapolis, MN, 1994),
1\textendash{}16, A K Peters, Wellesley, MA, 1996. 

\bibitem{cgh}Campana, F; Guenancia, H; P\u{a}un, M: Metrics with
cone singularities along normal crossing divisors and holomorphic
tensor fields. arXiv:1104.4879 (as communicated to us by Henri Guenancia
the precise results appears in a new version in preparation).

\bibitem{c-t-z}Cao, H-D; Tian, Gang, T: Zhu, X: Kähler-Ricci solitons
on compact complex manifolds with C1(M)>0. Geom. Funct. Anal. 15 (2005),
no. 3, 697\textendash{}719. 

\bibitem{c-l-s}Cox, D. A.; Little, J. B.; Schenck, H. K. Toric varieties.
Graduate Studies in Mathematics, 124. American Mathematical Society,
Providence, RI, 2011

\bibitem{de2:}Debarre, O: Fano varieties. Higher dimensional varieties
and rational points (Budapest, 2001), 93\textendash{}132, Bolyai Soc.
Math. Stud., 12, Springer, Berlin, 2003. 

\bibitem{d-g-pk-z}Demailly, J-P; Dinew, S; Guedj,V; Pham, H.H; Kolodziej,
S; Zeriahi, A: Hölder continuous solutions to Monge-Ampère equations.
arXiv:1112.1388

\bibitem{d-t}Ding, W.Y; Tian, G: Kähler-Einstein metrics and the
generalized Futaki invariant. Invent. Math. 110 (1992), no. 2, 315\textendash{}335.

\bibitem{d0}Donaldson, S.K. Scalar curvature and stability of toric
varities. J. Diff. Geom. 62 (2002), 289-349

\bibitem{do}Donaldson, S. K. Kähler geometry on toric manifolds,
and some other manifolds with large symmetry. Handbook of geometric
analysis. No. 1, 29\textendash{}75, Adv. Lect. Math. (ALM), 7, Int.
Press, Somerville, MA, 2008. 

\bibitem{do2}Donaldson, S. K. Constant scalar curvature metrics on
toric surfaces. Geom. Funct. Anal. 19 (2009), no. 1, 83\textendash{}136. 

\bibitem{do3}Donaldson, S. K.: Kahler metrics with cone singularities
along a divisor. arXiv:1102.1196, 2011 - arxiv.org 

\bibitem{e-e}Edmunds, D. E. Evans, ;W. D.; Spectral Theory and Differential
Operators, Oxford University Press, New York, 1987.

\bibitem{f-i-k}Feldman, M; Ilmanen, T; Knopf, D: Rotationally symmetric
shrinking and expanding gradient Kähler-Ricci solitons. J. Differential
Geom. 65 (2003), no. 2, 169\textendash{}209.

\bibitem{g-z}Guedj,V; Zeriahi, A: Intrinsic capacities on compact
Kähler manifolds. J. Geom. Anal. 15 (2005), no. 4, 607--639.

\bibitem{gu}Gutierrez, C.E: The Monge-Ampère equation. Progress in
Nonlinear Differential Equations and their Applications, 44. Birkhäuser
Boston, Inc., Boston, MA, 2001. xii+127 pp. ISBN: 0-8176-4177-7 

\bibitem{j-m-r}Jeffres, T.D.; Mazzeo, R; Rubinstein, Y.A: K\textasciidieresis{}ahler-Einstein
metrics with edge singularities. Preprint (2011) arXiv:1105.5216. 

\bibitem{leg}Legendre, E: Toric Kähler-Einstein metrics and convex
compact polytopes. arXiv:1112.3239 

\bibitem{k-s}Kreuzer, M; Skarke, H: PALP: A package for analyzing
lattice polytopes with applications to toric geometry. Computer Phys.
Comm., 157, 87-106 (2004)

\bibitem{li1}Li, C: Greatest lower bounds on Ricci curvature for
toric Fano manifolds. Adv. Math. 226 (2011), no. 6, 4921\textendash{}4932

\bibitem{li2}Li, C: Remarks on logarithmic K-stability. arXiv:1104.042 

\bibitem{li-s}Li,C; Sun, S: Conical Kahler-Einstein metric revisited.
arXiv:1207.5011 

\bibitem{l-x}Li, C; Xu, C: Special test configurations and K-stability
of Fano varieties. arXiv:1111.5398 

\bibitem{mab}Mabuchi, T: Einstein-Kähler forms, Futaki invariants
and convex geometry on toric Fano varieties. Osaka J. Math. 24 (1987),
no. 4, 705\textendash{}737

\bibitem{od}Odaka, Y. The GIT stability of Polarized Varieties via
Discrepancy. arXiv:0807.1716.

\bibitem{o-s}Odaka, Y; Sun, S: Testing log K-stability by blowing
up formalism. arXiv:1112.1353 

\bibitem{p-s-s-w}Phong, D. H. Song, J; Sturm, J; Weinkove, B: The
Moser-Trudinger inequality on K\textasciidieresis{}ahler-Einstein
manifolds. Amer. J. Math. 130 (2008), no. 4, 1067\textendash{}1085.

\bibitem{r-t}Rauch, J; Taylor, B.A.: The dirichlet problem for the
multidimensional monge-ampere equation, Rocky Mountain J. Math. Volume
7, Number 2 (1977), 345-364. 

\bibitem{r}Rockafellar, R. T: Convex analysis. Reprint of the 1970
original. Princeton Landmarks in Mathematics. Princeton Paperbacks.
Princeton University Press, Princeton, NJ, 1997. 

\bibitem{s-x}Shi, Y Zhu, X.H:  K\textasciidieresis{}ahler\textendash{}Ricci
solitons on toric Fano orbifolds. Math. Z. (to appear). preprint arXiv:math/1102.2764

\bibitem{so-t}Song, J; Tian, G: The K\textasciidieresis{}ahler-Ricci
flow through singularities. Preprint (2009) arXiv:0909.4898.

\bibitem{so-j}Song, J; Wang, X: The greatest Ricci lower bound, conical
Einstein metrics and the Chern number inequality. arXiv:1207.4839

\bibitem{sz}Székelyhidi, G: Greatest lower bounds on the Ricci curvature
of Fano manifolds. Compos. Math. 147 (2011), no. 1, 319\textendash{}331

\bibitem{tr}Troyanov, M: Metrics of constant curvature on a sphere
with two conical singularities. Differential geometry (Peñíscola,
1988), 296\textendash{}306, Lecture Notes in Math., 1410, 

\bibitem{ti0}Tian, G: Canonical metrics in Kähler geometry. Lectures
in Mathematics ETH Zürich. Birkhäuser Verlag, Basel, 2000. vi+101
pp.

\bibitem{ti1}Tian, G: Kähler-Einstein metrics with positive scalar
curvature. Invent. Math. 130 (1997), no. 1, 1\textendash{}37.

\bibitem{t-z}Tian, G: ; Zhu, X: Uniqueness of Kähler-Ricci solitons.
Acta Math. 184 (2000), no. 2, 271\textendash{}305.

\bibitem{t-z2}Tian, G; Zhu, X: Convergence of Kähler-Ricci flow.
J. Amer. Math. Soc. 20 (2007), no. 3, 675\textendash{}699.

\bibitem{w-z}Wang, X; Zhu, X: K\textasciidieresis{}ahler\textendash{}Ricci
solitons on toric manifolds with positive first Chern class, Advances
in Mathematics 188 (2004), 87\textendash{}103.

\bibitem{z-z0}Zhou, B; Zhu, X: K-stability on toric manifolds. Proc.
Amer. Math. Soc. 136 (2008), no. 9, 3301\textendash{}3307,

\bibitem{z-z}Zhou, B; Zhu, X: Relative K-stability and modified K-energy
on toric manifolds. Adv. Math. 219 (2008), no. 4

\bibitem{z-z2}Zhou, Bin; Zhu, X: Minimizing weak solutions for Calabi's
extremal metrics on toric manifolds. Calc. Var. Partial Differential
Equations 32 (2008), no. 2, 191\textendash{}217.\end{thebibliography}
\end{document}